\documentclass[11pt]{amsart}
\usepackage{verbatim}
\usepackage{amsmath,amsfonts,amssymb,amsthm}
\usepackage{graphicx}
\usepackage[utf8]{inputenc}
\usepackage[T1]{fontenc}
\usepackage{makecell}
\usepackage{adjustbox}
\usepackage{tikz, tikz-cd}
\usetikzlibrary{decorations.markings, knots, backgrounds, fit, positioning, calc, arrows.meta, shapes.geometric, shapes.misc, through}
\usepackage{fullpage}
\usepackage[all]{xy}
\usepackage[margin=1in]{geometry}
\usepackage{array}
\usepackage{color}
\usepackage[english]{babel}
\usepackage[dvipsnames]{xcolor}
\usepackage[colorlinks=true, linkcolor=NavyBlue, citecolor=NavyBlue]{hyperref}
\usepackage{pdfpages}
\usepackage{textcomp}
\usepackage{caption,subcaption}
\usepackage{slashbox}
\usepackage{tablists}
\usepackage[authormarkup=right, commandnameprefix=ifneeded]{changes}
\usepackage[nameinlink,capitalize]{cleveref}
\usepackage{enumerate}
\usepackage{booktabs,array}
\usepackage{keytheorems}

\Crefname{subsection}{Subsection}{Subsections}
\Crefname{subsubsection}{Subsubsection}{Subsubsections}

\newkeytheorem{theorem}[parent=section]
\newkeytheorem{corollary, lemma, proposition, conjecture}[sibling=theorem]
\newkeytheorem{definition, remark, example, claim, question, construction, observation}[
sibling=theorem,
style=definition,
]

\newcommand{\Z}{{\mathbb{Z}}}
\newcommand{\R}{\mathcal{R}}
\newcommand{\M}{\mathcal{M}}
\newcommand{\A}{\mathcal{A}}
\newcommand{\F}{\mathcal{F}}
\newcommand{\Pt}{\mathcal{P}}
\newcommand{\PCST}[2]{{{{CST}^+}_{#1,#2}(G)}}
\newcommand{\NCST}[2]{{{{CST}^-}_{#1,#2}(G)}}

\newcommand{\cyclespace}{\operatorname{{\mathcal{Z}_G}}}

\newcommand{\EN}{\operatorname{EN}}
\newcommand{\IA}{\operatorname{IA}}
\newcommand{\IN}{\operatorname{IN}}
\newcommand{\cut}{\operatorname{cut}}
\newcommand{\cyc}{\operatorname{cyc}}
\newcommand{\nbc}{\operatorname{NBC}}

\newcommand{\C}{{\mathbf{C}}} % for chain complex shorthand notation

\newcommand{\Om}{{\Omega}} % for graded basis set

\newcommand{\BO}{{\mathbf{\Omega}}} % for total basis set

\newcommand{\G}{{\mathcal{G}}} % for hasse diagram

\newcommand{\NBC}{{\operatorname{\textbf{NBC}}}} % for NBC simplicial complex

\newcommand{\chCom}[3]{{C_{#1}^{\A_{#2}}(#3)}} % for chromatic complex over A_m

\newcommand{\chBiCom}[4]{{C_{#1, #2}^{\A_{#3}}(#4)}} % for bigraded chromatic complex over A_m

\newcommand{\chBiDiff}[3]{{\partial_{#1, #2}^{\A_{#3}}}} % for chromatic complex differential over A_m

\newcommand{\nbcCom}[2]{{C_{\operatorname{NBC}}^{\A_{#1}}(#2)}} % for NBC complex over A_m

\newcommand{\nbcDiff}[1]{{\partial_{\operatorname{NBC}}^{\A_{#1}}}} % for NBC complex differential over A_m

\newcommand{\morCom}[2]{{C_{#1}^{#2}}} % for Morse complex

\newcommand{\morDiff}[2]{{\partial_{#1}^{#2}}} % for Morse complex differential

\newcommand{\cstCom}[3]{{\operatorname{CST}_{#1}^{\A_{#2}}(#3)}} % for spanning tree complex over A_m

\newcommand{\cstredCom}[3]{{\operatorname{\widetilde{CST}}_{#1}^{\A_{#2}}(#3)}} % for spanning tree complex over A_m

\newcommand{\cstBiCom}[4]{{\operatorname{CST}_{#1, #2}^{\A_{#3}}(#4)}} % for bigraded spanning tree complex over A_m

\newcommand{\cstDiff}[2]{{\partial_{\operatorname{ST}#1}^{\A_{#2}}}} % for spanning tree complex differential over A_m

\newcommand{\cstBiDiff}[3]{{\partial_{\operatorname{ST}#1, #2}^{\A_{#3}}}} % for bigraded spanning tree complex differential over A_m

\newcommand{\cstHom}[3]{{\operatorname{HST}_{#1}^{\A_{#2}}(#3)}} % for cohomology of spanning tree complex over A_m

\newcommand{\cstBiHom}[4]{{\operatorname{HST}_{#1, #2}^{\A_{#3}}(#4)}} % for bigraded cohomology of spanning tree complex over A_m

\newcommand{\TmI}[3]{\begin{tikzpicture}[scale=0.5,baseline]
    \tikzset{node style/.style={fill=black, draw=black, shape=circle, inner sep=2pt}}
    \node [node style, color=red, label=above:$#1$] (0) at (0, 2) {};
    \node [node style, label=left:$#2$] (1) at (-1.5, 0) {};
    \node [node style, label=right:$#3$] (2) at (1.5, 0) {};
    \draw[dotted] (1) to node[left] {$L$} (0);
    \draw[dotted] (0) to node[right] {$L$} (2);
    \draw[dashed, color=blue] (1) to node[below, color=black] {$d$} (2);
\end{tikzpicture}}

\newcommand{\TmII}[3]{
\begin{tikzpicture}[scale=0.5,baseline]
    \tikzset{node style/.style={fill=black, draw=black, shape=circle, inner sep=2pt}}
    \node [node style, color=red, label=above:$#1$] (0) at (0, 2) {};
    \node [node style, label=left:$#2$] (1) at (-1.5, 0) {};
    \node [node style, label=right:$#3$] (2) at (1.5, 0) {};
    \draw[dotted] (1) to node[left] {$L$} (0);
    \draw[dotted] (0) to node[right] {$L$} (2);
    \draw[dashed, color=blue] (0) to[out=60,in=10, looseness=2.5] node[below, color=black] {$d$} (2);
\end{tikzpicture}}

\newcommand{\TmIII}[3]{\begin{tikzpicture}[scale=0.5,baseline]
    \tikzset{node style/.style={fill=black, draw=black, shape=circle, inner sep=2pt}}
    \node [node style, label=above:$#1$] (0) at (0, 2) {};
    \node [node style, color=red, label=left:$#2$] (1) at (-1.5, 0) {};
    \node [node style, label=right:$#3$] (2) at (1.5, 0) {};
    \draw[dotted] (1) to node[left] {$L$} (0);
    \draw[dotted] (0) to node[right] {$L$} (2);
    \draw[dashed, color=blue] (1) to node[below, color=black] {$d$} (2);
\end{tikzpicture}}

\newcommand{\TmIV}[3]{\begin{tikzpicture}[scale=0.5,baseline]
    \tikzset{node style/.style={fill=black, draw=black, shape=circle, inner sep=2pt}}
    \node [node style, label=above:$#1$] (0) at (0, 2) {};
    \node [node style, color=red, label=left:$#2$] (1) at (-1.5, 0) {};
    \node [node style, label=right:$#3$] (2) at (1.5, 0) {};
    \draw[dotted] (1) to node[left] {$L$} (0);
    \draw[dotted] (0) to node[right] {$L$} (2);
    \draw[dashed, color=blue] (0) to[out=60,in=10, looseness=2.5] node[below, color=black] {$d$} (2);
\end{tikzpicture}}

\newcommand{\Tm}[4]{\begin{tikzpicture}[scale=0.5,baseline]
    \tikzset{node style/.style={fill=black, draw=black, shape=circle, inner sep=2pt}}
    \node [node style, color=red, label=above:$#1$] (0) at (0, 2) {};
    \node [node style] (1) at (-1.5, 0) {};
    \node [node style] (2) at (1.5, 0) {};
    \draw[dotted] (1) to node[left] {$L$} (0);
    \draw[dotted] (0) to node[right] {$d$} (2);
    \draw[color=blue] (1) to node[below, color=black] {$#2$} (2);
\end{tikzpicture}\:\: \begin{tikzpicture}[baseline]
    \draw[thick] (-0.3,-0.3) -- (0,1.3);
\end{tikzpicture} \begin{tikzpicture}[scale=0.5,baseline]
    \tikzset{node style/.style={fill=black, draw=black, shape=circle, inner sep=2pt}}
    \node [node style, label=above:$#3$] (0) at (0, 2) {};
    \node [node style, color=red] (1) at (-1.5, 0) {};
    \node [node style] (2) at (1.5, 0) {};
    \draw[dotted] (1) to node[left] {$L$} (0);
    \draw[dotted] (0) to node[right] {$d$} (2);
    \draw[color=blue] (1) to node[below, color=black] {$#4$} (2);
\end{tikzpicture}}

\tikzset{vtx/.style={circle,fill=black,inner sep=1.7pt},
root/.style={circle,fill=red,inner sep=2pt},
live/.style={densely dashed,line width=.8pt},
dead/.style={line width=1.1pt},
external/.style={dotted,line width=.8pt,black}}
\renewcommand{\arraystretch}{1.15}

\title{A spanning tree model for chromatic homology}
\author[Banerjee]{Aninda Banerjee}
\address[]{IAI, TCG CREST, Kolkata, India}
\email{anindabanerjee24@gmail.com}

\author[Chakraborty]{Apratim Chakraborty}
\address[]{IAI, TCG CREST, Kolkata, India}
\email{apratimn@gmail.com}

\author[Das]{Swarup Kumar Das}
\address[]{IAI, TCG CREST, Kolkata, India}
\email{313swarup@gmail.com}

\author[Paul]{Pravakar Paul}
\address[]{Mathematical and Physical Sciences Division, School of Arts and Sciences, Ahmedabad University, Ahmedabad 380009, Gujarat, India}
\email{pravakar.paul@ahduni.edu.in}

\begin{document}

    % %%%%%%%%%%%%%%%%%%%%%%%%%%%% ABSTRACT %%%%%%%%%%%%%%%%%%%%%%%%%%%%%%%%

    \begin{abstract}
        After the discovery of Khovanov homology, which categorifies the Jones polynomial, an analogous categorification of the chromatic polynomial, known as chromatic homology, was introduced. Its graded Euler characteristic recovers the chromatic polynomial. In this paper, we present a spanning tree model for the chromatic complex, i.e., we describe a chain complex generated by certain spanning trees of the graph that is chain homotopy equivalent to the chromatic complex.
        We employ the spanning tree model over $\A_m:=\Z[x]/\langle x^m\rangle$ algebra to answer two open questions. First, we establish the conjecture posed by Sazdanovic and Scofield regarding the homological span of chromatic homology over the $\A_m$ algebra, demonstrating that for any graph $G$ with $v$ vertices and $b$ blocks, the homological span is $v - b$. Additionally, we prove a conjecture of Helme-Guizon, Przytycki, and Rong concerning the existence of torsion of order dividing $m$ in chromatic homology over the $\A_m$ algebra.
    \end{abstract}

    \maketitle

    % %%%%%%%%%%%%%%%%%%%%%   INTRODUCTION   %%%%%%%%%%%%%%%%%%%%%%%%%%%%%

    \begin{section}{Introduction}
        It is a fundamental problem in graph theory to determine the number of ways a graph can be colored using  $\lambda$ colors. The chromatic polynomial $P(G,\lambda)$ counts the number of such colorings when evaluated at $\lambda$. It was first defined in 1912 by George David Birkhoff in an attempt to solve the four-color theorem for planar graphs, indeed proving that  $P(G,4)>0$ for all planar graphs would have implied the result.\\

        Let $G$  be a connected simple graph with $N$ vertices, and fix an ordering on its edges. A spanning tree $T$ of $G$ is called an \emph{NBC spanning tree} if, for every edge $e  \in E(G) \setminus E(T)$, $e$ is not the smallest ordered edge in the unique cycle of  the subgraph $ T \cup \{e\}$ of $G$.\\

        There are several expansions of the chromatic polynomial in terms of different graph-theoretic objects. One such expansion expresses the chromatic polynomial in terms of NBC spanning trees. This formulation enables a purely algebraic computation of the chromatic polynomial.\\

        Helme-Guizon and Rong \cite{Helme_Guizon_Rong_2005_categorification} introduced a Khovanov-type bi-graded homology theory which categorifies the chromatic polynomial. Let $\A_m$ denote the algebra defined by
        \[\A_m := \Z[x]/\langle x^m\rangle.\]  Throughout the paper, the chromatic complex over the algebra $\A_m$ will be denoted by $\chBiCom{*}{*}{m}{G}$ for a given graph $G$. The chromatic complex of $G$ is generated by the enhanced spanning subgraphs of $G$. It should be noted that, although the original definition of chromatic homology was given for the algebra $\A_2$, it can be generalized to any graded commutative algebra. Unless explicitly stated otherwise, it will always be assumed that we are working over $\A_2$. However, in Section 4, we will explore chromatic homology over general $\A_m$.\\

        Spanning-tree models were introduced independently by Champanerkar-Kofman \cite{ChampanerkarKofman2009} % POint 1 addressed
        and by Wehrli \cite{Wehrli2008} in the context of Khovanov homology, by constructing chain complexes generated by spanning trees of the Tait (checkerboard) graph associated to a knot diagram  \cite{ChampanerkarKofman2009, Wehrli2008}.  Subsequently, related spanning-tree descriptions appeared in $\delta$-graded knot Floer homology \cite{BaldwinLevine2012} and in Heegaard Floer theory for the branched double cover of a link \cite{Greene2013}. We introduce a spanning-tree model for chromatic homology in which the differential can be computed explicitly from combinatorial data.\\

        We define a bi-graded chain complex $\cstBiCom{*}{*}{m}{G}$ generated by signed NBC spanning trees whose differential can be easily combinatorially computed from the graph. We % Point 2 addressed
        call it the \textit{spanning tree complex}. Its homology will be denoted by $\cstBiHom{*}{*}{m}{G}$. Our main theorem states that the spanning tree complex is chain homotopy equivalent to the chromatic complex $\chBiCom{*}{*}{m}{G}$.

        \begin{theorem}[Spanning tree model for chromatic homology]\label{main}
            The spanning tree complex $\cstBiCom{*}{*}{m}{G}$ and the chromatic complex $\chBiCom{*}{*}{m}{G}$ are chain homotopy equivalent and therefore, $\cstBiHom{i}{j}{m}{G} \cong H_{i,j}^{\A_m}(G)$.
        \end{theorem}

        % \begin{restatable}[Spanning tree model for chromatic homology]{thm}{maintheorem}\label{main}

        % \textcolor{red}{Theorem formula is not rendered}
        % The spanning tree complex $\cstBiCom{*}{*}{m}{G}$ and the chromatic complex $\chBiCom{*}{*}{m}{G}$ are chain homotopy equivalent and therefore, $\cstBiHom{i}{j}{m}{G} \cong H_{i,j}^{\A_m}(G)$.
        % \end{restatable}

        The graded Euler characteristic $\chi_q(\cstBiCom{*}{*}{2}{G})$ is the alternating sum of the graded ranks of its cohomology groups; in other words, $\chi_q(\cstBiCom{*}{*}{2}{G})=\sum_{i, j}(-1)^i \cdot q^j \operatorname{dim}(\cstBiHom{i}{j}{2}{G})$. As a corollary of \Autoref{main}, we obtain the spanning tree expansion of the chromatic polynomial.

        \begin{corollary} \label{expancor}
            Let $G$ be a graph with $N$ vertices. The graded Euler characteristic of $\cstBiCom{*}{*}{2}{G}$ coincides with the chromatic polynomial of the graph $G$ when evaluated at $\lambda=1+q$. Hence, we obtain
            \[
            P(G ; \lambda)=(-1)^{N-1} \lambda \sum_{i=1}^{N-1} t_{i}(1-\lambda)^i
            \]
            where $t_{i}$ is the number of NBC spanning trees of $G$ having internal activity $i$  (with respect to any fixed ordering of $E(G)$).
        \end{corollary}

        \Autoref{main} can be regarded as a categorification of the spanning tree expansion formula of the chromatic polynomial.\\

        Our theorem builds on the work of Chandler and Sazdanovic \cite{Chandler_Sazdanovic_2022}. Using methods of algebraic Morse theory, they proved that the chromatic complex (over any commutative graded algebra) deformation retracts to a complex generated by enhanced NBC subgraphs. This can be viewed as a categorification of Whitney's Broken Circuit Theorem. \\

        In a similar vein, using methods from algebraic Morse theory, we prove that the chromatic complex is chain homotopy equivalent to a chain complex generated by the signed NBC spanning trees of the graph. This results in a significant reduction in the number of generators. To emphasize this point, we provide some computations of chromatic homology using the spanning tree model without computer assistance. Moreover, the spanning tree model perspective enables us to explore the chromatic complex using graph-theoretic properties.\\

        The spanning tree model can be extended over the $\A_m$ algebra. We prove the following conjecture on the homological span of chromatic homology over the $\A_m$ algebra posed by Sazdanovic and Scofield \cite{sazdanovic2018patternskhovanovlinkchromatic}.

        \begin{theorem}[Conjecture 43 \cite{sazdanovic2018patternskhovanovlinkchromatic}]\label{hspanconjecture}
            The homological span of chromatic homology over algebra $\A_m$ of any graph $G$ with $v$ vertices and $b$ blocks is equal to $hspan \left(H_{*,*}^{\A_m}(G)\right)=v-b$.
        \end{theorem}

        % \begin{restatable}[Conjecture 43 \cite{sazdanovic2018patternskhovanovlinkchromatic} ]{thm}{Homologicalspanconjecture}\label{hspanconjecture}
        %     \textcolor{red}{Theorem Not rendered}
        %     The homological span of chromatic homology over algebra $\A_m$ of any graph $G$ with $v$ vertices and $b$ blocks is equal to $hspan \left(H_{*,*}^{\A_m}(G)\right)=v-b$.

        % \end{restatable}
        We also prove the following conjecture posed by Helme-Guizon, Przytycki, and Rong on the existence of torsion of order dividing $m$ in chromatic homology over the $\A_m$ algebra.

        \begin{theorem}[Conjecture 25 \cite{HelmeGuizonPrzytyckiRong2006_torsion}]\label{torsionconjecture}
            The cohomology $H_{*,*}^{\A_m}(G)$ of a graph $G$ contains a torsion part if and only if $G$ has no loops and contains a cycle of order greater than or equal to $3$. In this case, $H_{*,*}^{\A_m}(G)$ has a torsion of order dividing $m$.
        \end{theorem}

        % \begin{restatable}[Conjecture 25 \cite{HelmeGuizonPrzytyckiRong2006}]{thm}{Torsionconjecture}\label{torsionconjecture}
        %     \textcolor{red}{Theorem Not rendered}
        %     The cohomology $H_{*,*}^{\A_m}(G)$ of a graph $G$ contains a torsion part if and only if $G$ has no loops and contains a cycle of order greater than or equal to $3$. In this case, $H_{*,*}^{\A_m}(G)$ has a torsion of order dividing $m$.
        % \end{restatable}

        The paper is organized as follows. In Section 2, we review some background material. In Section 3, we introduce the spanning tree model for chromatic homology in the context of \(\A_2\) algebras. We describe the exact combinatorial formula of the differential and discuss some applications and computations. Finally, in Section 4, we generalize the spanning tree model for chromatic homology over \(\A_m\) algebras and provide proofs of \Autoref{hspanconjecture} and \Autoref{torsionconjecture}.
    \end{section}

    % %%%%%%%%%%%%%%% notation AND CONVENTIONS %%%%%%%%%%%%%%%%%%%%%%%%%%%

    \section*{notation and conventions}

    Throughout the paper, we fix the following notation and conventions.\\

    \begin{itemize}
        \setlength{\itemsep}{8pt}
        \item Throughout this paper, $G$ will denote an undirected simple graph and $\G$ will denote a directed graph.
        % Chromatic homology for graphs with loops and multiple edges can be recovered from the case of simple graphs using Corollary 3.2 of \cite{helmeguizon2005graphcohomologiesarbitraryalgebras}.
        All the directed graphs $\G$ that appear in this paper are equipped with a partial matching $\M$ of its edges, where a matching $\M$ is a collection of edges such that no two edges are incident to a common vertex. $\G^{\M}$ denotes the di-graph obtained from $\G$ by reversing the arrows in $\M$. A vertex of $\G$ is called critical if it is not incident to any edges in the matching $\M$. If $V$ denotes the vertex set of $\G$, then we shall use the notation $V^{\M}$ to denote the collection of critical vertices of $\G$ with respect to the partial matching $\M$.

        \item $\pi_0(H)$ denotes the collection of connected components of the graph $H$.

        \item $[A,B] := \{C \mid A \subseteq C \subseteq B\}$ for any two sets $A$ and $B$.

        \item $\A_m$ denotes the graded commutative algebra $\Z[x]/\langle x^m\rangle$ where $\operatorname{gr}(1) = 0$ and $\operatorname{gr}(x^i) = i$ for $i \in \{1,\cdots,m-1\}$.

        \item We will use a subscript instead of a superscript throughout the paper to denote the grading of a cochain complex and its associated cohomology groups.

        \item For a bi-graded cochain complex $(C_{\ast,\ast},\partial_{\ast,\ast})$ of free abelian groups of finite rank, where bidegree of $\partial_{\ast,\ast}$ is $(1,0)$, $\chi_q(C_{*,*})$ will denote its graded Euler characteristic, $$\chi_q(C_{*,*}) := \sum_{i,j}(-1)^i q^j \operatorname{rank}(C_{i,j}) = \sum_{i,j}(-1)^i q^j \operatorname{rank}(H_{i,j}(C_{*,*}))$$.
        % Point 3 adderssed.
        \item $\NBC(G)$ denotes the simplicial complex whose face set is the collection of all non-broken circuits of a graph $G$ with respect to an ordering on $E(G)$. We will simply use this notation without mentioning the ordering, as it will be clear from the context.

        \item $\F(\NBC(G))$ denotes the collection of facets of $\NBC(G)$ which are precisely the NBC spanning trees.

        \item $(\chBiCom{*}{*}{m}{G}, \chBiDiff{*}{*}{m})$ denotes the chromatic cochain complex of a graph $G$ over the graded commutative algebra $\A_m$. We will also sometimes omit the differential and just use $\chBiCom{*}{*}{m}{G}$ to denote this complex.

        \item $(\nbcCom{m}{G},\nbcDiff{m})$ denotes the cochain complex generated by the non-broken circuits of the graph $G$ for a fixed ordering on the edges of $G$. The ordering will not be mentioned if it is clear from the context. We will also sometimes use $\nbcCom{m}{G}$ to denote this same cochain complex.

        \item For a given cochain complex $\C=(C_i,\partial_i)_{i\geq 0}$ of finitely generated free modules and a fixed graded basis $\BO =\bigcup_{i \geq 0}\Om_{i} $ of $\C$, we denote by $\G(\C,\BO)$ the weighted directed graph corresponding to $\C$ and $\BO$.

        \item Given a cochain complex $\C=(C_i,\partial_i)_{i\geq 0}$ and a graded basis $\BO =\bigcup_{i \geq 0}\Om_{i} $ of $\C$, let $c \in \Om_i$ and $c' \in \Om_{i+1}$ then $[c:c']$ denotes the coefficient of $c'$ in the basis expansion of $\partial_i(c)$ with respect to the basis $\Om_{i+1}$.

        \item $\G^\M(\C,\BO)$ denotes the weighted directed graph $\G(\C,\BO)$ along with a matching $\M$. % Point 4 addressed.

        \item For a given cochain complex of finitely generated free modules $\C=(C_\ast,\partial_\ast)$ and a fixed graded basis $\BO=\cup_{i \geq 0}\Om_{i}$, we denote by $(\morCom{*}{\M},\morDiff{*}{\M})$ the Morse cochain complex associated to an acyclic matching $\M$ on $\G^\M(\C,\BO)$.

        \item $(\cstBiCom{\ast}{\ast}{m}{G}, \cstBiDiff{\ast}{\ast}{m})$ denotes the spanning tree complex of the graph $G$ over the graded commutative algebra $\A_m$.

        \item $\cstBiHom{\ast}{\ast}{m}{G}$ denotes the cohomology of the spanning tree complex of the graph $G$ over the graded commutative algebra $\A_m$.

    \end{itemize}

    \subsection*{Acknowledgments}
    The authors would like to thank the anonymous referees for their helpful comments and valuable suggestions, which have greatly improved this article.

    % %%%%%%%%%%%%%%%%%%%% BACKGROUND %%%%%%%%%%%%%%%%%%%%%%%%%%%%%%%%%%%%%%%

    \begin{section}{Background}
        \begin{subsection}{Algebraic Morse theory}{\label{algebraic morse theory}}

            Morse theory for regular CW-complexes was originally defined by Forman \cite{FORMAN199890}. An algebraic version of this was introduced by J{\"o}llenbeck and Welker in \cite{Jllenbeck2009MinimalRV} and independently by Sk{\"o}ldberg in \cite{skol} for chain complexes of arbitrary free $R$-modules.\\

            Let $R$ be a ring with unity, and $\C=(C_i,\partial_i)_{i\geq 0}$ be a cochain complex of free $R$-modules $C_i$ for all $i$. We fix a basis $\BO=\bigcup_i\Om_{i}$ of $\C$  where $\Om_{i}$ is a basis of $C_{i}$.   For two basis elements $c,c'$ of homological grading $i$ and $i+1$ respectively, the coefficient of $c'$ in the differential of $ \partial_{i}\left(c \right)$ is denoted by $[c:c']$ that is, \[ \partial_{i}(c) = \sum_{c' \in \Om_{i+1}} [c: c'] \cdot c'. \]

            The $i$-th cochain group $C_i$ is given by:

            $$
            C_i\cong\bigoplus_{c \in \Omega_i} Rc
            $$

            The differential $\partial_i$ for the $i$-th grading is given by:
            \begin{equation}
                \partial_i:
                \begin{cases}
                    C_i \rightarrow C_{i+1} \\
                    c \mapsto \partial_i(c)=\sum_{c'\in \Om_{i+1}}[c:c'] \cdot c'
                \end{cases}
            \end{equation}

            Given a cochain complex $\C=(C_{\ast}, \partial_{\ast})$ with a graded basis $\BO=\{\Om_i \}_{i}$, one can construct a weighted directed graph
            $\G(\C,\BO)$ such that the vertex set is $\BO$, and the directed edge set contains elements of the form $c \rightarrow c'$ whenever $[c:c']\neq 0$ with $c \in \Om_i$ and $c' \in \Om_{i+1}$. The weight on each such edge $c \rightarrow c'$ is defined to be $[c:c']$.\\

            A matching $\M$ on this directed weighted graph is a collection of its directed weighted edges such that no two edges are incident to a common vertex and the weights on those edges are invertible over $R$. Given a matching $\M$ on $\G(\C,\BO)$, we can define another directed weighted graph $\G^{\M}(\C,\BO)$, with the same underlying vertex set as $\G(\C,\BO)$ and the directed weighted edge set is defined as follows:
            \[
            E(\G^{\M}(\C,\BO)) = \begin{cases}
                (c \rightarrow c', [c:c']) \quad \text{if the directed edge } (c \rightarrow c') \not\in \M \\
                \left(c' \rightarrow c, -\frac{1}{[c:c']}\right) \quad \text{otherwise}
            \end{cases}
            \]

            Given a cochain complex $\C$ and a graded basis $\BO$, we say that a matching $\M$ on $\G(\C,\BO)$ is acyclic if there are no directed cycles on the graph $\G^\M(\C,\BO)$.

            Given a matching $\M$ on the graph $\G(\C,\BO)=(V,E)$, we shall use the following notation:

            \begin{enumerate}
                \item A vertex $v \in V$ is called \emph{critical} with respect to $\M$ if $v$ is not incident to an edge $e \in \M$. We use the notation
                \[
                \Om_i^{\M}:=\{c \in \Om_i \mid c \text{ is critical}\}
                \]
                for the set of all critical vertices in the $i$-th homological grading. Let us denote the collection of all possible critical vertices by $\BO^\M := \cup_i \Om_i^{\M}$.

                \item $\Pt(c,c')$ is the set of all alternating paths from $c$ to $c'$ in the graph $\G^\M(\C,\BO)$

                \item $w(c\rightarrow c')$ denotes the weight of the directed edge $(c \rightarrow c')$ in the graph $\G^\M(\C,\BO)$.

                \item The weight of an alternating path $p \in \Pt(c,c')$ given by $w(p)$ for $p=(c=c_1\rightarrow \cdots \rightarrow c_r=c')$ is given by
                \[
                w(p):=\prod_{i=1}^{r-1} w(c_i \rightarrow c_{i+1})
                \]

                \item $\Gamma(c,c'):=\sum_{p \in \Pt(c,c')}w(p)$ denotes the sum of the weights of all alternating paths from $c$ to $c'$.
            \end{enumerate}

            \begin{definition}[Morse complex]
                Let $(\C,\BO)$ be a free cochain complex with a graded basis $\BO$ and an acyclic matching $\M$ on $\G(\C,\BO)$. The Morse complex with respect to $\M$ is defined to be the cochain complex with an $i$-th cochain group being the free $R$-module $\morCom{i}{\M} :=\bigoplus_{c \in \Om_{i}^{\M}} Rc$, and the differential map $\partial_*^\M$ is defined as follows:

                \[
                \morDiff{i}{\M}:
                \begin{cases}
                    \morCom{i}{\M} \rightarrow \morCom{i+1}{\M}\\
                    c \mapsto \sum_{c'\in \Omega_{i+1}^{\M}} \Gamma(c,c') \cdot  c',
                \end{cases}
                \]
            \end{definition}

            Associated to a given freely generated cochain complex with a graded basis $(\C,\BO)$ and an acyclic matching $\M$ on $\G^\M(\C,\BO)$, we have the following maps:

            \begin{equation} \label{f map in dmt}
                f_i: C_i^\M \rightarrow C_i, \quad   f_i(c) := \sum_{c' \in \Om_i} \Gamma(c,c') \cdot c'
            \end{equation}

            \begin{equation} \label{g map in dmt}
                g_i: C_i \rightarrow C_i^\M, \quad   g_i(c) := \sum_{c' \in \Om_i^{\M}}\Gamma(c,c') \cdot c'
            \end{equation}

            \begin{equation}
                \eta_i: C_i \rightarrow C_{i-1}, \quad \eta_i(c) := \sum_{c' \in \Om_{i-1}}
                \Gamma(c,c')\cdot c'
            \end{equation}

            \begin{lemma} \label{chain homotopic maps}
                For the maps defined above, we have the following:
                \begin{enumerate}
                    \item $f$ and $g$ are chain maps.

                    \item $f$ and $g$ are chain homotopy equivalences. In particular, we have
                    \begin{enumerate}
                        \item $f_i \circ g_i - id_{C_i} = \partial \circ \eta_{i} + \eta_{i+1} \circ \partial$,

                        \item $g_i \circ f_i - id_{C_i^\M} = 0$.
                    \end{enumerate}
                \end{enumerate}
            \end{lemma}

            As a consequence of the above lemma, we obtain the main result in algebraic discrete Morse theory, which is the following theorem:

            \begin{theorem}[\cite{Jllenbeck2009MinimalRV}]\label{dmt maintheorem}
                Let $(\C,\BO)$ be a freely generated cochain complex with a graded basis, and $\M$ be an acyclic matching on $\G(\C,\BO)$.
                Then, we have
                \[
                (C_\ast,\partial_\ast) \backsimeq \left(\morCom{\ast}{\M},\morDiff{\ast}{\M}\right)
                \]

                In particular, for all $i \geq 0$, there is an isomorphism on the (co)-homology level,

                $$
                H_i(C_\ast) \cong H_i(C_*^\M)
                $$
            \end{theorem}
        \end{subsection}

        \begin{subsection}{Shellable complex and posets}\label{posets}

            \begin{definition}
                An abstract simplicial complex $\Delta(V)$ on a vertex set $V$, is a non-empty collection of subsets of $V$, such that:
                \begin{enumerate}
                    \item $\{v\} \in \Delta, \forall v \in V$.
                    \item If $F \subseteq G$ and $G \in \Delta$ then, $F \in \Delta$.
                \end{enumerate}
            \end{definition}

            A subcomplex $\Delta'$ of $\Delta$ is a subset of $\Delta$, and it is itself an abstract simplicial complex. The elements of $\Delta$ are called faces. A face that is not properly contained in any other face is called a facet of $\Delta$. The dimension of a face is one less than its cardinality, that is, $dim \ F:=\lvert F \rvert -1$.
            An abstract simplicial complex $\Delta$ is said to be \textit{pure} if all its facets are of the same dimension. The dimension of a simplicial complex is defined as the maximum of the dimensions of its facets. So, in the case of pure simplicial complexes, its dimension is the dimension of any facet.

            \begin{definition}
                A shelling of a simplicial complex $\Delta$ is a linear order $\{F_1,F_2, \cdots, F_p \}$ of the facets of $\Delta$, such that for each pair $F_i,F_j$ of facets such that $1 \leq i<j \leq p$, there is a facet $F_k$ satisfying $1\leq k <j$ and an element $x \in F_j$ such that $F_i \cap F_j \subseteq F_k \cap F_j = F_j -\{x\}$. A simplicial complex is said to be shellable if it is pure and it has a shelling.
            \end{definition}

            Let us fix some notation for the following section. Let $\Delta$ be an $r$-dimensional shellable complex with a shelling on its maximal facets \{$F_1,F_2, \cdots, F_p$\}.
            Let $\Delta_i$ be the subcomplex of $\Delta$ generated by the first $i$ facets for $i=1,2, \cdots, p$, that is,
            $\Delta_i=\{G \mid G \subseteq F_k  \text{ for }  k\leq i\}$. We also define $\Delta_0 := \emptyset$.\\

            Let $\R(F_i)=\{x \in F_i: F_i-x \in \Delta_{i-1}\}$, called the restrictions of $F_i$ induced by the shelling. Consider $\Delta$ to be partially ordered by set inclusion. Given $G_1, G_2 \in \Delta$, we denote the interval between $G_1$ and $G_2$ as $[G_1,G_2] := \{G \in \Delta \vert G_1 \subseteq G \subseteq G_2 \}$. % Point 5 addressed.

            \begin{example}
                Let $V=\{1,2,3,4,5\}$. Define
                \[
                \begin{aligned}
                    \Delta(V)= &\{\emptyset, \{1\},\{2\},\{3\},\{4\}, \{5\},\{1,2\},\{1,3\},\{1,4\},\{1,5\},\\
                    &\{2,3\},\{2,4\},\{2,5\},\{1,2,3\},\{1,2,4\},\{1,2,5\}\}
                \end{aligned}
                \]
                This defines a simplicial complex. In this case there are three facets $\{F_1=\{1,2,3\}, F_2=\{1,2,4\},F_3=\{1,2,5\}\}$ for $\Delta(V)$. Let $<$ be the lexicographic ordering on the facets of $\Delta$. Then, $F_1<F_2<F_3$ defines a shelling for the complex $\Delta(V)$ making it a shellable complex.
            \end{example}

            \begin{proposition}\label{partition}\cite{bjoner}
                The intervals $[\R(F_i),F_i], i=1,2, \cdots,p$ partition the shellable complex $\Delta$.
            \end{proposition}

            \begin{proof}
                Since $F_1, F_2, \cdots, F_i$ is a shelling of $\Delta_i$, it is enough to show inductively that for any $1 \le t \le p$,
                $\Delta_t=[\R(F_t),F_t] \cup \Delta_{t-1}$, and $[\R(F_t),F_t]\cap \Delta_{t-1}=\emptyset$. For $t=1$, we have $\Delta_1=[\R(F_1),F_1]$ and $\R(F_1)=\emptyset$.
                \par Let $P \in \Delta_t \setminus \Delta_{t-1}$. Then $P \subseteq F_{t}$, since $\Delta_t=\Delta_{t-1}\cup F_t$. Also if $x \in \R(F_t)-P$ then $F_t - x \in \Delta_{t-1}$ and $x \notin P$. Now this implies that $P \subseteq F_t-x$ and this contradicts the fact that $P \notin \Delta_{t-1}$. This proves that $\R(F_t) \subseteq P$ and hence the first claim is proved. Now if $G \in [\R(F_t),F_t]\cap \Delta_{t-1}$, then $\R(F_t) \subseteq G \subseteq F_t$ and $G \in \Delta_{t-1}$ implies $\R(F_t) \in \Delta_{t-1}$. If $\R(F_t)=F_t$ then $F_t \in \Delta_{t-1}$, a contradiction. So, $\R(F_t)\subsetneq F_t$. Hence there exists some $i<t$ such that $\R(F_t) \subset F_i \cap F_t$. Since $\{F_i\}_{i=1}^t$ is a shelling
                there exists some $k<t$ and $x \in F_t$ such that  $\R(F_t) \subset F_i \cap F_t \subseteq F_k \cap F_t=F_t-\{x\}$. Observe that $F_t-\{x\} \subseteq F_k$ for $k<t$. Hence, $F_t-\{x\} \in \Delta_{t-1}$ implies $x\in \R(F_t)$, which is a contradiction.
                This completes our proof.

            \end{proof}

            \begin{definition}
                A poset $\left(P, \leq \right)$ on a set $P$ is a set with an order relation $\leq$ such that the relation is
                \begin{enumerate}
                    \item Reflexive, i.e. $a\leq a \ \forall a \in P$.
                    \item Antisymmetric: $a\leq b$ and $b\leq a \implies a=b$.
                    \item Transitive: $a\leq b$ and $b\leq c \implies a\leq c$.
                \end{enumerate}
            \end{definition}

            To every simplicial complex $\Delta$, one can associate a poset $P(\Delta)$, called the \emph{face poset} of $\Delta$, which is defined to be the poset of non-empty
            faces, ordered by inclusion. By abuse of notation, we shall use the same notation $P(\Delta)$ to denote the associated Hasse diagram.

            Given any graph $ G = (V,E)$, the power set of its edge set $2^E$ is an abstract simplicial complex on the edge set $E$. Consider its face poset $P(2^E)$ ordered by
            inclusion of subsets. We denote this poset by $\left(P(2^E),\subseteq \right)$. Given any two elements $s,s' \in P(2^E)$, we say that $s$ is covered by $s'$, if $s\subsetneq s'$ and there does not exist any $p \in P(2^E)$ such that $s\subsetneq p\subsetneq s'$. We denote the cover relation by $s\lessdot s'$.

            \begin{example}
                \begin{figure}[ht]
                    \centering
                    \begin{subfigure}{0.45\textwidth}
                        \centering
                        \begin{tikzpicture}[scale=0.7]
                            \tikzset{vertex/.style={fill=black, draw=black, shape=circle, inner sep=2pt}}

                            \node [style=vertex] (0) at (-9, 5.75) {};
                            \node [style=vertex] (1) at (-11, 2.75) {};
                            \node [style=vertex] (2) at (-7, 2.75) {};
                            \node [style=vertex] (3) at (-5, 5.75) {};

                            \draw (1) to node[midway, above] {1} (2);
                            \draw (1) to node[midway, above] {2} (0);
                            \draw (0) to node[midway, above] {3} (2);
                            \draw (2) to node[midway, above] {4} (3);

                        \end{tikzpicture}
                        \caption{$G=(V,E)$}
                    \end{subfigure}
                    \hfill
                    \begin{subfigure}{0.45\textwidth}
                        \centering
                        \begin{tikzpicture}[scale=0.7,baseline, transform shape, every node/.style={draw=none, inner sep=2pt},
                            every path/.style={->,>=latex}]

                            \node (1234) at (0,6)  {$\{1,2,3,4\}$};

                            \node (123)  at (-4,4.5) {$\{1,2,3\}$};
                            \node (124)  at (-1.5,4.5) {$\{1,2,4\}$};
                            \node (134)  at (1.5,4.5)  {$\{1,3,4\}$};
                            \node (234)  at (4,4.5)  {$\{2,3,4\}$};

                            \node (12)   at (-5,3) {$\{1,2\}$};
                            \node (13)   at (-3,3) {$\{1,3\}$};
                            \node (14)   at (-1,3)  {$\{1,4\}$};
                            \node (23)   at (1,3)  {$\{2,3\}$};
                            \node (24)   at (3,3)  {$\{2,4\}$};
                            \node (34)   at (5,3)  {$\{3,4\}$};

                            \node (1)    at (-4,1.5) {$\{1\}$};
                            \node (2)    at (-2,1.5) {$\{2\}$};
                            \node (3)    at (2,1.5)  {$\{3\}$};
                            \node (4)    at (4,1.5)  {$\{4\}$};

                            \node (empty) at (0,0)  {$\emptyset$};

                            \draw (empty) -> (1);
                            \draw (empty) -> (2);
                            \draw (empty) -> (3);
                            \draw (empty) -> (4);

                            \draw (1) -> (12);
                            \draw (1) -> (13);
                            \draw (1) -> (14);

                            \draw (2) -> (12);
                            \draw (2) -> (23);
                            \draw (2) -> (24);

                            \draw (3) -> (13);
                            \draw (3) -> (23);
                            \draw (3) -> (34);

                            \draw (4) -> (14);
                            \draw (4) -> (24);
                            \draw (4) -> (34);

                            \draw (12) -> (123);
                            \draw (12) -> (124);

                            \draw (13) -> (123);
                            \draw (13) -> (134);

                            \draw (14) -> (124);
                            \draw (14) -> (134);

                            \draw (23) -> (123);
                            \draw (23) -> (234);

                            \draw (24) -> (124);
                            \draw (24) -> (234);

                            \draw (34) -> (134);
                            \draw (34) -> (234);

                            \draw (123) -> (1234);
                            \draw (124) -> (1234);
                            \draw (134) -> (1234);
                            \draw (234) -> (1234);

                        \end{tikzpicture}
                        \caption{Face Poset}
                    \end{subfigure}
                    \caption{A graph G and the weighted directed graph corresponding to the face poset $P(2^E)$}
                    \label{fig:graph}
                \end{figure}

                Consider the graph $G=(V,E)$ with the edge set $E=\{1,2,3,4\}$ as given in
                \autoref{fig:graph}. The simplicial complex $2^E$ and its corresponding face poset are given. Note that by $\emptyset$ we mean the spanning subgraph of $G$ with no edges. Similarly, for any subgraph $H \in 2^E$, we mean the spanning subgraph of $G$ with edge set the same as $H$.
            \end{example}
        \end{subsection}

        \begin{subsection}{A short review of graph co-homology}

            In this subsection, we recall the construction of the chromatic homology associated with a graph. There are multiple formulations of chromatic homology in the literature. We define chromatic homology following Viro’s interpretation of Khovanov homology via enhanced states, adapted to the setting of chromatic homology \cite[Section 4.3]{Viro2004}. This enhanced-state viewpoint was already provided in the original categorification of the chromatic polynomial \cite{Helme_Guizon_Rong_2005_categorification} and is used throughout subsequent work on chromatic homology and its variants, for example \cite{HelmeGuizonPrzytyckiRong2006_torsion,ChmutovChmutovRong2008,PabiniakPrzytyckiSazdanovic2009,LowranceSazdanovic2017}. In this paper, we shall keep using % Point 6 addressed.
            the enhanced state formulation of chromatic homology as it will be the most convenient for describing alternating-path counts in the chromatic chain complex. We shall only be concerned with the algebra $\A_m= \Z[x]/\langle x^m \rangle$, but the constructions can also be done with a more general commutative algebra.

            Given a graph $G=(V,E)$, an enhanced state of $G$ is an ordered pair $\left( H, \epsilon \right)$ where $H$ is a spanning subgraph of $G$ and $\epsilon$ is a coloring of the connected components of $H$ by the colors $\{1, x, \ldots, x^{m-1} \}$. More precisely, let $ H \subseteq G$ be a subgraph with $V(H) = V(G)$ and $E(H) \subseteq E(G)$. If $\pi_0(H)$ denotes the collection of connected components of $H$, then $\epsilon$ is a map \[ \epsilon: \pi_0(H) \to \{1, x, \ldots, x^{m-1} \}.\]

            \subsubsection{Chain groups}

            Let $\chCom{\ast, \ast}{m}{G}$ denote the free abelian group generated by all possible enhanced states $\left( H, \epsilon \right)$. We turn this into a graded abelian group where the grading is inherited by the cardinality of $|E(H)|$. Let $n$ denote the number of edges of $G$.  That is, we have the following decomposition:

            \begin{align*}
                \chBiCom{*}{*}{m}{G} &= \bigoplus_{i=0}^{n} \chBiCom{i}{*}{m}{G}\\
                \chCom{i, \ast}{m}{G} :&= \bigoplus_{\{H \vert \, |E(H)|=i\}} \Z \left< (H, \epsilon) \right>
            \end{align*}

            Essentially $\chCom{i, \ast}{m}{G}$ consists of all possible enhanced spanning subgraphs $H$ of $G$ where $\lvert E(H)\rvert =i$, see \autoref{fig:chromatic}.

            \begin{figure}
                \centering
                \resizebox{0.65\textwidth}{!}{\includegraphics{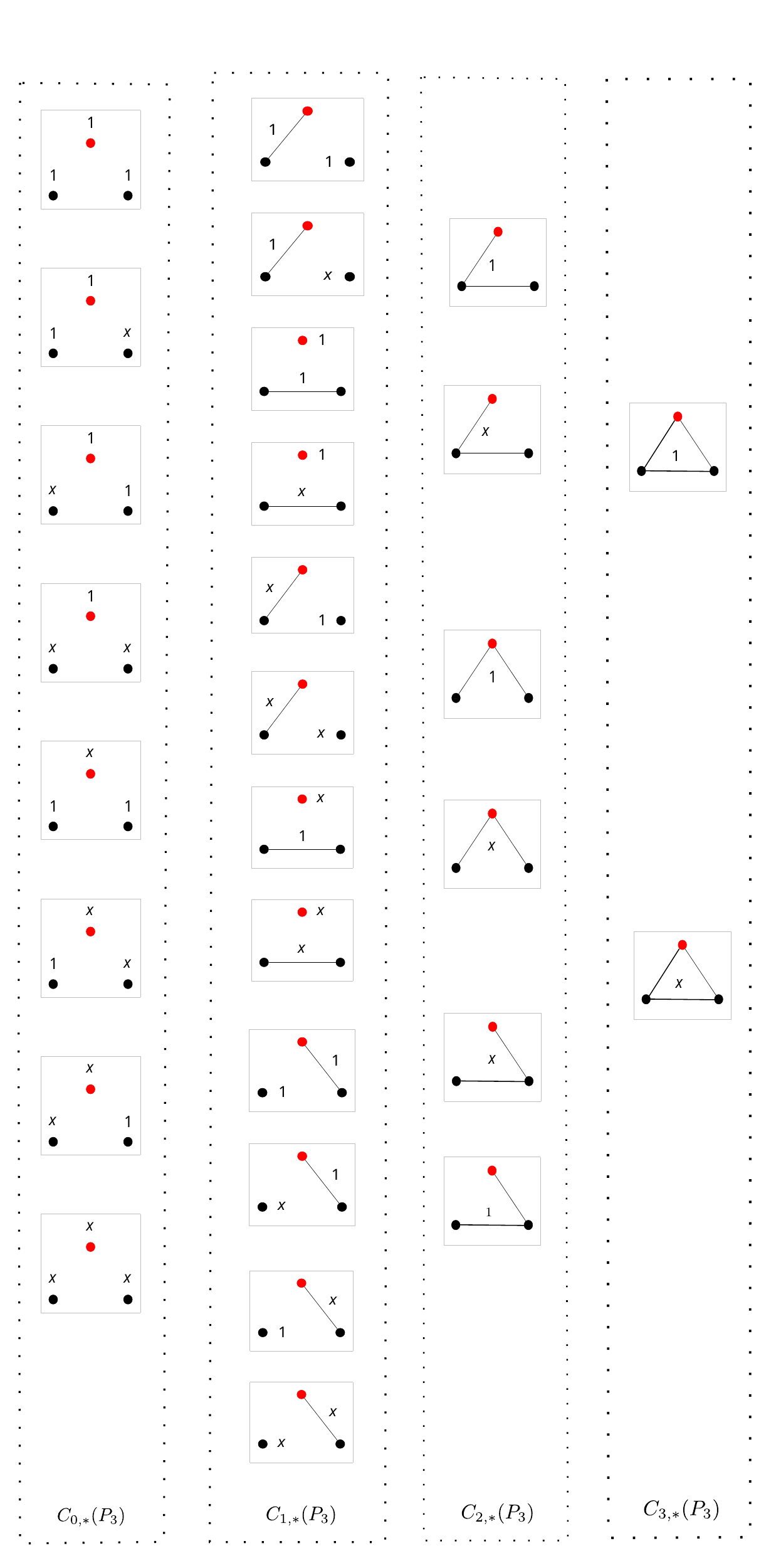}}
                \caption{Chromatic chain groups in terms of enhanced spanning subgraphs for the cycle graph with three vertices $P_3$.}
                \label{fig:chromatic}
            \end{figure}

            \begin{figure}
                \centering
                \resizebox{0.7\textwidth}{!}{\includegraphics{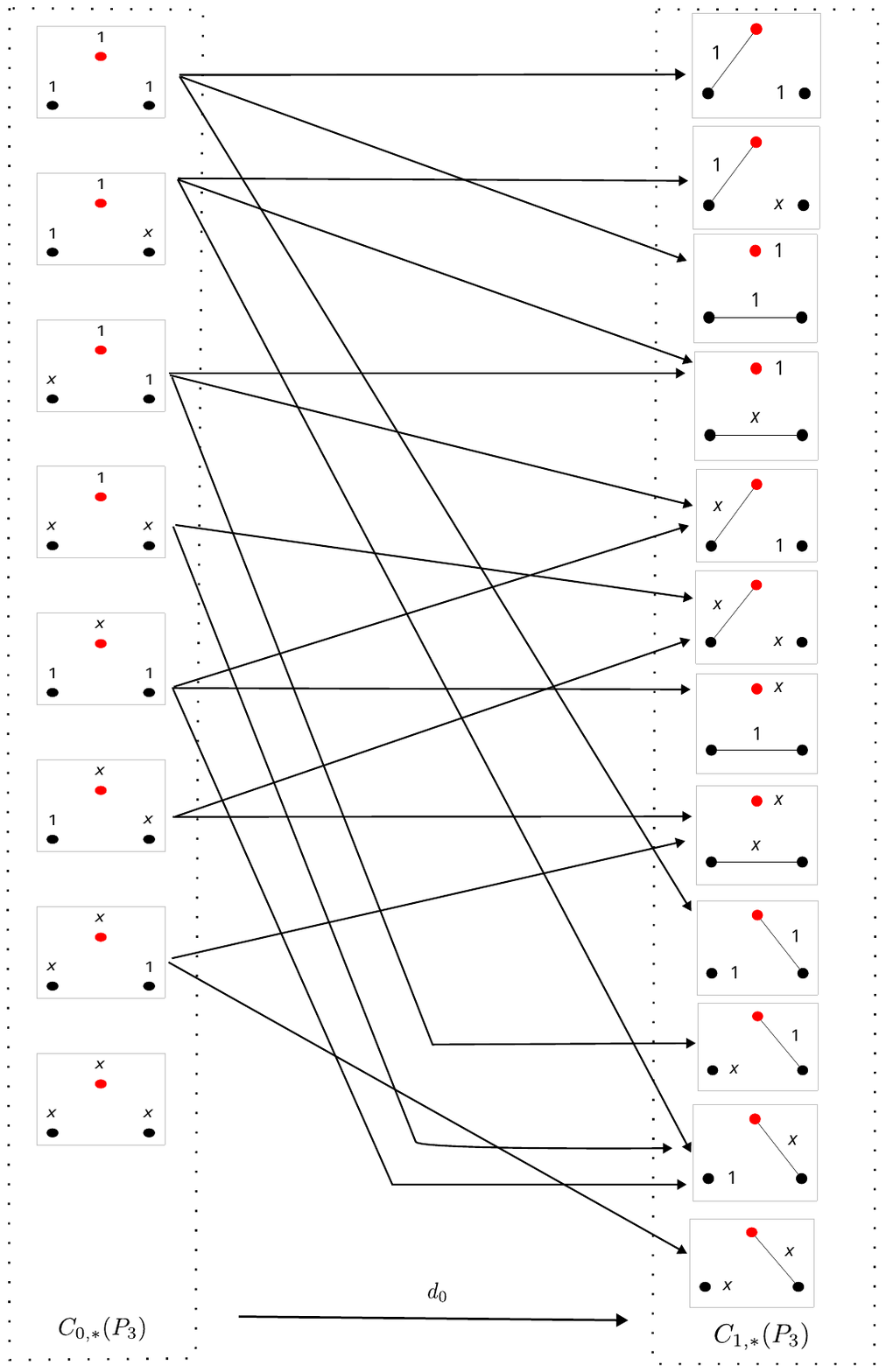}}
                \caption{Differential between the chain groups $ \chBiCom{0}{*}{2}{P_3}$ and $ \chBiCom{1}{*}{2}{P_3}$}
                \label{fig:differential}
            \end{figure}

            \subsubsection{Differential}
            In this subsection, we recall the definition of the differential
            \[
            d_{i}: {C_{i, \ast}}^{\A_{m}}(G) \to {C_{i+1, \ast}}^{\A_{m}}(G)
            \]
            in the chromatic complex.
            % The differential $d_{i}$ has two ingredients $(1.)$ multiplication/identity and $(2.)$ the sign correction which we describe as follows.

            The differential $d_{i}$ evaluated at each generator is given by:

            \[ d_{i} \left( (H, \epsilon) \right) = \sum_{e \notin E(H)} (-1)^{e} \cdot d_i^{e}\left( (H,\epsilon) \right)\]

            where $d_i^{e}$ is defined as per-edge maps and is given by $d_i^{e}\left( (H,\epsilon) \right)=(H \cup e, \epsilon_e)$.\\
            $H \cup e$ is the spanning subgraph $H$ with the edge $e$ inserted in it.
            Insertion of an edge $e$ to the spanning subgraph $H$ can result in two possible scenarios.
            The edge $e$ can either connect two disjoint connected components of $H$ or $e$ connects two vertices within the same connected component of $H$.

            Let $\pi_0(H)=\{Z_1,Z_2, \cdots, \ Z_k \}$. In the first case, see \autoref{fig:differential}, if $e$ connects two distinct components of $H$ say $Z_i$ and $Z_{j}$ then $\pi_0(H \cup e)=\{Z_1,Z_2, \cdots Z_i \cup Z_j \cup e, \cdots, Z_k \}$ and $\epsilon_{e}$ is defined as follows

            \begin{align*}
                \epsilon_e \left( Z \right) :&= \epsilon(Z) \, \,  \text{if} \, \, Z \neq Z_i, Z_{j} \\
                \epsilon_{e} (Z_i \cup Z_j \cup e) : &= \epsilon(Z_{i}) \cdot \epsilon (Z_{j})
            \end{align*}
            where $\epsilon(Z_{i}) \cdot \epsilon (Z_{j})$ is defined by the algebra multiplication in $\A_{m}$.
            Thus, $(H \cup e, \epsilon_{e}) \equiv 0$ when $\epsilon(Z_{i}) \cdot \epsilon (Z_{j}) =0$ in the algebra multiplication of $\A_{m}$ and we do not consider $\epsilon_e$ as a coloring in this case.

            In the second case, if $e$ connects two vertices within the same connected component of $H$ say $Z_{i}$ then $\epsilon_{e}$ is defined by identity, that is,
            \begin{align*}
                \epsilon_{e} (Z) :&= \epsilon(Z) \, \, \text{if} \, \, Z \neq Z_{i} \\
                \epsilon_{e} (Z_i \cup e) :&= \epsilon (Z_{i})
            \end{align*}
            Essentially, $\epsilon_e=\epsilon$ in this case.

            Next, we discuss the sign correction $(-1)^{e}$. We start with a well ordering on the edges of $G$, say $e_1 < e_2 < \ldots < e_{n}$. Then, we have a canonical basis $\mathcal{B}$ for the exterior algebra generated by $\{ e_{i} \}_{i=1}^{n}$
            \[ \bigwedge\left( e_1, e_2, \ldots, e_{n} \right) \]
            given by $ \mathcal{B} = \{ e_{i_1} \wedge e_{i_2} \wedge \ldots \wedge e_{i_{k}} \ \mid \ k, 1 \leq i_1 < i_2< \ldots < i_{k} \leq n \}$.
            Suppose the edge set of $H$ consists of $E(H)= \{ e_{k_1} < e_{k_2}< \ldots < e_{k_{i}} \}$ then $(-1)^{e}$ is defined by the sign such that \[ (-1)^{e} e_{k_1}\wedge e_{k_{2}} \wedge \ldots \wedge e_{k_{i}} \wedge e \in \mathcal{B}.  \]

            Another description of the sign correction can be given by the number of edges that are ordered before $e$. Moreover, given $H$, let $ H_{e}:= \{ e_{k_{i}}: e_{k_{i}} \in E(H) \, \,  \text{and} \, \, e_{k_{i}} < e \}$. Then we can also take $(-1)^{e} := (-1)^{|H_{e}|}$.
        \end{subsection}

        A straightforward calculation implies:
        \begin{proposition}
            The differential that is defined above satisfies $d^2=0$.
        \end{proposition}
        For any $i$, the $i$-th part of the chromatic chain complex ${C_{i}}^{\A_{m}}(G)$ inherits an additional grading known as the quantum grading that is
        \[
        {C_{i, \ast}}^{\A_{m}}(G) = \bigoplus_{j \geq 0} {C_{i,j}}^{\A_{m}}(G).
        \]
        Essentially, the quantum grading $j:= j(H, \epsilon)$ for each enhanced state $(H, \epsilon)$ is given by the formula
        \[
        j(H, \epsilon) := \sum_{i=1}^{m-1} i \times  |\{Z \in \pi_0(H)  : \epsilon(Z)= x^{i} \}|
        \]

        \begin{remark}
            Although the definition of the chromatic complex depends on a chosen ordering of the edges of the graph, it is a standard result that the resulting homology is independent of this choice (see Theorem 2.12 of \cite{Helme_Guizon_Rong_2005_categorification} and Theorem 2.5 (b) of \cite{graphcohomologiesarbitraryalgebras}).
            % Point 7 addressed.
            Let $G$ be a graph whose edges are labeled
            $\{1,2, \cdots,n\}$, and let $\sigma$ be a permutation of $\{1,2,\cdots,n\}$. We denote by $G_{\sigma}$ the graph $G$ in which the edge labeled $i$ is relabeled $\sigma(i)$ for every $i$. Thus $G$ and $G_{\sigma}$ are the same graph, equipped with two edge orderings that differ by $\sigma$. In fact, we have the following theorem:
            % Point 8 addressed.
            \begin{theorem}\cite{Helme_Guizon_Rong_2005_categorification}
                The chain complexes $C_{\ast, \ast}^{\A_m}(G)$ and $C_{\ast, \ast}^{\A_m}(G_{\sigma})$ are isomorphic and the co-homology groups are isomorphic.
            \end{theorem}
        \end{remark}

        The bi-graded homology group will be denoted by $H_{i,j}^{\A_m}(G)$ for homological grading $i$ and quantum grading $j$.

        \begin{subsection}{Broken circuit complex}
            Let $G=(V,E)$ be a graph. The \textit{edge space} $\mathcal{E}(G)$ is the $\Z_2$ vector space with the basis $E$. We can canonically identify the elements of the edge space with the subsets of the edge set. Under this identification  the sum of the two elements of $\mathcal{E}(G)$ can be identified with the symmetric difference of the corresponding sets, that is, given $S, T \in 2^E$, $S+T:=S \Delta T=S \cup T -S \cap T$. Throughout the paper, we will freely switch between these descriptions of the edge space without referring to this fact.
            \begin{definition}\label{cyclesumnotation}
                The \textit{cycle space} $\cyclespace$ is the subspace of $\mathcal{E}(G)$ generated by cycles of $G$.
            \end{definition} % Point 9 addressed.

            \par Let $\mathcal{O}=\{e_1, e_2, \cdots e_n\}$ be a fixed ordering on the  edges of $G$.
            A \textit{broken circuit} in $G$ is a set $C+e \in \mathcal{E}(G)$, where $C$ is a cycle in $G$ and $e$ is the least ordered edge in the cycle. Note that we use a different convention for the broken circuit, compared to \cite{Chandler_Sazdanovic_2022}. These conventions are equivalent up to reversing the order $\mathcal{O}$.
            Define $\NBC(G)$ to be the abstract simplicial complex whose face set is the collection of all subsets $S \subseteq E$ such that S contains no broken circuits. Let $\operatorname{BC}(G)$ be the complement of $\NBC(G)$ in $2^E$, that is, $\operatorname{BC}(G)=2^E \setminus \NBC(G)$. % Point 10 addressed.
            Whenever $G$ and $\mathcal{O}$ are clear from the context, then we shall omit them from the definition and denote the broken and non-broken circuit complex by just $\operatorname{BC}(G)$ and $\NBC(G)$ respectively.
            Recall from Subsection \ref{algebraic morse theory} that given a matching $\M$ on a directed graph $\G = \left(V, E \right)$, the symbol $V^{\M}$ denotes the
            set of critical vertices.
            Given a graph $G=(V,E)$, let $S \subseteq E$. Let us denote the subgraph spanned by $S$ by $G(S)$. Let $\lambda(S)$ denote the number of connected components of the spanning subgraph $G(S)$, that is, $\lambda(S)= |\pi_{0}(G(S))|$.  Let $k(S)$ denote the integer partition of $\lvert V \rvert$ consisting of the sizes of the connected components of $S$.
            \begin{lemma}\cite[Lemma 4.3]{Chandler_Sazdanovic_2022} \label{Initial Matching}
                Let $G = (V, E)$ be a graph with a fixed ordering $\mathcal{O}$ of its edges. The face poset $P(2^E)$ has an acyclic matching $\M_{\operatorname{BC}}$ such that each matched pair $(S, S') \in \M_{\operatorname{BC}}$, satisfies
                \begin{itemize}
                    \item $S'=S \cup e$, for some edge $e \in E$,
                    \item $k(S)=k(S')$,
                    \item $\lambda(S)=\lambda(S')$,
                \end{itemize} and the corresponding Morse complex is spanned by \[ \left(2^E \right) ^{\M_{\operatorname{BC}}}= \nbc(G).  \]
            \end{lemma}

            Observe that, for a given graph $G=(V,E)$ and the commutative graded $R$-algebra $A=\A_m$(here $R=\Z$), the chromatic complex is the cochain complex $C_{\ast,\ast}(F^{\mathrm{Ch}}_{G,A},c)$ obtained from the thin poset
            construction of Chandler \cite[Definitions~2.1, 4.6 and 6.2]{chandler2019posetscwposetscategorification},
            applied to the chromatic functor $F^{\mathrm{Ch}}_{G,A}\colon P(2^E)\to R\text{-}\mathbf{gmod}$ on the
            Boolean lattice, with the balanced coloring $c$. The balanced coloring $c$ is determined by the ordering of the edges \cite[Section~5]{Chandler_Sazdanovic_2022}. The only posets
            required are thin posets, namely $P(2^E)$ and its order ideals. Since
            $\NBC(G)$ is a lower order ideal of $P(2^E)$, it is thin, and restricting
            $F^{\mathrm{Ch}}_{G,A}$ and $c$ to $\NBC(G)$ yields the complex
            $(\nbcCom{m}{G},\nbcDiff{m})$ for $A=\A_m$. We denote its homology by
            $H_{\ast,\ast}^{\A_m}(\nbc)$. As $BC(G)$ is an upper order ideal, the
            sequence $0\to C^{*}(F|_{BC(G)})\to C^{*}(F)\to C^{*}(F|_{\NBC(G)})\to 0$ is exact
            \cite[Theorem~6.7]{chandler2019posetscwposetscategorification}, so $\nbcCom{m}{G}$ is a quotient
            complex, and not a sub-complex, of $\chBiCom{*}{*}{m}{G}$. In this framework,
            the discrete Morse matching $\M_{\operatorname{BC}}$ of
            \Autoref{Initial Matching} identifies $\NBC(G)$ with the critical elements of
            $P(2^E)$. Consequently, by \Autoref{Initial Matching} and
            \Autoref{dmt maintheorem}, Chandler and Sazdanovi\'c proved the following
            theorem. % Point 11,12,13 Addressed.

            \begin{theorem}
                {\cite[Theorem 5.8]{Chandler_Sazdanovic_2022}}\label{NBC Isomorphism}

                Let $A=\oplus_{k \in \mathbb{Z}} A_k$ be a commutative, graded $R$-algebra with identity such that each $A_k$ is a free $R$-module of finite rank. % Point 14 addressed.
                For any graph $G=(V, E)$, and any fixed ordering of the edge set $E$,
                \[ H_{\ast, \ast}^{A}(G) \cong H_{\ast, \ast}^{A}(\nbc).\]
            \end{theorem}

            Here, $H_{\ast, \ast}^{A}(G)$ denotes the chromatic homology of $G$ over algebra $A$ and $H_{\ast, \ast}^{A}(\nbc)$ denotes the homology associated with the poset $\NBC(G)$ using the same construction as the chromatic homology of $G$. % Point 15 addressed.
            \Autoref{NBC Isomorphism} can be thought of as the categorification of Whitney’s broken circuit theorem for the chromatic polynomial.\\

            In this paper, we will be only concerned with $\A_m$ algebras. Observe that \Autoref{NBC Isomorphism} can be interpreted in the following way.

            \begin{observation}\label{observation:NBC}
                Given an ordering on the edges of a graph $G$, the homology of the
                chromatic complex over $\A_m$, $H_{*,*}^{\A_m}(G)$, is isomorphic to the homology of the quotient complex $\nbcCom{m}{G}$ generated by enhanced $NBC$ subgraphs. Moreover, \Autoref{dmt maintheorem} implies that
                $\chBiCom{*}{*}{m}{G}$ is chain homotopy equivalent to $\nbcCom{m}{G}$.
            \end{observation} % Point 15 addressed.

        \end{subsection}

        \begin{subsection}{Activity word on a spanning tree}

            \begin{definition}
                Let $G$ be a graph with a fixed ordering on its edges. Then for a spanning tree $T$ of $G$ and for each edge $e \in T$, the \textit{cut-set} of $e$ is defined as, $\cut(T,e)=\{e' \in E(G) \mid e' \text{ connects two components of } T \setminus e\}$. If $f \not\in T$, then the \textit{cycle} of $f$ is defined as $\cyc(T,f)=\{f' \in E(G) \mid f' \text{ belongs to the unique cycle in } T \cup f\}$.
            \end{definition}
            \begin{proposition}
                For any spanning tree $T$, $e \in \cut(T,p) \iff p \in \cyc(T,e).$
            \end{proposition}
            \begin{proof}
                Let $e \in \cut(T,p)$. This implies that $p \in T$ and $e \notin T$ and $e$ connects two disjoint components of $T \setminus p$, say $C_1, C_2$. Now, if we insert the edge $e$ into $T$, it creates a unique cycle which contains both $p$ and $e$. This proves the proposition.
            \end{proof}

            Given an edge $e \in T$, if $e$ is the smallest edge in $\cut(T,e)$, then $e$ is said to be an \textit{internally active (live)} edge ($L$) of the tree $T$, otherwise it is said to be an \textit{internally inactive (dead)} edge ($D$). Similarly, if $e \not\in T$ and $e$ is the smallest edge in $\cyc(T,e)$, then $e$ is said to be an \textit{externally active (live)} edge ($l$) of $T$ or else it is said to be \textit{externally inactive (dead)} edge ($d$) of $T$.
            \par So given a spanning tree $T$ of a graph $G=(V,E)$, we define an activity word function for an edge $e$, $a_T(e):E(G) \rightarrow \{L,D,l,d\}$. The activity word function for the tree $W(T)$ is defined to be the formal free product of $a_T(e)$ in increasing order of edges $e$.	We denote the collection of internally inactive edges of a tree by $\IN(T)$, the collection of internally active edges of a tree by $\IA(T)$, and the collection of externally inactive edges by $\EN(T)$. Observe that in an NBC spanning tree, there are no externally active edges. In other words, if $T$ is an NBC spanning tree and $e \notin T$, then $a_T(e)=d$.

            % \vspace{-2 cm}
            \begin{figure}
                \centering
                % First subfigure
                \begin{subfigure}[b]{0.4\textwidth}
                    \centering
                    \begin{tikzpicture}[vertex/.style={circle,fill,inner sep=2pt}]
                        \node[vertex] at (0,0) (a) {};
                        \node[vertex] at (2,3) (b) {};
                        \node[vertex] at (4,0) (c) {};
                        \draw (a)--(b) node[midway,above,yshift=0.2cm,xshift=-0.3cm] {$1$};
                        \draw (c)--(a) node[midway,above] {$3$};
                        \draw (b)--(c) node[midway,above,yshift=0.2cm,xshift=0.3cm]{$2$};
                    \end{tikzpicture}
                    \caption{A graph $G$, with a fixed ordering}
                \end{subfigure}
                \hfill
                % Second subfigure
                \begin{subfigure}[b]{0.4\textwidth}
                    \centering
                    \begin{tikzpicture}[vertex/.style={circle,fill,inner sep=2pt}]
                        \node[vertex] at (0,0) (a) {};
                        \node[vertex] at (2,3) (b) {};
                        \node[vertex] at (4,0) (c) {};
                        \draw (a)--(b) node[midway,above, yshift=0.2cm,xshift=-0.3cm] {$1$};
                        \draw (c)--(a) node[midway,above] {$3$};
                        \draw[dotted] (b)--(c) node[midway,above,yshift=0.2cm,xshift=0.3cm]{$2$};
                    \end{tikzpicture}
                    \caption{A spanning tree $T$ of $G$, with $W(T) = LdD$}
                \end{subfigure}

                \caption{A graph $G$ with an edge ordering and a spanning tree $T$}
            \end{figure}

        \end{subsection}
    \end{section}

    \begin{section}{Spanning Tree Model for Chromatic Homology}
        We will construct a spanning tree model for the chromatic complex, namely, a complex generated by signed spanning trees that is chain-homotopy equivalent to the chromatic complex. Our approach is to decompose the $\NBC(G)$ part of the chromatic complex into interval pieces indexed by NBC spanning trees. This interval decomposition is inspired by Björner’s work on homology and shellability; in fact, the interval pieces arise from a shelling of $\NBC(G)$ \cite{Bjorner1982,bjoner}.\\

        Since the $\NBC(G)$ part of the chromatic complex is generated by the enhanced NBC subgraphs, it admits a similar decomposition into interval pieces. On each interval piece, we construct an acyclic matching with exactly two critical points. Applying algebraic Morse theory then produces a chain complex generated by signed spanning trees, where the sign $\pm$ records the two critical points associated with the interval piece of a given NBC spanning tree. Finally, we obtain an explicit combinatorial description of the differential by counting alternating paths between critical points induced by the matching.\\

        In this entire section, we consider the chromatic complex of a graph over the algebra $\A_2$. Also, we consider any graph $G$ with a distinguished vertex on it, which we call a root. Although the matching that we provide on the $\NBC(G)$ depends on the choice of the distinguished vertex on the graph, the chromatic or spanning tree homology is independent of it.
        Recall that the non-broken circuit complex $\NBC(G)$ is a simplicial
        subcomplex of $2^E$ with facets as spanning trees that do not contain any broken circuit. In this section, we first provide an acyclic matching on the chromatic complex of a tree, and using this matching, we can provide a matching to the entire sub-complex $\NBC(G)$ for any given graph $G$.\\

        \par Let $G = (V, E)$ be a simple graph with a fixed ordering $\mathcal{O} = \{e_1, e_2, \dots, e_n\}$ on its edges $E$. Let $\lvert V \rvert = N$. Let $T$ be any spanning tree of $G$. Then we write $T$ as an ordered $(N-1)$-tuple in increasing order of its edges, that is, $T = [e_{i_1}, e_{i_2}, \dots, e_{i_{N-1}}]$ if and only if $e_{i_1} < e_{i_2} < \dots < e_{i_{N-1}}$.\\

        One can provide an ordering on the facets, that is, spanning trees, called \textit{lexicographic} ordering, defined as follows:

        \[
        T_i = [e_{i_1}, e_{i_2}, \dots, e_{i_{N-1}}] < T_j = [e_{j_1}, e_{j_2}, \dots, e_{j_{N-1}}]
        \]

        if and only if $e_{i_k} = e_{j_k}$ for all $k < p$ and $e_{i_p} < e_{j_p}$, for some $p \in \{1,2,\dots,N-1\}$.  \\

        In this subsection, we give proofs of several propositions on NBC spanning trees that we shall need. These results are known in the more general setting of matroids \cite{bjoner}, and we include the proofs for the convenience of the reader. We begin with the following basic construction on NBC spanning trees, which is used repeatedly in the rest of the sections. % Point 17 addressed.

        \begin{observation}\label{observation}
            If $T$ is an NBC spanning tree and $e \in T$ is an internally dead edge, then there exists a spanning tree
            \[ \widetilde{T} := \left(T - e \right) \cup \min \cut \left( T, e  \right). \]
            It is also an NBC spanning tree. Here, $\min \cut\left(T,e\right)$ denotes the minimum edge in the cut set of $e$ in $T$.
        \end{observation}

        \begin{proposition}{\label{shelling order}}
            The \textit{lexicographic} ordering is a shelling order of $\NBC(G)$. Thus, $\NBC(G)$ is a shellable complex.
        \end{proposition}
        \begin{proof}
            Let $G=(V,E)$ be a graph with $\lvert V\rvert=k$. Let $\{T_1 < T_2< \cdots <T_p\}$ be a lexicographic ordering on the facets of $\NBC(G)$. Let $T_i, T_j$ be any pair of facets such that $T_i < T_j$. \\
            Let $T_i=[e_{t_1 ^i}, e_{t_2 ^i}, \cdots e_{t_{k-1}^i}]$ and $T_j=[e_{t_1 ^j}, e_{t_2 ^j}, \cdots e_{t_{k-1}^j}]$.
            Then there exists a smallest integer $l \in \{ 1,2, \cdots ,k-1\}$ such that $e_{t_s^i}=e_{t_s ^j}, \text{for all} \, \, s<l$ and $e_{t_l ^i}<e_{t_l ^j}$. Moreover, it follows that $e_{t_l ^i}< e_{t_s^j}, \text{for all} \, \,  s>l$, since $e_{t_s ^j}< e_{t_{s+1}^j} \text{for all} \, \,   s\in \{1, 2,  \cdots ,  k-2 \}$. Thus, we have the following chain of inequalities \[ e_{t^{i}_{l}} < e_{t^{j}_{l}} < e_{t^{j}_{l+1}} < \cdots < e_{t^{j}_{k-1}}. \]

            Let $T' :=T_j \cup \{ e_{t_l ^i}\}$. Since $T_j$ is a spanning tree, $T'$ contains a unique cycle $C$. The cycle contains an edge $e \in C$  satisfying the following conditions:
            \begin{itemize}
                \item $e= e_{t^{j}_{p}} $ and $p \geq l$.
                \item $e \notin T_{i}$.
            \end{itemize}
            Such an edge $e \in C$ must exist. Otherwise, $C \subset T_{i}$ contradicts the fact that $T_{i}$ is a spanning tree. Also, such an edge $e$ is internally dead with respect to $T_{j}$, since $e_{t^{i}_{l}} \in \text{cut}(T_{j}, e) $ and $e_{t^{i}_{l}}< e$.
            % Point 18 addressed.

            Let $  T=T'-\{e_{t_p^j}\}$. Note that, $T < T_{j}$ in the lexicographic order and \[ T_{i} \cap T_{j} \subseteq T \cap T_{j} = T_{j}- e.  \]

            If $T$ is an NBC spanning tree, we are done. Otherwise, we can change $T$ by \[ \widetilde{T} :=  \left(T- e_{t_l ^i} \right) \cup \min \cut\left( T, e_{t_l ^i} \right).  \]
            Then by \Autoref{observation}
            $\widetilde{T}$ is a facet of $\NBC(G)$ and by construction $\widetilde{T} < T < T_{j}$. This modified $\widetilde{T}$ also satisfies \[ T_{i} \cap T_{j} \subseteq \widetilde{T} \cap T_{j} = T_{j}- e.  \]
            % Point 19 addressed.
            Hence, the lexicographic order is a shelling.
        \end{proof}
        Recall from \Autoref{partition} that if we have a shellable complex, then we have a partition of the complex.
        We shall use this lexicographic ordering on $\NBC(G)$ to extend the matching on the whole NBC complex using these partitions.

        \begin{proposition} {\label{First Tree}}
            If $T_1$ is the NBC spanning tree with the smallest lexicographic order, then all the internal edges of $T_1$ are active.
        \end{proposition}
        \begin{proof}
            Assume there is an internally inactive edge $e \in T_1$, then let $f \in \cut(T_1,e)$ be the minimum edge with $f \neq e$. Consider $T=T_1 \cup \{f\} \setminus \{e\}$ then, $a_{T}(f)=L$ and $a_{T}(e)=d$ and the activity of the rest of the edges of $T$ remains same with that of $T_1$. Since $f < e$, it follows that $T < T_1$, which contradicts the minimality of $T_1$ with respect to the lexicographic order.
        \end{proof}

        Also recall that, for a tree $T$, we denote the collection of inactive edges of $T$ by $\IN(T)$. Each NBC spanning tree $T \subseteq G$ is a facet of the complex $\NBC(G)$. % Point 20 addressed.
        As discussed in Subsection \ref{posets}, $\mathcal{R}(T)$ corresponds to the restriction of
        $T$ induced by this shelling order. Then we have the following lemma:
        \begin{proposition}\label{inactiveedgelemma}
            Let $\{T_1,T_2, \cdots, T_p\}$ be the collection of NBC spanning trees of a graph $G$, then $\R(T_i)=\IN(T_i) , \forall i \in \{1,2, \cdots, p\}$.
        \end{proposition}
        \begin{proof}
            For $T_1$, we have that $\R(T_1)=\emptyset$, and from \Autoref{First Tree} the result follows. For $i\geq 2$ let us recall that $e \in \R(T_i) \iff T_i-\{e\} \in \Delta_{i-1}$. Hence there exists some $j\leq i-1$, such that $T_i-\{e\} \subseteq T_j$. Thus, $T_{j}= \left( T_{i}-e  \right)\cup f$ where $f \in \cut(T_{i}, e)$. Now the edge sets of $T_{i}$ and $T_{j}$ agree except for $e$ and $f$. Since $T_{j} < T_{i} $, we conclude $f < e$. In other words, this implies $e$ must be internally dead with respect to $T_{i}$. Hence, we conclude $\R(T_{i}) \subseteq \IN(T_{i})$.

            In the opposite direction, let us choose $e \in \IN(T_{i})$. As $e$ is internally dead, $e \in \cut(T_{i},e)$ is not minimal. Thus, there exists $f \in \cut(T_{i}, e)$, $f < e$ and $f$ is minimal in the set $\cut(T_{i},e)$. We look at the tree $T:= (T_{i}-e) \cup f$. Clearly $T< T_{i}$ by the definition of lexicographic ordering. It remains to show that $T$ is an NBC spanning tree.

            Assume, for the sake of contradiction, % Point 23 addressed.
            that $T$ is not an NBC spanning tree. Now $T_{i}-e$ is a disjoint union of two trees $T_{i}^{(1)}$ and $T_{i}^{(2)}$.  If $T$ contains a broken circuit of the form $\left(C+g \right) \subset T$, then we note that $f \in C$. Otherwise, $C+g$ would be a broken circuit in $T_{i}$, contradicting the hypothesis that $T_{i}$ is an NBC spanning tree. Thus, $f \in C$ and $g < f$. Now $g \notin \cut(T_{i},e)$ as $f$ is the minimal element in $\cut(T_{i},e)$. Hence $g$ must connect two vertices in $T_{i}^{(1)}$ or in  $T_{i}^{(2)}$.  This is a contradiction as the broken circuit $C+g$ lies in $T_{i}$. Hence, $T$ is an NBC spanning tree and thus $T= T_{j}$ for some $j< i$. This concludes the proof.

            \begin{figure}[!htb]
                \begin{center}
                    \hspace{0cm} \scalebox{1}{\includegraphics{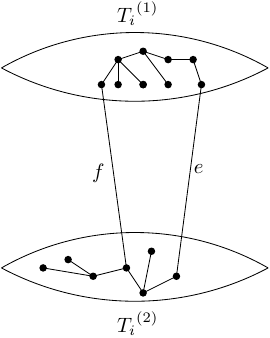}}
                    \vspace{0cm}\caption{Two components of $T_i - e$ connected by $f$}
                    \label{the relations}
                \end{center}
            \end{figure}
        \end{proof}

        We shall end this subsection with an important observation which would be crucial to prove the acyclicity of the matching given on the cochain complex $\nbcCom{2}{G}$.

        \begin{lemma} \label{crucial}
            Suppose $S \in [\R(T_{i}), T_{i}]$ is such that $S \cup{e} \in [\R(T_{j}),T_{j}]$, then $j \geq i$.
        \end{lemma}
        \begin{proof}
            If $e \in T_{i}$, then $j=i$. So without loss of generality, we may assume $e \notin T_{i}$. Let us assume $j < i$. Then, by hypothesis \[ \R(T_{j}) \subseteq S\cup e \subseteq T_{j}.   \]
            Now, either $e \in \R(T_{j})$ or $e \notin \R(T_{j})$. We shall show that both of these cases lead to a contradiction.

            If $e \notin \R(T_{j})$, then the first inclusion implies $\R(T_{j}) \subset (S\cup e)-e = S$. But then $S \in [\R(T_{i}), T_{i}] \cap [\R(T_{j}), T_{j}] $, contradicting the fact that they are disjoint.
            Assume $e \in \R(T_{j})$. By the definition of $\R(T_{j})$, this implies $T_{j} -e \in \Delta_{j-1}$ where $\Delta_{j-1}$ is the complex generated by $\{T_1, T_2, \ldots, T_{j-1}\}$. Now, we have the following chain of inclusions
            \[ S= (S\cup e)- e \subset T_{j}-e \subset \Delta_{j-1}.\]
            Then, $S \in \Delta_{j-1} \cap [\R(T_{i}), T_{i}]$. The only possibility of such an intersection is if $j-1=i$, since $j-1 < j < i$, which leads to a contradiction. Hence, $j \geq i.$

        \end{proof}

        \subsection{Acyclic matching on chromatic complex of a tree}
        {\label{Matching on tree}}
        Let $T$ be a tree, rooted at the vertex $v_{root}$. Recall that the generators of the chromatic complex of a graph correspond to enhanced spanning subgraphs \cite{Helme_Guizon_Rong_2005_categorification}. Formally, an enhanced spanning subgraph $(H, \epsilon)$  of a graph $G$ is a spanning subgraph $H$ of $G$ together with a set-theoretic map
        $\epsilon : \pi_0(H) \rightarrow \{1,x\}$ where $\pi_0(H)$ is the collection of connected components of $H$.\\
        Recall that, algebraic Morse theory associates a directed weighted graph $\G(\C,{\BO})$ to a cochain complex $\C$ with basis $\BO$. Now for a tree $T=(V,E)$, with $n$  edges, we will inductively provide an acyclic matching $\M_{T}$ on the chromatic complex $\chCom{\ast, \ast}{2}{T}$, that is, on the graph $\G(\chCom{\ast, \ast}{2}{T})$. The induction statement is on the number of edges of the rooted tree $T$.

        \begin{enumerate}
            \item \textbf{Base case: $n=1$.} For a rooted tree $T$, with one edge, let $v_{root}$ denote the root vertex. The matching in this case is shown in \autoref{base case matching}.
            \begin{figure}
                \centering
                \begin{tikzpicture}[
                    dot/.style={circle, fill, inner sep=1.3pt},
                    diff/.style={-{Stealth[length=2.2mm]}, gray!70, thick},
                    match/.style={-{Stealth[length=3mm]}, red!70!black, ultra thick}]
                    \newcommand{\twostate}[6]{ %
                    \node[dot,label={[font=\small]above:$#4$},label={[font=\scriptsize]below:$v_{root}$}] (#1d) at (#2-0.5,#3) {};
                    \node[dot,label={[font=\small]above:$#5$},label={[font=\scriptsize]below:$v_m$}] (#1m) at (#2+0.5,#3) {};
                    \begin{scope}[on background layer]\node[draw, rounded corners=2pt, #6, fit=(#1d)(#1m), inner xsep=9pt, inner ysep=11pt, yshift=1pt] (#1) {};\end{scope}}
                    \newcommand{\onestate}[5]{ %
                    \node[dot,label={[font=\scriptsize]below:$v_{root}$}] (#1d) at (#2-0.5,#3) {};
                    \node[dot,label={[font=\scriptsize]below:$v_m$}] (#1m) at (#2+0.5,#3) {};
                    \draw[thick] (#1d) -- (#1m) node[midway, above, font=\small] {$#4$};
                    \begin{scope}[on background layer]\node[draw, rounded corners=2pt, #5, fit=(#1d)(#1m), inner xsep=9pt, inner ysep=11pt, yshift=1pt] (#1) {};\end{scope}}
    
    \twostate{a}{0}{3.6}{1}{1}{black!70!black, fill=none}
                    \twostate{b}{0}{2.1}{1}{x}{red!70!black, very thick, dotted}
                    \twostate{c}{0}{0.6}{x}{1}{black!70!black, fill=none}
                    \twostate{d}{0}{-0.9}{x}{x}{red!70!black, very thick, dotted}
                    \onestate{e1}{4.5}{3.6}{1}{black!70!black, fill=none}
                    \onestate{ex}{4.5}{0.6}{x}{black!70!black, fill=none}
                    \draw[diff] ([yshift=-0.2cm]a.east) -- ([yshift=-0.2cm]e1.west);
                    \draw[diff] (b.east) -- ([yshift=-0.2cm]ex.west);
                    \draw[diff] ([yshift=-0.2cm]c.east) -- ([yshift=-0.2cm]ex.west);
                    \draw[diff] (d.east) -- ++(1.2,0) node[right, gray!70!black] {\scriptsize $0$};
                    \draw[match] ([yshift=0.2cm]e1.west) -- ([yshift=0.2cm]a.east);
                    \draw[match] ([yshift=0.2cm]ex.west) -- ([yshift=0.2cm]c.east);
                    \node at (0,4.8) {$i=0$};
                    \node at (4.5,4.8) {$i=1$};
                \end{tikzpicture}
                \caption{The chromatic complex of the tree with one edge. The light(grey) arrows denote the differential, dark(red) arrows the matching $\M_T$ (reversed), and the states in dotted boxes are critical.}
                \label{base case matching}
            \end{figure}

            In this case, we consider the tree with two vertices $v_{root}$ and $v_m$. This is certainly a matching. The chromatic complex is supported at $0$ and $1$ homological gradings. We call the critical vertices with respect to this matching critical states. The critical states of this matching are supported at grading $0$ and consist of two states, $H_1$ and $H_2$. The enhancement of $H_1$ is $1$ at the root vertex and $x$ at the vertex $v_m$, whereas the enhancement of $H_2$ is $x$ on both of the vertices $v_{root}$ and $v_m$.

            \item \textbf{Inductive step.}
            Assume that the matching has been constructed for every rooted
            tree with $n-1$ edges, with exactly two critical states in
            homological grading $0$. The first critical state has enhancement
            $1$ at the root and $x$ at every other vertex, whereas the second
            has enhancement $x$ at every vertex.

            Let $T$ have $n>1$ edges and root $v_{root}$. Choose the least ordered
            leaf edge $e_{i_m}$ among those whose leaf vertex $v_{i_m}$ is
            different from $v_{root}$. Thus $T\setminus\{v_{i_m}\}$ is a tree with
            $n-1$ edges, rooted at $v_{root}$.

            Partition the enhanced spanning subgraphs of $T$ into
            \[
            \begin{aligned}
                \mathcal V_p
                &= \{(H,\epsilon)\mid e_{i_m}\in E(H)\},\\
                \mathcal V_q
                &= \{(H,\epsilon)\mid e_{i_m}\notin E(H)\}.
            \end{aligned}
            \]
            Since $v_{i_m}$ is a leaf of $T$, it is isolated in every state
            belonging to $\mathcal V_q$.

            For $(H,\epsilon)\in\mathcal V_p$, put
            $H'=H\setminus\{e_{i_m}\}$. Let $Z_{v_{i_m}}$ be the component
            of $H$ containing $v_{i_m}$, and let $Z_{\mathrm{adj}}$ be the
            component of $H'$ containing the other endpoint of $e_{i_m}$.
            Pair $(H,\epsilon)$ with $(H',\epsilon')$, where
            \[
            \epsilon'(Z)=
            \begin{cases}
                1,
                & Z=\{v_{i_m}\},\\
                \epsilon(Z_{v_{i_m}}),
                & Z=Z_{\mathrm{adj}},\\
                \epsilon(Z),
                & \text{otherwise}.
            \end{cases}
            \]
            These pairs match every state in $\mathcal V_p$ with exactly one
            state in $\mathcal V_q$ whose isolated leaf has enhancement $1$.
            The differential along $e_{i_m}$ takes the latter state to the
            former with coefficient $\pm1$, so these are valid matching edges.\\
            The remaining states form
            \[
            \mathcal V_r
            = \{(H,\epsilon)\in\mathcal V_q
            \mid \epsilon(\{v_{i_m}\})=x\}.
            \]
            Deleting the isolated leaf defines a bijection
            \[
            \Psi((H,\epsilon))
            =
            \left(
            H\setminus\{v_{i_m}\},
            \epsilon|_{\pi_0(H\setminus\{v_{i_m}\})}
            \right)
            \]
            between $\mathcal V_r$ and the enhanced spanning subgraphs of
            $T\setminus\{v_{i_m}\}$.

            For each matched pair in
            $\M_{T\setminus\{v_{i_m}\}}$, match its two inverse images
            under $\Psi$. Leave the inverse images of the critical states
            unmatched. Together with the pairs constructed in Item~(2), this
            defines $\M_T$.

            Its only critical states are the lifts of the two critical states
            of the smaller tree. Both lie in homological grading $0$.
            Adjoining the isolated vertex $v_{i_m}$ with enhancement $x$
            gives a critical state with enhancement $1$ at the root and $x$
            at every other vertex, and a critical state with enhancement $x$
            at every vertex. \autoref{neg-cross} summarizes the construction.

            \begin{figure}[ht]
                \centering
                \begin{tikzpicture}[vertex/.style={circle,fill,inner sep=2pt},scale=0.7]
                    \node[vertex] [label=above:$x$] at (0.5,0.5) (a) {};
                    \node[vertex] [label=below:$1$,label=above:$v_{i_{m}}$] (b) [above right=of a] {};
                    \node [draw,thick,circle through=(a)] at (0,0) {};
                    \node at (0,0) {$H'$} (a);
                    \draw[->] (2,0.25) -- (4,0.25) node[midway,above] {$d$};
                    \draw[<-] (2,-0.75) -- (4,-0.75) node[midway,above] {$\M$};
                    \node[vertex] at (6.5,0.5) (c) {};
                    \node[vertex] [label=above:$v_{i_{m}}$] (d) [above right=of c] {};
                    \draw (c) -- (d) node[midway,above] {$x$};
                    \node [draw,thick,circle through=(c)] at (6,0) {};
                    \node at (6,0) {$H$} (c);
                \end{tikzpicture}
                \hfill
                \begin{tikzpicture}[vertex/.style={circle,fill,inner sep=2pt},scale=0.7]
                    \node[vertex] [label=above:$1$] at (0.5,0.5) (a) {};
                    \node[vertex] [label=below:$1$,label=above:$v_{i_{m}}$] (b) [above right=of a] {};
                    \node [draw,thick,circle through=(a)] at (0,0) {};
                    \node at (0,0) {$H'$} (a);
                    \draw[->] (2,0.25) -- (4,0.25) node[midway,above] {$d$};
                    \draw[<-] (2,-0.75) -- (4,-0.75) node[midway,above] {$\M$};
                    \node[vertex] at (6.5,0.5) (c) {};
                    \node[vertex] [label=above:$v_{i_{m}}$] (d) [above right=of c] {};
                    \draw (c) -- (d) node[midway,above] {$1$};
                    \node [draw,thick,circle through=(c)] at (6,0) {};
                    \node at (6,0) {$H$} (c);
                \end{tikzpicture}
                \begin{tikzpicture}[vertex/.style={circle,fill,inner sep=2pt},scale=0.7]
                    \node[vertex] at (0.5,0.5) (a) {};
                    \node[vertex] [label=below:$x$,label=above:$v_{i_{m}}$] (b) [above right=of a] {};
                    \node [draw,thick,circle through=(a)] at (0,0) {};
                    \node at (0,0) {$H'$} (a);
                    \draw[->] (2,0.25) -- (4,0.25) node[midway,above] {$d$};
                    \draw[<-] (2,-0.75) -- (4,-0.75) node[midway,above] {$\M_{T_{n-1}}$};
                    \node[vertex] at (6.5,0.5) (c) {};
                    \node[vertex] [label=above:$v_{i_{m}}$,label=below:$x$] (d) [above right=of c] {};
                    \node [draw,thick,circle through=(c)] at (6,0) {};
                    \node at (6,0) {$H'$} (c);
                \end{tikzpicture}
                \caption{The matching for any tree with $n$ edges}
                \label{neg-cross}
            \end{figure}

            % \item Now we provide the matching on the set of remaining enhanced subgraphs of $T$ those we denote by $ \mathcal{V}_r$. By induction let $\M_{T \setminus \{v_{i_m}\}}$ be a matching on the chromatic complex of  $T \setminus \{ v_{i_m} \}$ which is a tree with $n-1$ edges.
            % We observe that in $\mathcal{V}_r$, all enhanced spanning subgraphs have $\{ v_{i_m} \}$ as a connected component with cardinality $1$ and enhancement $x$. So the enhanced spanning subgraphs of $ \mathcal{V}_r$ are in bijective correspondence with the enhanced spanning subgraphs of $T \setminus \{ v_{i_m} \}$. The bijective correspondence is provided by the map $\Psi((H, \epsilon))= (H\setminus \{ v_{i_m} \}, \epsilon_{|H\setminus \{ v_{i_m} \}})$. We can use the matching  $\M_{T \setminus \{ v_{i_m} \}}$ to provide a matching on $ \mathcal{V}_r$ in the following way. We pair off $(H, \epsilon)$ with $\Psi^{-1}( \M_{T \setminus \{ v_{i_m} \}} (\Psi( H, \epsilon)))$. Thus, by induction hypothesis, we have a matching on the entire chromatic complex of the tree with $n$  edges, with only two critical states which are supported at the homological grading $0$ and the same enhancement as described before.
        \end{enumerate}

        Before we show that the matching $\M$ defined above is indeed acyclic, we make some constructions.  Recall that the edge set $E$ is already equipped with an order $\mathcal{O} = \{ e_1< e_2< \cdots < e_{n}\}$. We define a new ordering $<_{\M}$ on $E$ induced by the inductive construction of the matching $\M$. We say, \[ e<_{\M} f \] if the edge $e$ appears before the edge $f$ in the inductive definition of $\M$. More precisely if \[ E = \{ f_{1} <_{\M} f_2 <_{\M} \cdots <_{\M} f_{n} \} \] then $f_1$ is the least ordered leaf edge in $T$. Note that $T- {f_1}= T_{n-1} \sqcup w_{1} $, where $w_1$ is the leaf incident to $f_1$ and $T_{n-1}$ is a tree on $(n-1)$ edges. Next, $f_{2}$ is the least ordered leaf edge in $T_{n-1}$ and $T- \{ f_1, f_2 \}= T_{n-2} \sqcup \{ w_1,w_2 \} $ where $w_2$ is the leaf incident to $f_2$ in $T_{n-1}$ and $T_{n-2}$ is a tree on $(n-2)$ edges. Recursively, $f_{i}$ is the least ordered leaf edge in $T_{n-i+1}$ and $T- \{ f_1, \cdots, f_{i} \}= T_{n-i} \sqcup \{w_1, \cdots, w_{i} \}$ where $w_{i}$ is the leaf incident to $f_{i}$ in $T_{n-i+1}$ and $T_{n-i}$ is a tree on $(n-i)$ edges. \\
        Suppose $\M$ matches an enhanced subgraph $(H, \epsilon)$ to the enhanced subgraph $(H- f_{i}, \epsilon_{i})$ if and only if
        \begin{itemize}
            \item \textbf{Condition $1$:} \label{condition1} $f_{i} \in H$ is the least ordered edge in $H$ with respect to $<_{M}$,
            \item \textbf{Condition $2$:} \label{condition2} $\epsilon(Z_{w_{i-1}})= \epsilon(Z_{w_{i-2}}) = \cdots = \epsilon(Z_{w_{1}})=x $,
        \end{itemize}

        where $Z_{w_{k}}$ denotes the connected component consisting of the isolated vertex $\{w_{k} \}$ in $H$ for all $1 \leq k \leq i-1$. If the conditions are satisfied, then the matching $\M$ is defined by the scheme in (\autoref{neg-cross}). Now we are ready to state our next lemma.
        \begin{lemma}\label{acyclicity}
            The matching $\M_T$ on the chromatic complex of $T$ is acyclic.
        \end{lemma}

        \begin{proof}
            Assume for the sake of contradiction that there is a directed cycle of enhanced spanning subgraphs
            \[ (H_1,\epsilon_1) \to d_p^{e_{i_1}}((H_1,\epsilon_1)) \to (H_2,\epsilon_2) \to d_p^{e_{i_2}}((H_2,\epsilon_2)) \to (H_3,\epsilon_3) \to \cdots \to d_p^{e_{i_{k}}}((H_k,\epsilon_k)) \to (H_1,\epsilon_1) \]
            in the weighted directed graph corresponding to the chromatic complex of \( T \), with matching \( \M_T \). Let \( p \) be the homological grading of $(H_i, \epsilon_i )$, and $  d_p^{e_{i_j}} (H_j,\epsilon_j) $ is obtained by taking the differential with respect to the edge \( e_{i_{j}} \), see \autoref{dicycle}.

            \begin{figure}[ht]
                \centering
                \begin{tikzpicture}[x=1cm,y=1cm,
                    st/.style={font=\small},
                    dif/.style={-{Stealth[length=2mm]},thin},
                    mat/.style={-{Stealth[length=2mm]},line width=1.3pt,dashed}]
                    \node[st] (h1) at (0,0) {$(H_1,\epsilon_1)$};
                    \node[st] (h2) at (3,0) {$(H_2,\epsilon_2)$};
                    \node[st] (h3) at (6,0) {$(H_3,\epsilon_3)$};
                    \node[st] (hk) at (10,0) {$(H_k,\epsilon_k)$};
                    \node[st] (k1) at (1.5,1.6) {$d_p^{e_{i_1}}((H_1,\epsilon_1))$};
                    \node[st] (k2) at (4.5,1.6) {$d_p^{e_{i_2}}((H_2,\epsilon_2))$};
                    \node[st] (k3) at (7.5,1.6) {$\cdots$};
                    \node[st] (kk) at (11.2,1.6) {$d_p^{e_{i_k}}((H_k,\epsilon_k))$};
                    \node at (8.3,0) {$\cdots$};
                    \draw[dif] (h1)--(k1); \draw[mat] (k1)--(h2);
                    \draw[dif] (h2)--(k2); \draw[mat] (k2)--(h3);
                    \draw[dif] (h3)--(k3);
                    \draw[dif] (hk)--(kk);
                    \draw[mat] (kk.north) .. controls (9,3.2) and (-1,3.2) .. (h1.north west);
                \end{tikzpicture}
                \caption{A directed cycle in $\G^{\M_T}(\chCom{}{2}{T})$. Thin solid arrows denote differentials; thick dashed arrows denote reversed matching edges.}\label{dicycle}
            \end{figure}
            % Point 27 and 28 addressed.
            Define a distance function.
            \[
            \operatorname{dist}: \{v_{root}\} \times V(T)  \rightarrow \mathbb{N} \cup \{0\}, \quad
            \operatorname{dist}(v_{root}, v)  := \text{distance of the vertex } v \text{ from the root in the graph } T.
            \]
            Also, define
            \begin{equation}{\label{mindist}}
                \operatorname{dist}(v_{root},S) := \min\{ \operatorname{dist}(v_{root},v) \mid v \in S\}
            \end{equation}

            Let \( Z_j \) and \( Z'_j \) be the connected components of \( H_j \) containing the two endpoints of the edge \( e_{i_{j}} \), labeled so that \( \operatorname{dist}(v_{root},Z_j) < \operatorname{dist}(v_{root},Z'_j) \). Let \( f_{i_{j}} \) be the unique edge in \( E(H_j \cup e_{i_j}) \setminus E(H_{j+1}) \) (the index is taken modulo $k$). In other words, in \autoref{dicycle} the matching edge that follows the differential along $e_{i_j}$ deletes the edge $f_{i_j}$. After re-indexing, without loss of generality we may assume $f_{i_1}$ is the minimal element of the set $\{f_{i_j}: 1 \leq j \leq k \}$ with respect to $<_{\M}$.

            \begin{enumerate}
                \item \textbf{Case 1:} Assume \( e_{i_1} \neq f_{i_1} \). Then observe that $f_{i_1} \in H_1$. But, the matching due to the edge $f_{i_k}$ matches $d_{p}^{e_{i_{k}}}(H_k)$ with $H_1$ by removing the edge $f_{i_{k}}$. Hence, $f_{i_1} \in d_{p}^{e_{i_{k}}}(H_k)$ and $f_{i_k} <_{\M} f_{i_1}$. Since in the inductive construction of $\M$, the matching occurs along the edge that is minimal with respect to the matching order $<_{\M}$. It is a contradiction to the hypothesis that $f_{i_1}$ is the minimal element.

                \item \textbf{Case 2:} Assume \( e_{i_1} = f_{i_1} \). Then, we must have \( \epsilon_{1}(Z_1) = 1, \epsilon_{1}(Z'_1) = x \) where $Z_1$ is the connected component in $H_1$ which is closer to the root vertex $v_{root}$ that is in other words, \[ \operatorname{dist}(v_{root}, H_1) < \operatorname{dist}(v_{root}, H_2).  \]

                This follows from the construction of $\M$ and the hypothesis $e_{i_{1}}= f_{i_1}$. Then, \( H_{1} \) and \( H_2 \) are the same as spanning subgraphs, and \( \epsilon_{2}(Z_{1})=x, \epsilon_{2}(Z'_{1})= 1 \). $\epsilon_1$ and $\epsilon_2$ agree on the other connected components. Define
                \[
                D_{H_j} = \sum_{\{Z \in \pi_0(H_j) \mid \epsilon(Z)=x\}} \operatorname{dist}(v_{root},Z).
                \]
                % Point 29,30,31,32 addressed.
                Then, in this case, we have \( D_{H_{2}} < D_{H_1} \). Also, notice that the function \( D_{H_j} \) is non-increasing in \( j \). This implies,
                \[
                D_{H_1} > D_{H_2} \geq D_{H_3} > \cdots \geq D_{H_i} > D_{H_{i+1}} \geq \cdots \geq D_{H_k} \geq D_{H_1},
                \]
                which is a contradiction. Thus, such a cycle cannot exist, and the matching is acyclic.
            \end{enumerate}
        \end{proof}

        We conclude this section with an observation on $\M$. Suppose we are given any two enhanced spanning subgraphs $(H, \epsilon)$ and $(H', \epsilon')$ of $T$ with $i ((H, \epsilon))= i((H', \epsilon'))$.  We use the notation

        \[\mathcal{P}_2^{\downarrow}((H, \epsilon), (H', \epsilon'))\] to denote the collection of alternating paths between $(H_1, \epsilon_1)$ and $(H_2, \epsilon_2)$ in the directed graph $\G^{\M}(\chCom{\ast, \ast}{2}{T})$ that starts with a differential.
        \begin{lemma}[Fundamental lemma]\label{fundamental}
            $  |\mathcal{P}_2^{\downarrow}((H, \epsilon), (H', \epsilon'))| \in \{ 0,1\}.$
        \end{lemma}
        \begin{proof}
            Suppose, the statement in the lemma is false. Then, there exist $(H, \epsilon)$ and $(H', \epsilon')$ violating the statement. We can choose them so that, among all pairs $(H, \epsilon)$ and $(H', \epsilon')$ such that $|\mathcal{P}_{2}^{\downarrow}((H, \epsilon), (H', \epsilon'))|>1$, the quantity \[ \min\left\{ l(P_1)+l(P_2) \mid P_1 \neq P_2 \text{ and } P_1, P_2 \in \mathcal{P}_{2}^{\downarrow}((H, \epsilon), (H', \epsilon')) \right\} \] is minimal, where $l$ denotes the length of a path. % Point 33 addressed.
            Also note that, for any such path(if it exists), the length of the path $l(P)$ must be even. Since by hypothesis, $i((H, \epsilon))= i((H', \epsilon'))= p \,  ( \text{say})$ then both $P_1$ and $P_2$ must end with a matching edge in $\M$. In fact, the matching edge must be the same for $P_1$ and $P_2$.
            Assume $P_1$ and $P_2$ look like,
            \begin{align*}
                P_1 &:  (H,\epsilon)   \rightarrow \ldots \rightarrow (H_{\alpha}, \epsilon_{\alpha}) \xrightarrow{d^{e_{\alpha}}_{p} }(K, \eta) \xrightarrow{\M_{f_{i}}}(H', \epsilon') \\
                P_2 &:  (H,\epsilon)   \rightarrow \ldots \rightarrow (H_{\beta}, \epsilon_{\beta}) \xrightarrow{d^{e_{\beta}}_{p} }(K, \eta) \xrightarrow{\M_{f_{i}}}(H', \epsilon'),
            \end{align*}
            where $\M_{f_{i}}$ denotes the matching edge corresponding to $d^{f_{i}}$.
            If both $l(P_1)>2$ and $l(P_2)>2$ we shall show that $(H_{\alpha}, \epsilon_{\alpha})= (H_{\beta}, \epsilon_{\beta})$ contradicting the minimality of the quantity $l(P_1)+ l(P_2).$    \\
            Our strategy for the proof involves finding a characterization of % Point 34 addreseed.
            $(H_{\alpha}, \epsilon_{\alpha})$ or $(H_{\beta}, \epsilon_{\beta})$. % Point 35 addressed.
            Thus, we shall only focus on $P_1$.  \\
            If $e_{\alpha} \neq f_{i}$, then by Condition 1, $ f_{i}$ is the minimal edge in $K=H_{\alpha} \cup e_{\alpha}$ with respect to $<_{\M}$. Here $f_1 <_{\M} \cdots <_{\M} f_{i-1}$ denote the edges preceding $f_i$ in the matching order $<_{\M}$, and $\{w_1, \ldots, w_{i-1}\}$ the corresponding leaves. Since $f_i$ is the minimal edge of $K$ with respect to $<_{\M}$, none of $\{f_1, \ldots, f_{i-1}\}$ belongs to $K$. Also, as $e_{\alpha} \in K$ and $e_{\alpha} \neq f_i$, this implies $e_{\alpha} \notin \{ f_1, f_2, \cdots, f_{i} \}$, and moreover $f_{i} \in H_{\alpha}$. If the portion of $P_1$ extends to the left hand side by a matching edge $\M_{f}$ as follows:
            \begin{align*}
                P_1 &:  (H,\epsilon)   \rightarrow \ldots \rightarrow (K', \eta') \xrightarrow{\M_{f}} (H_{\alpha}, \epsilon_{\alpha}) \xrightarrow{d^{e_{\alpha}}_{p} }(K, \eta) \xrightarrow{\M_{f_{i}}}(H', \epsilon') \\
            \end{align*}
            then the Condition $1$ will imply $f$ is the least ordered element in $K'$ and $f_i \in K'$. Thus, \[ f<_{\M} f_{i} \]

            Thus $f=f_k$ for some $1 \leq k \leq i-1$. By the construction of $\M_T$, the matching edge $\M_f$ (the matching edge corresponding to $d^{f}$) deletes the edge $f=f_k$ from $K'$ and assigns the enhancement $1$ to the isolated vertex $w_k$. In other words, $\epsilon_{\alpha}(Z_{w_{k}})=1$ where $Z_{w_{k}} \in \pi_{0}(H_{\alpha})$ is the component consisting of the isolated vertex $w_{k}$. Note that the differential cannot change the enhancement $\epsilon_{\alpha}(Z_{w_{k}})=1$ since $f <_{\M} f_{i} <_{\M} e_{\alpha}$. Said differently,
            \[ \eta(Z_{w_{k}})=1. \]
            But this essentially violates Condition $2$ for $\M_{f_{i}}$. In other words, for $(K',\eta')$ to exist in $P_1$, we must have $e_{\alpha} = f_{i}$.  If $e_{\alpha} = f_{i}$, then $d^{e_{\alpha}}_{p}$ is uniquely determined as follows. Note that $H_{\alpha}= H'$ as a spanning subgraph. If $f_i=\{w_i, w'_{i} \}$ and $Z(w'_{i}) \in \pi_{0}(H_{\alpha})$ denotes the connected component containing $w'_{i}$, then we have
            \begin{align*}
                \epsilon_{\alpha}(Z_{w_i}) &= x \text{ and} \, \,  \epsilon_{\alpha}(Z(w'_i)) =1 \\
                \epsilon'(Z_{w_i}) &= 1 \text{ and} \, \,  \epsilon'(Z(w'_i)) =x \: \\
                \epsilon_{\alpha}(Z) &= \epsilon'(Z),  \, \,  \text{ otherwise.}
            \end{align*}
            Similarly, for the path $P_2$ by applying the same arguments, we obtain $e_{\beta}=f_i$, which essentially implies that $(H_{\alpha},\epsilon_{\alpha})=(H_{\beta},\epsilon_{\beta})$. This essentially violates the minimality of the quantity $l(P_1)+l(P_2)$, for two distinct paths $P_1$ and $P_2$, with $l(P_1)>2$ and $l(P_2)>2$. Thus, we may assume one of $P_{i}$ must have length $2$. Moreover, two distinct paths of length $2$ cannot exist. This is because a path of length $2$ in $\mathcal{P}^{\downarrow}_2((H, \epsilon), (H', \epsilon'))$ consists of the differential from $(H,\epsilon)$ to the state matched with $(H',\epsilon')$, followed by the reversed matching edge, and is therefore determined by its end points. Since the lengths of the paths are even, without loss of generality, we may assume that $l(P_1)=2$ and $l(P_2) \geq 4$.

            Thus, $P_1$ and $P_2$ look like

            \begin{align*}
                P_1 &:  (H,\epsilon)    \xrightarrow{d^{e}_{p}} (K, \eta) \xrightarrow{\M_{f_{i}}}(H', \epsilon') \\
                P_2 &:  (H,\epsilon)\xrightarrow{d^{e'}_{p}} (K_1, \eta_1)\xrightarrow{\M_{f}} (H_1, \epsilon_1)\xrightarrow{d^{f_{i_1}}_{p}} (K_2, \eta_2)\xrightarrow{\M_{f_{i_1}}} (H_2, \epsilon_2) \rightarrow \ldots \ldots \rightarrow \\ &(H_{t}, \epsilon_{t}) \xrightarrow{d^{{f_{i_t}}}_{p} }(K, \eta) \xrightarrow{\M_{f_{i_t}}= \M_{f_{i}}}(H', \epsilon').
            \end{align*}
            Now if $e \neq f_i$ then $e \in H'$ and $e \notin H$.
            % Points 33-39 addressed.
            As spanning subgraphs, we have the following equalities \[H_1= H_2= \ldots = H_t=H'.\]
            Thus, $e \in H_1$, which in turn implies $e \in K_1$. But we already know $e \notin H$, implying $e'=e$.  In other words, $(K, \eta)= (K_1, \eta_1)$ and $\M_{f_i}= \M_{f}$. Restricting the path $P_2$ from $(H_1, \epsilon_1)=(H',\epsilon')$ yields a closed directed cycle which contradicts \Autoref{acyclicity}. \\
            We can now assume $ e = f_{i}$. By virtue of the fact that $H_1= H'=H$, we have $e'=f$. If $e' \neq f$, then the edge $e'$ will stay inside $H_1$. In other words, $e' \in H$. But this will contradict the fact that $d^{e'}_{p}$ is a differential.
            Now we focus on the last two edges of $P_2$. We have $f_{i_t}= f_{i}= e$. We have shown in the part of the proof that the pair $(H_t, \epsilon_t)$ is uniquely determined from the fact that the differential edge is equal to the matching edge. In other words, we must have
            \[ (H_t, \epsilon_t)= (H, \epsilon).  \]
            Ignoring the last two edges of $P_2$ yields a closed directed cycle, which again contradicts \Autoref{acyclicity}.
        \end{proof}

        % %%%%%%%%%%%%%%%%%%%%%%%%%%%%%%%%%%%%%%%%%%%%%%%%%%%%%%%%%%%%%%%%%%%%%%%%%%%%%%%%%%%%%%%%%%%%%%%%%%%%
        % %%%%%%%%%%%%%%%%%%%%%%%%%%%%%%%%%%%%%%%%%%%%%%%%%%%%%%%%%%%%%%%%%%%%%%%%%%%%%%%%%%%%%%%%%%%%%%%%%%%%%
        % %%%%%%%%% ACYCLIC MATCHING ON THE CHROMATIC COMPLEX OF A GRAPH G %%%%%%%%%%%%%%%%%%%%%%%%%%%%%%%%%%%%
        % %%%%%%%%%%%%%%%%%%%%%%%%%%%%%%%%%%%%%%%%%%%%%%%%%%%%%%%%%%%%%%%%%%%%%%%%%%%%%%%%%%%%%%%%%%%%%%%%%%%%%%
        % %%%%%%%%%%%%%%%%%%%%%%%%%%%%%%%%%%%%%%%%%%%%%%%%%%%%%%%%%%%%%%%%%%%%%%%%%%%%%%%%%%%%%%%%%%%%%%%%%%%%%%%%
        \subsection{Acyclic matching on the chromatic complex of a graph G} \label{subsec: matching on G for A_2}
        \begin{definition}
            For every NBC spanning tree $T$ of $G$, we introduce two formal symbols $T^{+}$ and $T^{-}$, which we call the \emph{signed NBC spanning trees} associated with $T$.
        \end{definition}
        $T^+$ and $T^-$ correspond to the two critical states $H_c^{T^{+}}$ and $H_c^{T^{-}}$ of the interval $[\R(T),T]$ constructed in \autoref{thm: Matching on G for A_2} below: $T^{+}$ corresponds to the critical state whose underlying spanning subgraph is $\IN(T)$ and whose enhancement is $1$ on the component containing the root and $x$ on all other components, while $T^{-}$ corresponds to the critical state with the same underlying spanning subgraph and enhancement $x$ on all components. Let us denote the collection of signed NBC spanning trees of $G$ by $SST(G)$.
        \[SST(G):=\{T^{\pm} \colon T \text{ is an} \nbc \text{ spanning tree of }G\}\]
        % Point 40 addressed.
        In this section, $\cstBiCom{*}{*}{2}{G}$ will denote the chain complex generated by signed NBC trees. In other words,  \[ \cstBiCom{*}{*}{2}{G} := \bigoplus_{ \nbc \text{tree }T} \mathbb{Z} \langle T^{+}, T^{-} \rangle \]
        Next, we shall define an acyclic matching $\M_{\nbc}$ on the NBC complex of $G$. Our strategy shall involve extending the acyclic matching $\M_{T}$ for a spanning tree $T$ of $G$. The critical states of $\M_{\nbc}$ are in one-to-one correspondence with signed NBC spanning trees of $G$, namely $T^{+}$ or $T^{-}$. Again, we fix a vertex $v_{root}$ of $G$ which will play the role of the root for the NBC spanning trees. Let $\{T_1 < T_2< \cdots <T_p\}$ be the lexicographic ordering on NBC spanning trees of $G$.
        We construct the matching $\M_{\nbc}$ on the NBC complex of $G$ recursively.

        Recall from \Autoref{partition} that, for a shellable complex $\Delta$ with facets $\{F_1,F_2, \cdots F_p\}$, the intervals $[\R(F_i),F_i]$ partition $\Delta$. It follows immediately from \Autoref{shelling order} that for any graph $G$, $\NBC(G)$ is a shellable complex.Furthermore, the intervals $[\R(T_i),T_i]$ form a partition of the complex $\NBC(G)$, where $T_i$'s are the NBC spanning trees of $G$.

        \begin{theorem}\label{thm: Matching on G for A_2}
            Let $G$ be a connected graph with a fixed ordering on its edges and a marked vertex $v_{root}$. Then there exists an acyclic matching $\M_{\nbc}$ on the weighted directed graph
            \[
            \G\left(\nbcCom{2}{G}, F(\NBC(G)\right),
            \]
            where $F(\NBC(G))$ denotes the collection of facets of the simplicial complex $\NBC(G)$. Moreover, the critical states of $\M_{\nbc}$ are in one-to-one correspondence with elements of $SST(G)$.
        \end{theorem}

        \begin{figure}[!htb]
            \begin{center} \scalebox{0.75}{\includegraphics{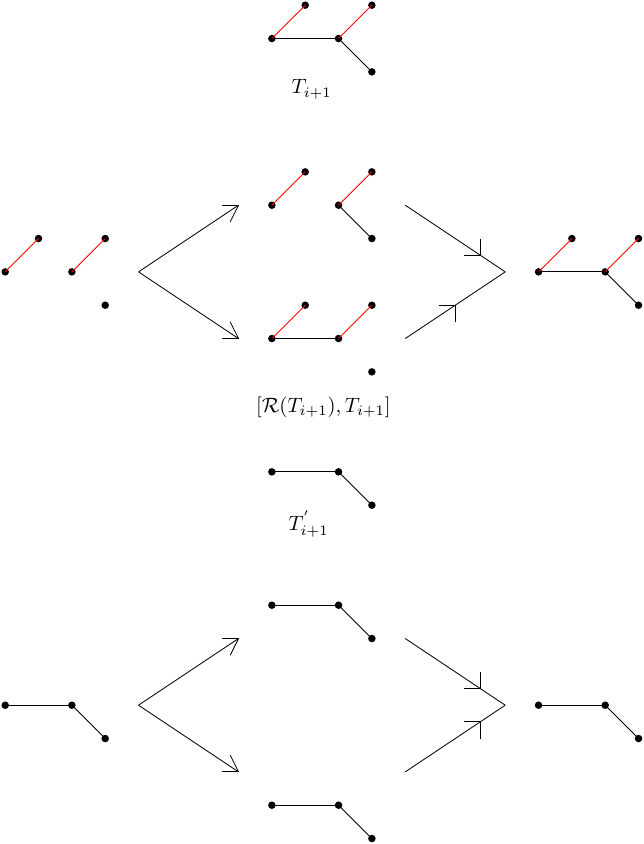}}
                \vspace{1cm}\caption{The edges in $\IN(T_{i+1})$ are colored red}
                \label{contraction}
            \end{center}
        \end{figure}
        \begin{proof}
            Let us assume $T_1< T_2 < \ldots < T_{p}$ denotes the shelling on the NBC spanning trees of the graph $G$. Our strategy to construct the acyclic matching on the NBC complex is to provide an acyclic matching in each of the intervals $[\R(T_{i}), T_{i}]$ separately which we denote by $\M_i$. Each interval $[\R(T_{i}), T_{i}]$ will contain exactly two critical states. They shall be denoted by $T_{i}^{+}$ and $T_{i}^{-}$. We shall provide an explicit description of $T_{i}^{+}$ and $T_{i}^{-}$. We proceed by induction on the indexing set $i$ of the NBC spanning trees.
            \begin{enumerate}
                \item $\R(T_1)= \emptyset$. By \Autoref{shelling order} the interval $[\R(T_{1}), T_{1}]$ can be identified with the complex generated by the facet $T_{1}$. We provide the matching $\M_1$ as described in Subsection \ref{Matching on tree}.
                Thus, $[\R(T_{1}), T_{1}]$ contains precisely two critical states in homological grading $0$. $H_c^{T_{1}^{+}}$ and $H_c^{T_{1}^{-}}$ can be described as follows:
                \begin{align*}
                    H_c^{T_{1}^{+}}  &:= \left( H, \epsilon_{+} \right) & \qquad  H_c^{T_{1}^{-}}  &:= \left( H, \epsilon_{-} \right)  \\
                    \epsilon_{+}(v) &= x \, \, \text{if}  \, \, v \neq v_{root} & \text{and} \qquad \epsilon_{-}(v) &= x \, \, \text{for all vertex} \, \, v \\
                    \epsilon_{+}(v_{root}) &= 1
                \end{align*}

                where $H$ denotes the spanning subgraph with no edges. This follows from the fact that $\R(T_1)=\emptyset$ % Point 42 adressed
                and combining \Autoref{inactiveedgelemma} and \Autoref{First Tree}. This concludes our base case.
                \item Inductively, assume we have constructed $\M_{\nbc}=\sqcup_{j=1}^{i} \M_j$ on $\Delta_i$ where $M_j$ is the matching on the interval $[\R(T_{j}), T_{j}]$ and $T_{j}^{+}$, $T_{j}^{-}$ denotes the two critical states for all $1 \leq j \leq i$ in $[\R(T_{j}), T_{j}]$. We extend $\M_{\nbc}$ on the interval $[\R(T_{i+1}), T_{i+1}]$ and hence on $\Delta_{i+1}$.

                We write the edge set of $T_{i+1}$ as disjoint union $E(T_{i+1}) = \IN(T_{i+1}) \sqcup \IA(T_{i+1})$ where $\IN(T_{i+1})$ and $\IA(T_{i+1})$ denotes the collection of inactive and active edges of $T_{i+1}$ respectively. Note that, by virtue of \Autoref{inactiveedgelemma} we have the following:

                \[\R(T_{i+1})= \IN(T_{i+1}).\]

                Let $T_{i+1}^{'}$ denote the tree which is obtained from $T_{i+1}$ by contracting the inactive edges $\IN(T_{i+1})$. Then the formal subcomplex generated by the enhanced spanning subgraphs in $[\R(T_{i+1}), T_{i+1}]$ can be canonically identified with the formal chromatic complex of the tree $T_{i+1}^{'}$. For example, see \autoref{contraction}. With this identification, we extend the matching $\M_{\nbc}$ on $[\R(T_{i+1}), T_{i+1}]$ as described in \Cref{Matching on tree}. Next, we explicitly describe the two critical states in $[\R(T_{i+1}), T_{i+1}]$. Let $H^{T_{i+1}}$ denote the spanning subgraph with $E(H_{T_{i+1}})= \IN(T_{i+1})$ and $Z_{v_{root}}$ denote the connected component of $H^{T_{i+1}}$ that contains the root vertex $v_{root}$. $H_c^{T_{i+1}^{+}}$ and $H_c^{T_{i+1}^{-}}$ can be described as follows:
                \begin{align*}
                    H_c^{T_{i+1}^{+}} & := \left( H^{T_{i+1}}, \epsilon_{+} \right) & \qquad H_c^{T_{i+1}^{-}} & := \left( H^{T_{i+1}}, \epsilon_{-} \right) \\
                    \epsilon_{+}(Z) &= x \, \, \text{if} \, \, v_{root} \notin Z & \text{and} \qquad \epsilon_{-}(Z) &= x \, \, \text{for all components} \, \, Z \, \, \text{of} \, \,  H_{T_{i+1}} \\
                    \epsilon_{+}(Z_{v_{root}}) &= 1
                \end{align*}
                It is easy to see that $\M_{i+1}$ contains only the edges in  the interval $[\IN(T_{i+1}),T_{i+1}]$ and again by \Autoref{partition}, we have that $\Delta_i \cap [\IN(T_{i+1}),T_{i+1}]= \emptyset$ and hence $\M_{\nbc}=\cup_{j=1}^{i+1}\M_{j}$ is indeed a matching. This concludes the inductive step.
            \end{enumerate}
            Moreover, $\M_{\nbc}:=\bigcup_{i=1}^{p}\M_i$, where $\M_i$ denotes the matching constructed on $[\R(T_i),T_i]$, is indeed a matching on $\G(\nbcCom{2}{G})$ and its critical states are precisely the states $T_i^{\pm}$, $1\le i\le p$, of the individual matchings $\M_i$. It remains to show that $\M_{\nbc}$ is acyclic.

            The acyclicity of the matching $\M_{\nbc}$ follows from \Autoref{acyclicity} and \Autoref{crucial} and the fact that the matching $\M_{\nbc}$ is defined on each disjoint interval $[\R(T_i), T_i]$ inductively. In fact any enhanced spanning subgraph $(H,\epsilon)$ lies in a unique interval $[\IN(T_i),T_i]$ for some $i \in \{1,2,\dots p\}$. So if $i \neq j$ and we move through a differential from $[\R(T_i),T_i]$ to $[\R(T_j),T_j]$, then we
            must have $j \geq i$ due to \Autoref{crucial}. Moreover, the matching edge in $\M_{\nbc}$ cannot leave the interval $[\R(T_i), T_{i}]$. So any directed cycle in $\nbcCom{2}{G}$ must be contained in a single interval $[\IN(T_i),T_i]$ and thus contradicting \Autoref{acyclicity}. Thus, $\M_{\nbc}$ provides an acyclic matching on $\nbcCom{2}{G}$.
        \end{proof}
        Observe that, the critical states obtained for each interval $[\R(T_i),T_i]=[\IN(T_i),T_i]$ due to the matching $\M_{\NBC}$ have homological grading $i(H_c^{{T_i}^+})=i(H_c^{{T_i}^-})=\lvert \IN(T_i) \rvert$. Moreover, the quantum gradings are $j(H_c^{{T_i}^-})=\lvert \IA(T_i)\rvert+1$ and
        $j(H_c^{{T_i}^+})=\lvert \IA(T_i)\rvert$ which denotes the number of connected components in $\IN(T_i)$ with the enhancement $x$.

        \begin{figure}
            \resizebox{0.64\textwidth}{!}{
            \begin{tikzpicture}[vertex/.style={draw, fill=black,minimum size=10pt, inner sep=0pt,circle}, root/.style={draw, fill=red,color=red,minimum size=10pt, inner sep=0pt,circle}]
                \node [style=root] (0) at (-10, 12.25) {};
                \node [style=vertex] (1) at (-11.5, 10.25) {};
                \node [style=vertex] (2) at (-8.5, 10.25) {};
                \node [] (3) at (-10, 13) {\scalebox{1.7}{$1$}};
                \node [] (5) at (-11.5, 11) {\scalebox{1.7}{$1$}};
                \node [] (6) at (-8.5, 11) {\scalebox{1.7}{$1$}};
                \node [] (7) at (-12, 13.5) {};
                \node [] (8) at (-8, 13.5) {};
                \node [] (9) at (-12, 9.75) {};
                \node [] (10) at (-8, 9.75) {};
                \node [style=root] (11) at (-10, 6.25) {};
                \node [style=vertex] (12) at (-11.5, 4.25) {};
                \node [style=vertex] (13) at (-8.5, 4.25) {};
                \node [] (14) at (-10, 7) {\scalebox{1.7}{$1$}};
                \node [] (15) at (-11.5, 5) {\scalebox{1.7}{$1$}};
                \node [] (16) at (-8.5, 5) {\scalebox{1.7}{$x$}};
                \node [] (17) at (-12, 7.5) {};
                \node [] (18) at (-8, 7.5) {};
                \node [] (19) at (-12, 3.75) {};
                \node [] (20) at (-8, 3.75) {};
                \node [style=root] (21) at (-10, 0.25) {};
                \node [style=vertex] (22) at (-11.5, -1.75) {};
                \node [style=vertex] (23) at (-8.5, -1.75) {};
                \node [] (24) at (-10, 1) {\scalebox{1.7}{$1$}};
                \node [] (25) at (-11.5, -1) {\scalebox{1.7}{$x$}};
                \node [] (26) at (-8.5, -1) {\scalebox{1.7}{$1$}};
                \node [] (27) at (-12, 1.5) {};
                \node [] (28) at (-8, 1.5) {};
                \node [] (29) at (-12, -2.25) {};
                \node [] (30) at (-8, -2.25) {};
                \node [style=root] (31) at (-10, -5.75) {};
                \node [style=vertex] (32) at (-11.5, -7.75) {};
                \node [style=vertex] (33) at (-8.5, -7.75) {};
                \node [] (34) at (-10, -5) {\scalebox{1.7}{$1$}};
                \node [] (35) at (-11.5, -7) {\scalebox{1.7}{$x$}};
                \node [] (36) at (-8.5, -7) {\scalebox{1.7}{$x$}};
                \node [] (37) at (-12, -4.5) {};
                \node [] (38) at (-8, -4.5) {};
                \node [] (39) at (-12, -8.25) {};
                \node [] (40) at (-8, -8.25) {};
                \node [style=root] (41) at (-10, -11.75) {};
                \node [style=vertex] (42) at (-11.5, -13.75) {};
                \node [style=vertex] (43) at (-8.5, -13.75) {};
                \node [] (44) at (-10, -11) {\scalebox{1.7}{$x$}};
                \node [] (45) at (-11.5, -13) {\scalebox{1.7}{$1$}};
                \node [] (46) at (-8.5, -13) {\scalebox{1.7}{$1$}};
                \node [] (47) at (-12, -10.5) {};
                \node [] (48) at (-8, -10.5) {};
                \node [] (49) at (-12, -14.25) {};
                \node [] (50) at (-8, -14.25) {};
                \node [style=root] (51) at (-10, -17.75) {};
                \node [style=vertex] (52) at (-11.5, -19.75) {};
                \node [style=vertex] (53) at (-8.5, -19.75) {};
                \node [] (54) at (-10, -17) {\scalebox{1.7}{$x$}};
                \node [] (55) at (-11.5, -19) {\scalebox{1.7}{$1$}};
                \node [] (56) at (-8.5, -19) {\scalebox{1.7}{$x$}};
                \node [] (57) at (-12, -16.5) {};
                \node [] (58) at (-8, -16.5) {};
                \node [] (59) at (-12, -20.25) {};
                \node [] (60) at (-8, -20.25) {};
                \node [style=root] (61) at (-10, -23.75) {};
                \node [style=vertex] (62) at (-11.5, -25.75) {};
                \node [style=vertex] (63) at (-8.5, -25.75) {};
                \node [] (64) at (-10, -23) {\scalebox{1.7}{$x$}};
                \node [] (65) at (-11.5, -25) {\scalebox{1.7}{$x$}};
                \node [] (66) at (-8.5, -25) {\scalebox{1.7}{$1$}};
                \node [] (67) at (-12, -22.5) {};
                \node [] (68) at (-8, -22.5) {};
                \node [] (69) at (-12, -26.25) {};
                \node [] (70) at (-8, -26.25) {};
                \node [style=root] (71) at (-10, -29.75) {};
                \node [style=vertex] (72) at (-11.5, -31.75) {};
                \node [style=vertex] (73) at (-8.5, -31.75) {};
                \node [] (74) at (-10, -29) {\scalebox{1.7}{$x$}};
                \node [] (75) at (-11.5, -31) {\scalebox{1.7}{$x$}};
                \node [] (76) at (-8.5, -31) {\scalebox{1.7}{$x$}};
                \node [] (77) at (-12, -28.5) {};
                \node [] (78) at (-8, -28.5) {};
                \node [] (79) at (-12, -32.25) {};
                \node [] (80) at (-8, -32.25) {};
                \node [style=root] (81) at (0, 16) {};
                \node [style=vertex] (82) at (-1.5, 14) {};
                \node [style=vertex] (83) at (1.5, 14) {};
                \node [] (85) at (-1.25, 15.25) {\scalebox{1.7}{$1$}};
                \node [] (86) at (0.75, 14) {\scalebox{1.7}{$1$}};
                \node [] (87) at (-2, 16.5) {};
                \node [] (88) at (2, 16.5) {};
                \node [] (89) at (-2, 13.25) {};
                \node [] (90) at (2, 13.25) {};
                \node [style=root] (91) at (0, 11.5) {};
                \node [style=vertex] (92) at (-1.5, 9.5) {};
                \node [style=vertex] (93) at (1.5, 9.5) {};
                \node [] (94) at (-1.25, 10.75) {\scalebox{1.7}{$1$}};
                \node [] (95) at (0.75, 9.5) {\scalebox{1.7}{$x$}};
                \node [] (96) at (-2, 12) {};
                \node [] (97) at (2, 12) {};
                \node [] (98) at (-2, 8.75) {};
                \node [] (99) at (2, 8.75) {};
                \node [style=root] (100) at (0, 7) {};
                \node [style=vertex] (101) at (-1.5, 5) {};
                \node [style=vertex] (102) at (1.5, 5) {};
                \node [] (103) at (0.75, 7) {\scalebox{1.7}{$1$}};
                \node [] (104) at (0, 5.5) {\scalebox{1.7}{$1$}};
                \node [] (105) at (-2, 7.5) {};
                \node [] (106) at (2, 7.5) {};
                \node [] (107) at (-2, 4.25) {};
                \node [] (108) at (2, 4.25) {};
                \node [style=root] (109) at (0, 2.5) {};
                \node [style=vertex] (110) at (-1.5, 0.5) {};
                \node [style=vertex] (111) at (1.5, 0.5) {};
                \node [] (112) at (0.75, 2.5) {\scalebox{1.7}{$1$}};
                \node [] (113) at (0, 1) {\scalebox{1.7}{$x$}};
                \node [] (114) at (-2, 3) {};
                \node [] (115) at (2, 3) {};
                \node [] (116) at (-2, -0.25) {};
                \node [] (117) at (2, -0.25) {};
                \node [style=root] (118) at (0, -2) {};
                \node [style=vertex] (119) at (-1.5, -4) {};
                \node [style=vertex] (120) at (1.5, -4) {};
                \node [] (121) at (-1.25, -2.75) {\scalebox{1.7}{$x$}};
                \node [] (122) at (0.75, -4) {\scalebox{1.7}{$1$}};
                \node [] (123) at (-2, -1.5) {};
                \node [] (124) at (2, -1.5) {};
                \node [] (125) at (-2, -4.5) {};
                \node [] (126) at (2, -4.5) {};
                \node [style=root] (127) at (0, -6.5) {};
                \node [style=vertex] (128) at (-1.5, -8.5) {};
                \node [style=vertex] (129) at (1.5, -8.5) {};
                \node [] (130) at (-1.25, -7.25) {\scalebox{1.7}{$x$}};
                \node [] (131) at (0.75, -8.5) {\scalebox{1.7}{$x$}};
                \node [] (132) at (-2, -6) {};
                \node [] (133) at (2, -6) {};
                \node [] (134) at (-2, -9.25) {};
                \node [] (135) at (2, -9.25) {};
                \node [style=root] (136) at (0, -11) {};
                \node [style=vertex] (137) at (-1.5, -13) {};
                \node [style=vertex] (138) at (1.5, -13) {};
                \node [] (139) at (0.75, -11) {\scalebox{1.7}{$x$}};
                \node [] (140) at (0, -12.5) {\scalebox{1.7}{$1$}};
                \node [] (141) at (-2, -10.5) {};
                \node [] (142) at (2, -10.5) {};
                \node [] (143) at (-2, -13.75) {};
                \node [] (144) at (2, -13.75) {};
                \node [style=root] (145) at (0, -15.5) {};
                \node [style=vertex] (146) at (-1.5, -17.5) {};
                \node [style=vertex] (147) at (1.5, -17.5) {};
                \node [] (148) at (0.75, -15.5) {\scalebox{1.7}{$x$}};
                \node [] (149) at (0, -17) {\scalebox{1.7}{$x$}};
                \node [] (150) at (-2, -15) {};
                \node [] (151) at (2, -15) {};
                \node [] (152) at (-2, -18.25) {};
                \node [] (153) at (2, -18.25) {};
                \node [style=root] (154) at (0, -21.5) {};
                \node [style=vertex] (155) at (-1.5, -23.5) {};
                \node [style=vertex] (156) at (1.5, -23.5) {};
                \node [] (157) at (1.25, -22.25) {\scalebox{1.7}{$1$}};
                \node [] (158) at (-0.75, -23.5) {\scalebox{1.7}{$1$}};
                \node [] (159) at (-2, -21) {};
                \node [] (160) at (2, -21) {};
                \node [] (161) at (-2, -24.25) {};
                \node [] (162) at (2, -24.25) {};
                \node [style=root] (163) at (0, -26) {};
                \node [style=vertex] (164) at (-1.5, -28) {};
                \node [style=vertex] (165) at (1.5, -28) {};
                \node [] (166) at (1.25, -26.75) {\scalebox{1.7}{$1$}};
                \node [] (167) at (-0.75, -28) {\scalebox{1.7}{$x$}};
                \node [] (168) at (-2, -25.5) {};
                \node [] (169) at (2, -25.5) {};
                \node [] (170) at (-2, -28.75) {};
                \node [] (171) at (2, -28.75) {};
                \node [style=root] (172) at (0, -30.5) {};
                \node [style=vertex] (173) at (-1.5, -32.5) {};
                \node [style=vertex] (174) at (1.5, -32.5) {};
                \node [] (175) at (1.25, -31.25) {\scalebox{1.7}{$x$}};
                \node [] (176) at (-0.75, -32.5) {\scalebox{1.7}{$1$}};
                \node [] (177) at (-2, -30) {};
                \node [] (178) at (2, -30) {};
                \node [] (179) at (-2, -33.25) {};
                \node [] (180) at (2, -33.25) {};
                \node [style=root] (181) at (0, -35) {};
                \node [style=vertex] (182) at (-1.5, -37) {};
                \node [style=vertex] (183) at (1.5, -37) {};
                \node [] (184) at (1.25, -35.75) {\scalebox{1.7}{$x$}};
                \node [] (185) at (-0.75, -37) {\scalebox{1.7}{$x$}};
                \node [] (186) at (-2, -34.5) {};
                \node [] (187) at (2, -34.5) {};
                \node [] (188) at (-2, -37.75) {};
                \node [] (189) at (2, -37.75) {};
                \node [style=root] (190) at (10, 2) {};
                \node [style=vertex] (191) at (8.5, 0) {};
                \node [style=vertex] (192) at (11.5, 0) {};
                \node [] (193) at (9.75, 0.75) {\scalebox{1.7}{$1$}};
                \node [] (195) at (8, 2.5) {};
                \node [] (196) at (12, 2.5) {};
                \node [] (197) at (8, -0.75) {};
                \node [] (198) at (12, -0.75) {};
                \node [style=root] (199) at (10, -4.25) {};
                \node [style=vertex] (200) at (8.5, -6.25) {};
                \node [style=vertex] (201) at (11.5, -6.25) {};
                \node [] (202) at (9.75, -5.5) {\scalebox{1.7}{$x$}};
                \node [] (203) at (8, -3.75) {};
                \node [] (204) at (12, -3.75) {};
                \node [] (205) at (8, -7) {};
                \node [] (206) at (12, -7) {};
                \node [style=root] (207) at (10, -10.5) {};
                \node [style=vertex] (208) at (8.5, -12.5) {};
                \node [style=vertex] (209) at (11.5, -12.5) {};
                \node [] (210) at (10, -11.75) {\scalebox{1.7}{$1$}};
                \node [] (211) at (8, -10) {};
                \node [] (212) at (12, -10) {};
                \node [] (213) at (8, -13.25) {};
                \node [] (214) at (12, -13.25) {};
                \node [style=root] (215) at (10, -16.75) {};
                \node [style=vertex] (216) at (8.5, -18.75) {};
                \node [style=vertex] (217) at (11.5, -18.75) {};
                \node [] (218) at (10, -18) {\scalebox{1.7}{$x$}};
                \node [] (219) at (8, -16.25) {};
                \node [] (220) at (12, -16.25) {};
                \node [] (221) at (8, -19.5) {};
                \node [] (222) at (12, -19.5) {};
                \node [] (223) at (-7.5, 11.25) {};
                \node [] (224) at (-2.5, 6.25) {};
                \node [] (225) at (-7.5, 6) {};
                \node [] (226) at (-2.5, 10.25) {};
                \node [] (229) at (-7.5, -0.5) {};
                \node [] (230) at (-2.5, 1.5) {};
                \node [] (235) at (-7.5, -12.25) {};
                \node [] (236) at (-2.5, -12.25) {};
                \node [] (237) at (-7.5, -23.25) {};
                \node [] (238) at (-2.5, -17.75) {};
                \node [] (239) at (-7.5, -17.25) {};
                \node [] (240) at (-2.5, -8.25) {};
                \node [] (241) at (2.5, -2.25) {};
                \node [] (242) at (7.5, -4.5) {};
                \node [] (243) at (2.5, 13.75) {};
                \node [] (244) at (7.5, 2) {};
                \node [] (245) at (2.75, -21.5) {};
                \node [] (246) at (7.5, -12.5) {};
                \node [] (247) at (2.5, -30.5) {};
                \node [] (248) at (7.5, -20.5) {};
                \node [] (251) at (6, -9) {};
                \node [] (252) at (14, -9) {};
                \node [] (253) at (-3, -20.5) {};
                \node [] (254) at (6, -20.5) {};
                \node [] (255) at (-3, -38.5) {};
                \node [] (256) at (14, -38.5) {};
                \node [] (257) at (-13, 17) {};
                \node [] (258) at (14, 17) {};
                \node [] (259) at (-13, -34) {};
                \node [] (260) at (-5, -34) {};
                \node [] (261) at (-5, -19) {};
                \node [] (262) at (3.75, -19) {};
                \node [] (263) at (3.75, -8) {};
                \node [] (264) at (14, -8) {};
                \node [] (265) at (9, 14.25) {\scalebox{1.8}{$[\R(T_1), T_1]=\Delta(T_1)$}};
                \node [] (266) at (10, -28) {\scalebox{1.8}{$[\R(T_2),T_2]$}};
                \draw [color=gray!50] (7.center) to (8.center);
                \draw [color=gray!50] (7.center) to (9.center);
                \draw [color=gray!50] (9.center) to (10.center);
                \draw [color=gray!50] (10.center) to (8.center);
                \draw [color=gray!50] (17.center) to (18.center);
                \draw [color=gray!50] (17.center) to (19.center);
                \draw [color=gray!50] (19.center) to (20.center);
                \draw [color=gray!50] (20.center) to (18.center);
                \draw [color=gray!50] (27.center) to (28.center);
                \draw [color=gray!50] (27.center) to (29.center);
                \draw [color=gray!50] (29.center) to (30.center);
                \draw [color=gray!50] (30.center) to (28.center);
                \draw [color=gray!50] (37.center) to (38.center);
                \draw [color=gray!50] (37.center) to (39.center);
                \draw [color=gray!50] (39.center) to (40.center);
                \draw [color=gray!50] (40.center) to (38.center);
                \draw [color=gray!50] (47.center) to (48.center);
                \draw [color=gray!50] (47.center) to (49.center);
                \draw [color=gray!50] (49.center) to (50.center);
                \draw [color=gray!50] (50.center) to (48.center);
                \draw [color=gray!50] (57.center) to (58.center);
                \draw [color=gray!50] (57.center) to (59.center);
                \draw [color=gray!50] (59.center) to (60.center);
                \draw [color=gray!50] (60.center) to (58.center);
                \draw [color=gray!50] (67.center) to (68.center);
                \draw [color=gray!50] (67.center) to (69.center);
                \draw [color=gray!50] (69.center) to (70.center);
                \draw [color=gray!50] (70.center) to (68.center);
                \draw [color=gray!50] (77.center) to (78.center);
                \draw [color=gray!50] (77.center) to (79.center);
                \draw [color=gray!50] (79.center) to (80.center);
                \draw [color=gray!50] (80.center) to (78.center);
                \draw [color=gray!50] (87.center) to (88.center);
                \draw [color=gray!50] (87.center) to (89.center);
                \draw [color=gray!50] (89.center) to (90.center);
                \draw [color=gray!50] (90.center) to (88.center);
                \draw (81) to (82);
                \draw [color=gray!50] (96.center) to (97.center);
                \draw [color=gray!50] (96.center) to (98.center);
                \draw [color=gray!50] (98.center) to (99.center);
                \draw [color=gray!50] (99.center) to (97.center);
                \draw (91) to (92);
                \draw [color=gray!50] (105.center) to (106.center);
                \draw [color=gray!50] (105.center) to (107.center);
                \draw [color=gray!50] (107.center) to (108.center);
                \draw [color=gray!50] (108.center) to (106.center);
                \draw (101) to (102);
                \draw [color=gray!50] (114.center) to (115.center);
                \draw [color=gray!50] (114.center) to (116.center);
                \draw [color=gray!50] (116.center) to (117.center);
                \draw [color=gray!50] (117.center) to (115.center);
                \draw (110) to (111);
                \draw [color=gray!50] (123.center) to (124.center);
                \draw [color=gray!50] (123.center) to (125.center);
                \draw [color=gray!50] (125.center) to (126.center);
                \draw [color=gray!50] (126.center) to (124.center);
                \draw (118) to (119);
                \draw [color=gray!50] (132.center) to (133.center);
                \draw [color=gray!50] (132.center) to (134.center);
                \draw [color=gray!50] (134.center) to (135.center);
                \draw [color=gray!50] (135.center) to (133.center);
                \draw (127) to (128);
                \draw [color=gray!50] (141.center) to (142.center);
                \draw [color=gray!50] (141.center) to (143.center);
                \draw [color=gray!50] (143.center) to (144.center);
                \draw [color=gray!50] (144.center) to (142.center);
                \draw (137) to (138);
                \draw [color=gray!50] (150.center) to (151.center);
                \draw [color=gray!50] (150.center) to (152.center);
                \draw [color=gray!50] (152.center) to (153.center);
                \draw [color=gray!50] (153.center) to (151.center);
                \draw (146) to (147);
                \draw [color=gray!50] (159.center) to (160.center);
                \draw [color=gray!50] (159.center) to (161.center);
                \draw [color=gray!50] (161.center) to (162.center);
                \draw [color=gray!50] (162.center) to (160.center);
                \draw [color=gray!50] (168.center) to (169.center);
                \draw [color=gray!50] (168.center) to (170.center);
                \draw [color=gray!50] (170.center) to (171.center);
                \draw [color=gray!50] (171.center) to (169.center);
                \draw [color=gray!50] (177.center) to (178.center);
                \draw [color=gray!50] (177.center) to (179.center);
                \draw [color=gray!50] (179.center) to (180.center);
                \draw [color=gray!50] (180.center) to (178.center);
                \draw [color=gray!50] (186.center) to (187.center);
                \draw [color=gray!50] (186.center) to (188.center);
                \draw [color=gray!50] (188.center) to (189.center);
                \draw [color=gray!50] (189.center) to (187.center);
                \draw (154) to (156);
                \draw (163) to (165);
                \draw (172) to (174);
                \draw (181) to (183);
                \draw [color=gray!50] (195.center) to (196.center);
                \draw [color=gray!50] (195.center) to (197.center);
                \draw [color=gray!50] (197.center) to (198.center);
                \draw [color=gray!50] (198.center) to (196.center);
                \draw (190) to (191);
                \draw (191) to (192);
                \draw [color=gray!50] (203.center) to (204.center);
                \draw [color=gray!50] (203.center) to (205.center);
                \draw [color=gray!50] (205.center) to (206.center);
                \draw [color=gray!50] (206.center) to (204.center);
                \draw (199) to (200);
                \draw (200) to (201);
                \draw [color=gray!50] (211.center) to (212.center);
                \draw [color=gray!50] (211.center) to (213.center);
                \draw [color=gray!50] (213.center) to (214.center);
                \draw [color=gray!50] (214.center) to (212.center);
                \draw (207) to (208);
                \draw [color=gray!50] (219.center) to (220.center);
                \draw [color=gray!50] (219.center) to (221.center);
                \draw [color=gray!50] (221.center) to (222.center);
                \draw [color=gray!50] (222.center) to (220.center);
                \draw (215) to (216);
                \draw (207) to (209);
                \draw (215) to (217);
                \draw [{Latex[length=3mm,width=5mm]}-,red] (223.center) to (224.center);
                \draw [{Latex[length=3mm,width=5mm]}-,red] (225.center) to (226.center);
                \draw [{Latex[length=3mm,width=5mm]}-,red] (229.center) to (230.center);
                \draw [{Latex[length=3mm,width=5mm]}-,red] (235.center) to (236.center);
                \draw [{Latex[length=3mm,width=5mm]}-,red] (237.center) to (238.center);
                \draw [{Latex[length=3mm,width=5mm]}-,red] (239.center) to (240.center);
                \draw [{Latex[length=3mm,width=5mm]}-,red] (241.center) to (242.center);
                \draw [{Latex[length=3mm,width=5mm]}-,red] (243.center) to (244.center);
                \draw [{Latex[length=3mm,width=5mm]}-,red] (245.center) to (246.center);
                \draw [{Latex[length=3mm,width=5mm]}-,red] (247.center) to (248.center);
                \draw [dashed] (253.center) to (254.center);
                \draw [dashed] (254.center) to (251.center);
                \draw [dashed] (251.center) to (252.center);
                \draw [dashed] (252.center) to (256.center);
                \draw [dashed] (256.center) to (255.center);
                \draw [dashed] (255.center) to (253.center);
                \draw [dashed] (257.center) to (258.center);
                \draw [dashed] (258.center) to (264.center);
                \draw [dashed] (257.center) to (259.center);
                \draw [dashed] (259.center) to (260.center);
                \draw [dashed] (261.center) to (260.center);
                \draw [dashed] (261.center) to (262.center);
                \draw [dashed] (263.center) to (262.center);
                \draw [dashed] (263.center) to (264.center);
            \end{tikzpicture}}

            \caption{The matching on the NBC complex of $P_3$. The arrows indicate the matching edges.}
            \label{Schem_matching}
        \end{figure}

        \subsection{Spanning tree complex}
        % Point 44 addressed.
        For a given graph $G = (V, E)$ with a fixed ordering on its edges and a distinguished vertex $v_{root}$ (which we will sometimes refer to as the root), we have a one-to-one correspondence between the signed NBC spanning trees $T^+, T^-$ of $G$ and the critical points of $\M_{\nbc}$ by \Autoref{thm: Matching on G for A_2}.

        We formally define the spanning tree complex $(\cstBiCom{\ast}{\ast}{2}{G},\cstDiff{\ast,\ast}{2})$ by identifying it bijectively with the Morse complex obtained from the acyclic matching $\M_{\nbc}$ from \Autoref{thm: Matching on G for A_2}. % Point 46 addressed.
        We define the gradings:

        \[
        \begin{cases}
            i(T^+) := i(H_c^{T^+}),  \hspace{1cm} j(T^+) := j(H_c^{T^+}), \\
            i(T^-) := i(H_c^{T^-}),  \hspace{1cm} j(T^-) := j(H_c^{T^-}).
        \end{cases}
        \]

        where $H_c^{T^+}$ and $H_c^{T^-}$ are critical states corresponding to $T$ with respect to the matching $\M_{\nbc}$.

        The chain groups of the spanning tree complex $(\cstBiCom{\ast}{\ast}{2}{G},\cstDiff{\ast,\ast}{2})$ are defined as follows:

        \[
        \cstBiCom{i}{\ast}{2}{G} := \mathbb{Z} \langle T^{\pm} \mid T^{\pm} \in SST(G), \, i(T^+) = i(T^-) = i \rangle.
        \]

        Observe that, as graded abelian groups, $\cstBiCom{i}{\ast}{2}{G}$ decomposes as:

        \[
        \cstBiCom{i}{\ast}{2}{G} := \bigoplus_{j}\PCST{i}{j} \oplus \NCST{i}{j+1},
        \] % Point 46 addressed.

        where $\PCST{i}{j} \subset \cstBiCom{i}{\ast}{2}{G}$ is defined as:

        \[
        \PCST{i}{j} := \mathbb{Z} \langle T^{+} \mid T^{+} \in SST(G), \, i(T^+) = i , j(T^+) = j \rangle.
        \]

        Similarly, $\NCST{i}{j} \subset \cstBiCom{i}{\ast}{2}{G}$ is defined as:

        \[
        \NCST{i}{j} := \mathbb{Z} \langle T^{-} \mid T^{-} \in SST(G), \, i(T^-) = i , j(T^-) = j \rangle.
        \]

        Observe that, $\PCST{i}{j}$ consists of critical states in the $(i,j)$-th grading such that $\epsilon(Z_{\text{root}}) = 1$,
        where $Z_{\text{root}}$ is the connected component of $H_c^{T^+}$ containing the root vertex $v_{root}$.

        Similarly, $\NCST{i}{j}$ consists of critical states in the $(i,j)$-th grading such that $\epsilon(Z_{\text{root}}) = x$,
        where $Z_{\text{root}}$ is the connected component of $H_c^{T^-}$ containing the root $v_{root}$.
        The gradings follow from \Autoref{thm: Matching on G for A_2}.

        Until now, $\cstBiCom{i}{j}{2}{G}$ has only been described as a bigraded group. The differential of this chain complex,
        \[
        \cstDiff{i}{2}: \cstBiCom{i}{\ast}{2}{G} \rightarrow \cstBiCom{i+1}{\ast}{2}{G},
        \]
        is defined via correspondence with the differential
        $\morDiff{i}{\M_{\nbc}}$ of the Morse complex associated to $\M_{NBC}$.
        Later, in \Autoref{differential}, we will show that it has an explicit graph-theoretic description.

        For a given graph $G=(V,E)$, we have defined an acyclic matching $\M_{\nbc}$ on the complex $\nbcCom{2}{G}$ in \Autoref{thm: Matching on G for A_2}. Recall from  \Cref{algebraic morse theory} that whenever we have an acyclic matching $\M$ for a given cochain complex $\left( C_{\ast}, \partial_{\ast}\right)$, we obtain the corresponding Morse complex $\left(C_{\ast}^{\M}, \partial_{\ast}^{\M}\right)$, which is generated by the critical vertices with respect to $\M$. In our case, the corresponding Morse complex is:
        \[
        \morCom{*,*}{\M_{\nbc}}(G) :=\bigoplus_{i}\morCom{i,*}{\M_{\nbc}}(G),
        \]
        where
        \[
        \morCom{i,*}{\M_{\nbc}}(G) =\Z \langle H_c^{T^\pm} \in \morCom{i,*}{\M_{\nbc}}(G) \mid i(H_c^{T^\pm})=i \rangle.
        \] % Point 48 addressed.
        Moreover, we can write:
        \[
        \morCom{i,*}{\M_{\nbc}}(G)=\bigoplus_{j} \morCom{i,j}{\M_{\nbc}}(G).
        \]

        The differential $\partial^{\M_{\nbc}}_{i}: C^{\M_{\nbc}}_{i,j} \rightarrow C^{\M_{\nbc}}_{i+1,j}$ is defined as follows:
        \[
        \partial^{\M_{\nbc}}_{i}(H_c^{T^{\pm}}) =  \sum_{i(H_c^{T'^{\pm}})=i+1}\Gamma(H_c^{T^{\pm}},H_c^{T'^{\pm}})\cdot H_c^{T'^{\pm}}
        \]

        The incidence $\Gamma(H_c^{T^{\pm}}, H_c^{T'^{\pm}})$ is defined as the sum of weights of all alternating paths from $H_c^{T^{\pm}}$ to $H_c^{T'^{\pm}}$ in $\G^{\M_{\nbc}}(\nbcCom{2}{G})$.\\

        Observe that the enhancement of $H_c^{T^+}$ is $1$ at the component containing the root $v_{root}$ and $x$ on all other connected components. Therefore, $j(H_c^{T^{+}})=\lvert \IA(T) \rvert$. The enhancement of $H_c^{T^-}$ is $x$ on all connected components, which implies  $j(H_c^{T^{-}})= \lvert \IA(T) \rvert+1$. If $N$ is the number of vertices of $G$, then we have,

        \[
        i(H_c^{T^{-}}) + j(H_c^{T^{-}}) = |\IN(T)| + |\IA(T)|+1 = N
        \]
        and
        \[
        i(H_c^{T^{+}}) + j(H_c^{T^{+}}) = |\IN(T)| + |\IA(T)| = N-1.
        \]

        Since the differential preserves the quantum grading, it follows that $\cstDiff{}{2}$ vanishes on $H_c^{T^{-}}$. When evaluated at $H_c^{T^{+}}$, we obtain:

        \[
        \partial^{\M_{\nbc}}_{i}(H_c^{T^{+}}) = \sum_{i(H_c^{T'^{\pm}})=i+1}\Gamma(H_c^{T^{+}},H_c^{T'^{\pm}})\cdot H_c^{T'^{\pm}}.
        \]

        There is a one-to-one correspondence between the generators of $\nbcCom{2}{G}$ and the generators of $\morCom{*,*}{\M_{\nbc}}$, leading to a $\mathbb{Z}$-module isomorphism between them. The differential
        \[
        \cstDiff{i,j}{2}: \cstBiCom{i}{j}{2}{G} \rightarrow \cstBiCom{i+1}{j}{2}{G}
        \]
        can then be described as:
        \[
        \cstBiDiff{i}{j}{2}(T^+)=\sum_{i(T'^-)=i+1} [T^+:T'^-]_{\cstBiDiff{i}{j}{2}}\cdot T'^-,
        \]
        where the incidence $[T^+:T'^-]_{\cstBiDiff{i}{j}{2}}$ is given by $\Gamma(H_c^{T^+},H_c^{T'^-})$.

        % \maintheorem*
        \begin{proof}[Proof of \autoref{main} for $m=2$]
            The proof follows immediately from \autoref{NBC Isomorphism}, \autoref{dmt maintheorem}, and \autoref{thm: Matching on G for A_2}.
        \end{proof}
        Now, we will take the graded Euler characteristic of $\cstBiCom{\ast}{\ast}{2}{G}$ and show that it gives us the NBC spanning tree expansion of the chromatic polynomial.

        \begin{proof}[Proof of \Autoref{expancor}]

            Observe that the number of internally inactive edges gives the homological grading of a spanning tree.  Consequently, for a spanning tree $T$ with $i$ internally active edges, the homological grading is $N-i-1$, where $N$ is the number of vertices of the graph $G$.

            Furthermore, in this setting, the tree $T^+$ has quantum grading $i$, while the tree $T^-$ has quantum grading $i+1$.

            Since the internal activity of NBC spanning trees satisfies $1 \leq i \leq N - 1$, we obtain the following expression for the graded Euler characteristic:

            \begin{align*}
                \chi_q(\cstBiCom{\ast}{\ast}{2}{G}) &= \chi_q(\PCST{*}{*}) + \chi_q(\NCST{*}{*}) \\
                &= \sum_{i=1}^{N-1} t_i q^{i} (-1)^{N-i-1} +  \sum_{i=1}^{N-1} t_i q^{i+1} (-1)^{N-i-1} \\
                &= (q+1) \sum_{i=1}^{N-1} t_i q^{i} (-1)^{N-i-1} \\
                &= \lambda \sum_{i=1}^{N-1} t_i (\lambda - 1)^{i} (-1)^{N-i-1} \quad \text{(Substituting } \lambda = 1+q \text{)} \\
                &= (-1)^{N-1} \lambda \sum_{i=1}^{N-1} t_{i} (1 - \lambda)^i
            \end{align*}

            Hence,

            \[
            P(G ; \lambda)=(-1)^{N-1} \lambda \sum_{i=1}^{N-1} t_{i}(1-\lambda)^i
            \]

        \end{proof}

        Another immediate corollary of the spanning tree model is the fact that chromatic homology is supported on the diagonals $i+j=N$ and $i+j=N-1$.
        The generators of $\PCST{i}{j}$ satisfy $i+j=N-1$ and the generators of $\NCST{i}{j}$ satisfy $i+j=N$.

        \begin{corollary}
            The chromatic homology $H_{\ast,\ast}^{\A_2}(G)$ of a graph $G$ is supported on the two diagonals $i+j=N$ and $i+j=N-1$, where $N$ is the number of vertices of $G$.
        \end{corollary}

        \begin{proof}
            The proof follows directly from the preceding discussion.
        \end{proof}
        \begin{remark}
            The above corollary was proved by Helme-Guizon, Przytycki and Rong \cite[Corollary 16]{HelmeGuizonPrzytyckiRong2006_torsion}, using the long exact sequence associated with the deletion and contraction of an edge. Also see Chmutov, Chmutov and Rong \cite{ChmutovChmutovRong2008}, where the same support property is obtained over $\mathbb{Q}$ and refined to the knight move theorem. In the spanning tree model, the statement becomes immediate, since every generator $T^{\pm}$ of $\cstBiCom{\ast}{\ast}{2}{G}$ already lies on one of the two diagonals.
        \end{remark} % Point 50 addressed.
        % %%%%%%%%%%%%%%%%%%%%%%%%%%%%%%%%%%%%%%%%%%%%%%%%%%%%%%%%%%%%%%%%%%%%%%%%%%%%%%%%%%%%%%%%%%%%%%%%%%%%%
        % %%%%%%%%%%%%%%%%%%%%%%%%%%%%%%%%%%%%%%%%%%%%%%%%%%%%%%%%%%%%%%%%%%%%%%%%%%%%%%%%%%%%%%%%%%%%%%%%%%%%%
        % %%%%%%%%%%SECTION 3.4 %%%%%%% Combinatorial form of the differential of the spanning tree complex%%%%
        % %%%%%%%%%%%%%%%%%%%%%%%%%%%%%%%%%%%%%%%%%%%%%%%%%%%%%%%%%%%%%%%%%%%%%%%%%%%%%%%%%%%%%%%%%%%%%%%%%%%%%
        % %%%%%%%%%%%%%%%%%%%%%%%%%%%%%%%%%%%%%%%%%%%%%%%%%%%%%%%%%%%%%%%%%%%%%%%%%%%%%%%%%%%%%%%%%%%%%%%%%%%%%

        \subsection{Combinatorial form of the differential of the spanning tree complex}

        Now we want to provide an explicit formula to compute the incidence $[T:T']$ in the differential $\cstDiff{i,\ast}{2}$, which will help significantly to compute the homology of the spanning tree complex. % Point 51 addressed.

        We start with a few basic observations that will be essential for the combinatorial description of the differential in the algebraic Morse complex. First, we prove a useful property relating to the cutsets of the trees obtained by basis exchange.

        \begin{proposition} \label{same cut prop}
            Given two spanning trees $T_i, T_j$ with $T_j=(T_i \cup \{f\}) \setminus \{e\}$, if $\alpha \not\in \cyc(T_i, f)$, then $\cut(T_i,\alpha)=\cut(T_j,\alpha)$.
        \end{proposition}

        \begin{proof}
            Suppose $h \in \cut(T_i,\alpha)$ then $h$ connects $Z_1$ and $Z_2$ where, $T_i \setminus \alpha = Z_1 \sqcup Z_2$. Now without loss of generality, assume that $e \in Z_1$. Then $\cyc(T_i, f) \setminus f \subseteq Z_1$ since $Z_1 \cap Z_2 = \emptyset$, $\alpha \notin \cyc(T_i, f)$ and $\cyc(T_i, f) \setminus f$ is connected. % Point 52 is addressed.
            Now,
            \[
            T_j \setminus \alpha = ((T_i \cup f)\setminus e)  \setminus \alpha = (Z_1 \cup f)\setminus e \sqcup Z_2
            \]
            Now as $\alpha \not\in \cyc(T_i, f)$ so, $h \neq e$ and thus, $h$ connects $(Z_1 \cup f)\setminus e$ and $Z_2$ which implies $h \in \cut(T_j,\alpha)$ and hence, $\cut(T_i,\alpha) \subseteq \cut(T_j,\alpha)$. Similarly, one can show that $\cut(T_j,\alpha) \subseteq \cut(T_i,\alpha)$ since $h \neq f$ and thus, we have our desired result.
        \end{proof}

        Next, we prove another important observation that is crucial for the combinatorial description of the differential.

        \begin{proposition}\label{main differential lemma}
            Suppose $T_{i} < T_{j}$ are two distinct NBC trees such that $\R(T_{i}) \cup f \in [\R(T_{j}), T_{j}]$ then $T_{j}=  \left(T_{i}-e\right) \cup f $ where $e$ is the maximal internally live edge in $T_i$ that belongs to the unique cycle in $T_{i} \cup f$. Moreover, $\R(T_{i}) \cup f = \R(T_{j})$.
        \end{proposition}
        \begin{proof}
            Observe that the hypothesis $T_{i}<T_{j}$ is redundant because of \Autoref{crucial}. We note that $f \in \IN(T_{j})$. Otherwise $\IN(T_{i}) = \left(\IN(T_{i}) \cup f \right)-f \in [\R(T_{j}), T_{j}]$ by virtue of \Autoref{inactiveedgelemma}. It contradicts the fact that $[\R(T_{i}), T_{i}] \cap [\R(T_{j}), T_{j}]= \emptyset$. Thus, the added edge $f$ must belong to $\R(T_{j})$.

            Let $T_{k}$ denotes the tree as described in the hypothesis, that is, $T_{k}$ is obtained by removing the maximal internally live edge in $T_i$ that belongs to the unique cycle in $T_{i}\cup f $. We show that such $T_{k}$ must always exist and $T_{k}$ is an NBC spanning tree. We start with the following observations:
            \begin{itemize}
                \item Let $Z$ denote the unique cycle in $T_{i} \cup f$. Then $Z$ must contain at least one internally live edge in $T_{i}$.
                \item $T_{k}$ must be an NBC spanning tree.
            \end{itemize}
            The first observation is trivial because otherwise $Z \subseteq  \left(\IN(T_{i}) \cup f \right)= \left( \R(T_{i}) \cup f \right) \subseteq T_{j} $. The equality follows from \Autoref{inactiveedgelemma}. It contradicts the fact that $T_{j}$ is a tree. Thus, we can assume such $T_{k}$ must exist.

            For the second part of the observations, we distinguish between two cases, depending on whether $Z$ contains more than one internally active edge or exactly one internally active edge. If $Z$ contains more than one internally live edge in $T_{i}$, then we shall show that $T_{k}$ must be an NBC spanning tree. If not, then there is an edge $c \notin T_{k}$ that is externally active with respect to $T_{k}$, that is, $c$ is the smallest edge of the cycle $C:= \cyc(T_{k},c)$. We separate two cases
            \begin{itemize}
                \item Suppose $c = e$. Then $C = Z$. As $Z$ contains more than one internally live edge so, there exists $\alpha \neq e$ with $\alpha \in \IA(T_{i})$, then $\alpha < e$ by the choice of $e$, which contradicts the minimality of $c=e$ in $C$.
                \item Suppose $c \neq e$. Then $c \notin T_{i} \cup f$, so $c$ is an external edge of $T_{i}$. If $e \notin \cyc(T_{i},c)$, then $C = \cyc(T_{i},c)$, and $c$ would be externally active with respect to $T_{i}$, which is impossible. Hence $e \in \cyc(T_{i},c)$, that is, $c \in \cut(T_{i},e)$, and therefore $e<c$, since $e$ is internally live in $T_{i}$. Moreover, $C$ is the symmetric difference of the cycles $\cyc(T_{i},c)$ and $Z$, since this symmetric difference is a cycle of $T_{k} \cup c$ containing $c$. We claim that $C-c$ contains no internally live edge of $T_{i}$. Indeed, let $\beta \in C - c$ be internally live in $T_{i}$; then $\beta \neq e$. If $\beta \in Z$, then $\beta < e < c$ by the choice of $e$; if $\beta \notin Z$, then $\beta \in \cyc(T_{i},c)$, so that $c \in \cut(T_{i},\beta)$ and hence $\beta < c$, since $\beta$ is internally live. In both cases, we obtain a contradiction with the minimality of $c$ in $C$. Consequently, $C - c \subseteq \IN(T_{i}) \cup f = \R(T_{i}) \cup f \subseteq T_{j}$, so that $T_{j}$ contains the broken circuit $C-c$, which is again a contradiction.
            \end{itemize}
            Thus, $T_k$ is an NBC spanning tree. So, without loss of generality, we may assume that the cycle $Z$ contains a unique internally live edge $ \alpha$ of $T_{i}$. That is \[ \alpha \in \IA(T_{i}). \]
            Moreover, for the sake of contradiction, we may also assume $\alpha$ is minimal in the cycle $Z$, that is, $Z-\alpha$ is a broken circuit. But then
            \[ Z -\alpha \subseteq \R(T_{i}) \cup f \subseteq T_{j} \] where $T_{j}$ is an NBC spanning tree. Thus, $\alpha$ cannot be minimal in $Z$. Hence, $T_{k}$ must be an NBC spanning tree.\\
            To finish the proposition, we show $T_{k}$ must be equal to $T_{j}$. Note that $f \in T_{k}$ must be internally dead, that is $f \in \IN(T_{k})$. Otherwise,
            \[ T_{k}-f= T_{i}-e \in [\R(T_{i}), T_{i}] \cap [\R(T_{k}), T_{k}] \]
            Thus, $f \in \IN(T_{k})$. If $\IN(T_{k}) \subset \IN(T_{i})\cup f$, then
            \[ \IN(T_{i}) \cup f = \R(T_{i}) \cup f \in [\R(T_{k}), T_{k}] \cap [\R(T_{j}), T_{j}]  \]
            concluding $k=j.$
            Thus, we may assume \[  \left(\IN(T_{k})-f \right) \not\subset \IN(T_{i}).   \]
            Thus, there exists $\alpha \in T_{i} \cap T_{k}= T_{i}-e= T_{k}-f $ such that $\alpha \in \IA(T_{i})$ and $\alpha \in \IN(T_{k})$. In other words, some internally active edge in $T_{i}$ becomes internally inactive in $T_{k}$. We shall show that it leads to a contradiction.

            If $\alpha \notin Z$, the unique cycle in $T_{i} \cup f$, then by \Autoref{same cut prop} we have the following:
            \[ \cut(T_{i}, \alpha) = \cut (T_{k}, \alpha) \]
            Hence,  we may assume $\alpha \in Z$ is internally active in $T_{i}$. Then, by the choice of $e$, we must have \[ \alpha < e \, \, \text{and} \, \, e<f. \]
            The first inequality follows from the fact that $e$ is the maximum among the internally active edges of $Z$ in $T_{i}$. The second inequality follows from the fact that $e$ is internally active and $f \in \cut(T_{i},e)$.

            Now $\alpha \in \IN(T_{k})$ implies that there exists $\beta \in \cut(T_{k}, \alpha)$ such that $\beta < \alpha$. Combining all the inequality we get
            \[\beta< \alpha< e< f.\]
            Note that, $\beta \notin \cut(T_{i}, e)$ as $e$ is the minimal edge in $\cut(T_{i},e)$. Thus, $\beta$ must connect two vertices in $T_{i}^{(1)}$ or $T_{i}^{(2)}$ depending on whether $\alpha$ belongs to $T_{i}^{(1)}$ or $T_{i}^{(2)}$ as in \autoref{the relations}. In both cases, this contradicts the hypothesis \[ \alpha \in  \IA(T_{i}).  \]

            This finishes the argument that $j=k$. \\

            Finally, observe that if any edge $x \in \R(T_{i})=\IN(T_{i})$, then it is not the minimum edge in $\cut(T_i,x)$. Now, either $x \notin \cyc(T_i,f)$, in which case $\cut(T_j,x)=\cut(T_i,x)$, or $x \in \cyc(T_i,f)$, in which case
            \[
            \cut(T_i,x) \setminus \cut(T_j,x) \subseteq \cut(T_i, e).
            \]
            In the second case, observe that $e$ is internally active in $T_i$, so it is the minimum edge in $\cut(T_i,e)$. Also, since $e \in \cut(T_j,x)$, it follows that $x$ is not the minimum edge in $\cut(T_j,x)$. Hence, $x \in \IN(T_j)= \R(T_{j})$. Therefore, it follows that
            \[
            \R(T_{i}) \cup f= \R(T_{j}).
            \]
        \end{proof}
        Before stating the next lemma we draw an important conclusion regarding alternating paths between two critical states. First observe that, any alternating path between two critical states must start and end with a differential edge. Thus, as a consequence of the \Autoref{main differential lemma}, any alternating paths between two critical states cannot change the $NBC$ trees more than once. Otherwise, it will increase the homological grading by more than $1$.
        \begin{lemma}\label{lextrick}

            Given two NBC spanning trees \(T, T'\) with \(T < T'\), the incidences satisfy the following conditions:
            \[
            [T^+:T'^+]_{\cstBiDiff{*}{*}{2}} = [T^-:T'^-]_{\cstBiDiff{*}{*}{2}} = [T^-:T'^+]_{\cstBiDiff{*}{*}{2}} = 0.
            \]
            The incidence \([T^+:T'^-]_{\cstBiDiff{*}{*}{2}}\) is nonzero only if
            \[
            T' = (T - e) \cup f,
            \]
            where \(f\) is an external edge in \(T\) and \(e\) is the maximal internally live edge in $T$ that belongs to the unique cycle in \(T \cup f\).
        \end{lemma}

        \begin{proof}
            As observed earlier, the generators are supported in the diagonals $i(T^+) + j(T^+) = N - 1$ and $i(T^-) + j(T^-) = N$. Since the differential must preserve the quantum grading ($j$-grading) and increase the homological grading($i$-grading) by $1$, it immediately follows from \Autoref{main differential lemma} and the conclusion before this lemma that for any two NBC spanning trees \(T, T'\) with \(T < T'\), we have

            \[
            [T^+:T'^+]_{\cstBiDiff{*}{*}{2}} = [T^-:T'^-]_{\cstBiDiff{*}{*}{2}} = [T^-:T'^+]_{\cstBiDiff{*}{*}{2}} = 0.
            \]

            Let $i$ be the homological grading of the tree $T$, i.e. $\lvert IN(T) \rvert=i$. Then, observe that if $P$ is an alternating path between $T^+$ and $T'^-$ in $\G^{\M}(\chCom{\ast,\ast}{2}{G})$, then the homological grading of all the enhanced spanning subgraph in $P$ alternates between $i$ and $i+1$.
            A differential edge raises the homological grading by $1$ and a matching edge reduces it by $1$. Also observe that a matching edge cannot leave the intervals $[IN(T_i),T_i]$, (\autoref{thm: Matching on G for A_2}). In other words, if two enhanced spanning subgraphs are matched by a matching edge, then they belong to a single interval $[\IN(T_i),T_i]$. So the path $P$ remains in $[\IN(T),T]$ unless it traverses an edge leading it outside of the interval $[\IN(T),T]$ and for $[T^+:T'^-]_{\cstDiff{i,\ast}{2}}\neq 0$ there must be a differential in the path $P$ leaving the interval $[\IN(T),T]$ as $i(T'^-)=i+1$. Moreover the corresponding differential is to be taken due to an edge $f \notin\IN(T) \sqcup \IA(T)$, i.e. $f$ must be an external edge of the tree $T$. \par

            Now, the first differential edge leaving $[\IN(T),T]$ starts with an enhanced spanning subgraph at grading $i$ in $[\IN(T),T]$, that is a subgraph $(\IN(T),\epsilon)$. Also taking a differential on an external edge $f$ takes $(\IN(T),\epsilon)$ to an enhanced spanning subgraph $(\IN(T) \cup f,\epsilon')$ in some interval $[\mathcal{R}(T_k),T_k]$. Since $f \notin T$ and by \Autoref{crucial} we have that, $T<T_k$. By \Autoref{main differential lemma}, $T_k$ is uniquely determined and,
            \[
            T_k = (T - e) \cup f,
            \]
            where \(f\) is an external edge in \(T\) and \(e\) is the maximal internally live edge in the unique cycle in \(T \cup f\). Also, $i(T_k) = i(T) + 1$ and $\R(T_k)=\IN(T_k) = \IN(T) \cup f$.

            Now, observe that if the path $P$ is to be continued further from the enhanced spanning subgraph $(\IN(T_k),\epsilon')$, the next edge in the path must be a matching edge leading to an enhanced spanning subgraph of homological grading $i$, in the interval $[\IN(T_k),T_k]$. There is no such subgraph in the interval $[\IN(T_k),T_k]$, since every enhanced spanning subgraph in this interval has homological grading at least $i+1$. Hence $P$ must terminate here and $(\IN(T_k), \epsilon')= H_c^{T'^-}$ for $T'=T_k$. This concludes the proof.
        \end{proof}

        From here onwards, we sometimes alternately denote any enhanced spanning subgraph $(H,\epsilon)$ in a partition $[\R(T),T]$ by $H_{\epsilon}^T$. \\

        Given an NBC spanning tree $T$ of $G$ and two enhanced spanning subgraphs $H^T_{\epsilon}$ and $(H')^T_{\epsilon'}$ in $[\R(T),T]$, define
        \[
        \Pt^{\downarrow}_2(H^T_{\epsilon},(H')^T_{\epsilon'})
        \]
        to be the collection of all alternating paths in $\G^{\M_{\nbc}}(\chCom{\ast,\ast}{2}{G})$ between $H^T_{\epsilon}$ and $(H')^T_{\epsilon'}$ with $i(H^T_{\epsilon})=i((H')^T_{\epsilon'})$.

        \begin{corollary}\label{atmost one path cor}
            $|\Pt^{\downarrow}_2(H^T_{\epsilon},(H')^T_{\epsilon'})| \in \{0,1\}$
        \end{corollary}

        \begin{proof}
            Every enhanced spanning subgraph $(H,\epsilon)$ of $[\R(T),T]$ can be visualized as an enhanced spanning subgraph of $\widetilde{T}$, where $\widetilde{T}$ is obtained by contracting the internally inactive edges of $T$. Now we use \Autoref{fundamental} for the graph $\G^{\M_{\widetilde{T}}}(\chCom{}{2}{\widetilde{T}})$ to deduce our claim.
        \end{proof}
        % Point 57 addressed.

        % For any spanning tree $T$ and a vertex $v$ of $G$, we have a unique path in $T$ between the root $v_{root}$ and $v$. Let us denote this unique path by $P_v(T)$. Let $Z$ be a connected component of $\IN(T)$. Then, we can travel from $v_{root}$ to $Z$ in $T$ via a unique alternating path in $\G^{\M_{\nbc}}(\chCom{\ast,\ast}{2}{G})$.

        \begin{figure}[ht]
            \centering

            \resizebox{0.8\textwidth}{!}{
            \begin{tikzpicture}
                \node at (0,0) (graph1) {
                \fbox{
                \begin{tikzpicture}[vertex/.style={draw, circle, inner sep=0pt, minimum size=6pt, fill=black, color=black}]
                    \node[vertex, red, fill=red] [label=above: $v_{root}$] at (0,0) (a) {};
                    \node[vertex] at (1,0) (b) {};
                    \node[vertex] [label=below: $1$] at (2,0) (c) {};
                    \node[vertex] [label=below: $x$] at (3,0) (d) {};
                    \node[vertex] [label=below: $x$] at (4,0) (e) {};
                    \node at (5,0) (f) {};
                    \draw[dashed] (a)--(b);
                    \draw[dashed] (e)--(f);
                \end{tikzpicture}
                }};
                \node at (3,-4) (graph2) {
                \fbox{
                \begin{tikzpicture}[vertex/.style={draw, circle, inner sep=0pt, minimum size=6pt, fill=black, color=black}]
                    \node[vertex, red, fill=red] [label=above: $v_{root}$] at (0,0) (a) {};
                    \node[vertex] at (1,0) (b) {};
                    \node[vertex] at (2,0) (c) {};
                    \node[vertex] at (3,0) (d) {};
                    \node[vertex] [label=below: $x$] at (4,0) (e) {};
                    \node at (5,0) (f) {};
                    \draw[dashed] (a)--(b);
                    \draw[dashed] (e)--(f);
                    \draw (c) -- (d) node[midway,above] {$x$};
                \end{tikzpicture}
                }};
                \node at (7,0) (graph3) {
                \fbox{
                \begin{tikzpicture}[vertex/.style={draw, circle, inner sep=0pt, minimum size=6pt, fill=black, color=black}]
                    \node[vertex, red, fill=red] [label=above: $v_{root}$] at (0,0) (a) {};
                    \node[vertex] at (1,0) (b) {};
                    \node[vertex] [label=below: $x$] at (2,0) (c) {};
                    \node[vertex] [label=below: $1$] at (3,0) (d) {};
                    \node[vertex] [label=below: $x$] at (4,0) (e) {};
                    \node at (5,0) (f) {};
                    \draw[dashed] (a)--(b);
                    \draw[dashed] (e)--(f);
                \end{tikzpicture}
                }};
                \node at (10,-4) (graph4) {
                \fbox{
                \begin{tikzpicture}[vertex/.style={draw, circle, inner sep=0pt, minimum size=6pt, fill=black, color=black}]
                    \node[vertex, red, fill=red] [label=above: $v_{root}$] at (0,0) (a) {};
                    \node[vertex] at (1,0) (b) {};
                    \node[vertex] [label=below: $x$] at (2,0) (c) {};
                    \node[vertex] at (3,0) (d) {};
                    \node[vertex] at (4,0) (e) {};
                    \node at (5,0) (f) {};
                    \draw[dashed] (a)--(b);
                    \draw[dashed] (e)--(f);
                    \draw (d) -- (e) node[midway,above] {$x$};
                \end{tikzpicture}
                }};
                \node at (14,0) (graph5) {
                \fbox{
                \begin{tikzpicture}[vertex/.style={draw, circle, inner sep=0pt, minimum size=6pt, fill=black, color=black}]
                    \node[vertex, red, fill=red] [label=above: $v_{root}$] at (0,0) (a) {};
                    \node[vertex] at (1,0) (b) {};
                    \node[vertex] [label=below: $x$] at (2,0) (c) {};
                    \node[vertex] [label=below: $x$] at (3,0) (d) {};
                    \node[vertex] [label=below: $1$] at (4,0) (e) {};
                    \node at (5,0) (f) {};
                    \draw[dashed] (a)--(b);
                    \draw[dashed] (e)--(f);
                \end{tikzpicture}
                }};
                \draw[->] (graph1) -- (graph2) node[midway,above, xshift=0.5cm] {$d$};
                \draw[->,color=red] (graph2) -- (graph3) node[midway,above,xshift=-0.5cm] {$\M$};
                \draw[->] (graph3) -- (graph4) node[midway,above,xshift=0.5cm] {$d$};
                \draw[->,color=red] (graph4) -- (graph5) node[midway,above,xshift=-0.5cm] {$\M$};
            \end{tikzpicture}}
            \caption{An alternating path between enhanced spanning subgraphs in $[\R(T),T]$. Here, nodes represent components in $\IN(T)$ with enhancement of $1$ or $x$.}
            \label{alternating path in a tree}
        \end{figure}

        \begin{lemma} \label{path within tree}
            Let $T$ be an NBC spanning tree of $G$. For a connected component $Z$ of $\IN(T)$, let $\epsilon_Z$ denote the enhancement of $\IN(T)$ that assigns $1$ to $Z$ and $x$ to all other connected components of $\IN(T)$. Then, for an enhanced spanning subgraph $H_{\epsilon}^T$ in $[\R(T),T]$,
            \[
            |\Pt^{\downarrow}_2(H_c^{T^+}, H_{\epsilon}^T)| = 1
            \]
            if and only if $H_{\epsilon}^T = (\IN(T),\epsilon_Z)$ for some connected component $Z$ of $\IN(T)$. Here $H_c^{T^+}$ denotes the critical state of $[\R(T),T]$ described in \Autoref{thm: Matching on G for A_2}.
        \end{lemma}

        \begin{proof}
            Let $Z_{\text{root}}$ be the unique connected component of $\IN(T)$ containing the root vertex $v_{root}$, and let $\pi_0(H_c^{T^+}) = \{Z_{\text{root}}, Z_1, \dots, Z_k\}$ be the connected components of $H_c^{T^+}$.Observe that $H_c^{T^+}=(\IN(T),\epsilon_{Z_{\text{root}}})$. Consider the unique path $P(Z)$ from $Z_{\text{root}}$ to $Z$ in the tree $T'$, obtained by contracting the internally inactive edges of $T$.

            Any alternating path starting from $H_c^{T^+}$ must begin with a non-matched directed edge. Such a path can arise only from an internally active edge $f_1$ that is incident to the component $Z_{\mathrm{root}}$. This is because $Z_{\mathrm{root}}$ is the unique component of $\IN(T)$ for which there exists an enhancement $\epsilon$ satisfying $\epsilon(Z_{\mathrm{root}})=1$.

            Let $Z_1$ be the other component of $\IN(T)$ that merges with $Z_{\mathrm{root}}$ upon the addition of $f_1$. Denote by $d^{f_1}(H_c^{T^+})$ the enhanced state obtained after including the edge $f_1$.

            The next directed edge in this path must be a matching, say $m_1$. By the description of $\M_{\nbc}$, this matching $m_1$ must occur solely due to the internally active edge $f_1$. Observe that $m_1(d^{f_1}(H_c^{T^+}))$ is an enhanced state with the same homological grading as $H_c^{T^+}$. Thus, there exists a path in $T$ from $H_c^{T^+}$ to $(\IN(T), \epsilon_{Z_1})$ via $f_1$.

            This alternating path can continue through a sequence of internally active edges $\{f_1, f_2, \dots, f_i\}$ of $T$ along $P(Z)$ (See \autoref{alternating path in a tree}). Therefore,
            \[
            \lvert \Pt^{\downarrow}_2(H_c^{T^+}, H_{\epsilon}^T) \rvert = 1.
            \]

            This also establishes the converse, as we have identified the necessary criterion that $H_{\epsilon}^T$ must satisfy for $\lvert \Pt^{\downarrow}_2(H_c^{T^+}, H_{\epsilon}^T)\rvert = 1$, ensuring the result holds.
        \end{proof}
        % Point 58 addressed.

        To state our theorem about the description of the differential, we first introduce some notation. Let $v \in V(G)$ and $T$ is a spanning tree of $G$. For $A \subseteq E$, we shall use the notation $L(A, T)$ to denote the collection of edges in $A$ that are internally active with respect to $T$. When $T$ is clear from the context, we shall abbreviate the notation to $L(A)$. For example, $L(P_v(T))$ denotes the collection of internally active edges in the unique path $P_v(T)$ in $T$ from the root vertex $v_{root}$ to $v$. Let $\xi(T,e)$ denote the number of edges in $\IN(T)$ that are ordered less than $e$.

        \begin{theorem}\label{differential}
            Suppose $T<T'$ and $T'=T \cup \{e\} \setminus \{f\}$ with $e \in \IN(T')  \setminus \IN(T)$. Then the incidence between two critical states $H_c^{T^+},H_c^{T'^-}$ is completely determined by its activity word $W(T)$ and $W(T')$. In fact, if $v_1$ and $v_2$ are the endpoints of $e$, then,
            \[[T^+:{T'}^-]_{\cstDiff{\ast,\ast}{2}}= {(-1)}^{\xi(T,e)}({(-1)}^{|L(P_{v_1}(T))|} +  {(-1)}^{|L(P_{v_2}(T))|}).\]

            For any other pair of trees $T$ and $T'$, we have $[T^+:{T'}^-]_{\cstDiff{\ast,\ast}{2}}=0$. Therefore, for any two trees $T$ and $T'$,  $[T^+:{T'}^-]_{\cstBiDiff{*}{*}{2}} \in \{ 0, 2, -2 \}.$
        \end{theorem}

        \begin{proof}

            By \Autoref{lextrick}, there is exactly one less internally inactive edge in $T$ than $T'$, and hence, there is a pair of edges $(e,f)$ such that $e \in \IN(T')$ and $f$ is the smallest edge in $\cut(T',e)$. So, $T'=T \cup \{e\} \setminus \{f\}$.\\

            \begin{figure}[ht]
                \centering
                \scalebox{0.9}{
                \begin{tikzpicture}[vertex/.style={circle, draw, inner sep=0pt, minimum size=6pt, fill=black}, edge/.style={draw}, baseline=0]
                    \node[vertex,fill=red,color=red] [label=above:$v_{root}$]  (a) at (0,2) {};
                    \node[vertex] (b) at (0,1.2) {};
                    \node[vertex] [label=right:$v_2$]  (c) at (0,0.4) {};
                    \node[vertex] [label=right:$v_1$]  (d) at (0,-1) {};
                    \node[vertex] (e) at (-1.5,1) {};
                    \node[vertex]  (f) at (-1.5,0) {};
                    \node[vertex] (g) at (1.5,0.5) {};
                    \path[edge] (a) to node[left] {$L$} (b);
                    \path[edge] (b)   to node[left] {$L$}  (c);
                    \path[edge] (a) to[out=180,in=70] node[ left ]{$L$} (e);
                    \path[edge] (e) to node[left]{$L$} node[ right ]{$f$} (f);
                    \path[edge] (f) to[out=290,in=180]node[ left ]{$L$} (d);
                    \path[edge, dotted] (c) to node[ left ]{$e$} (d);
                    \path[edge] (a) to[out=0,in=90]node[ left ]{$L$} (g);
                    \node at (0,-2) {$T^+$ };
                \end{tikzpicture}

                \begin{tikzpicture}[baseline]
                    \node at (0,0) (a) {};
                    \node at (1,0) (b) {};
                    \draw[->] (a)--(b) node[midway,above] {$\partial_{ST}$};
                \end{tikzpicture}

                \begin{tikzpicture}[vertex/.style={circle, draw, inner sep=0pt, minimum size=6pt, fill=black}, edge/.style={draw}, baseline=0]
                    \node[vertex,fill=red,color=red][label=above:$v_{root}$]  (a) at (0,2) {};
                    \node[vertex] (b) at (0,1.2) {};
                    \node[vertex] (c) at (0,0.4) {};
                    \node[vertex] (d) at (0,-1) {};
                    \node[vertex] (e) at (-1.5,1) {};
                    \node[vertex] (f) at (-1.5,0) {};
                    \node[vertex] (g) at (1.5,0.5) {};
                    \path[edge] (a) to node[ left ]{$L$} (b);
                    \path[edge] (b) to node[ left ]{$L$} (c);

                    \path[edge] (c) to node[ left ]{$D$}  (d);
                    \path[edge] (a) to[out=180,in=70] node[ left ]{$L$} (e);
                    \path[edge] (f) to[out=290,in=180] node[ left ]{$L$}  (d);
                    \path[edge] (a) to[out=0,in=90] node[ left ]{$L$} (g);
                    \node at (0,-2) {${T'}^{-}$ };
                \end{tikzpicture}}

                \caption{Counting incidences between spanning trees $T^+$ and ${T'}^{-}$}
                \label{incidencecount}
            \end{figure}

            \begin{figure}[ht]
                {
                \centering
                \scalebox{0.65}
                {\includegraphics{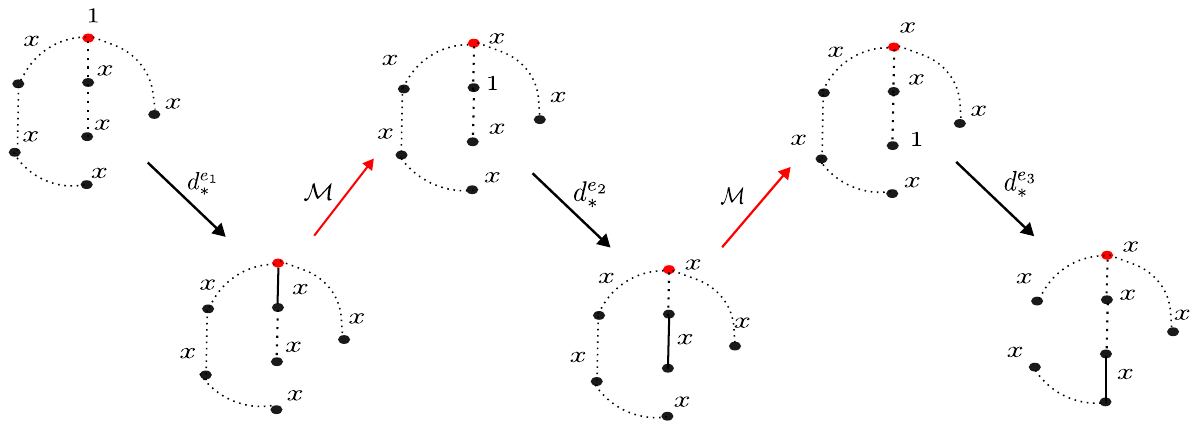}
                }
                \caption{Alternating path corresponding to $P_{v_2}(T)$}
                \label{pv2vd}
                }
            \end{figure}

            \begin{figure}[ht]
                {
                \centering
                \scalebox{0.65}
                {\includegraphics{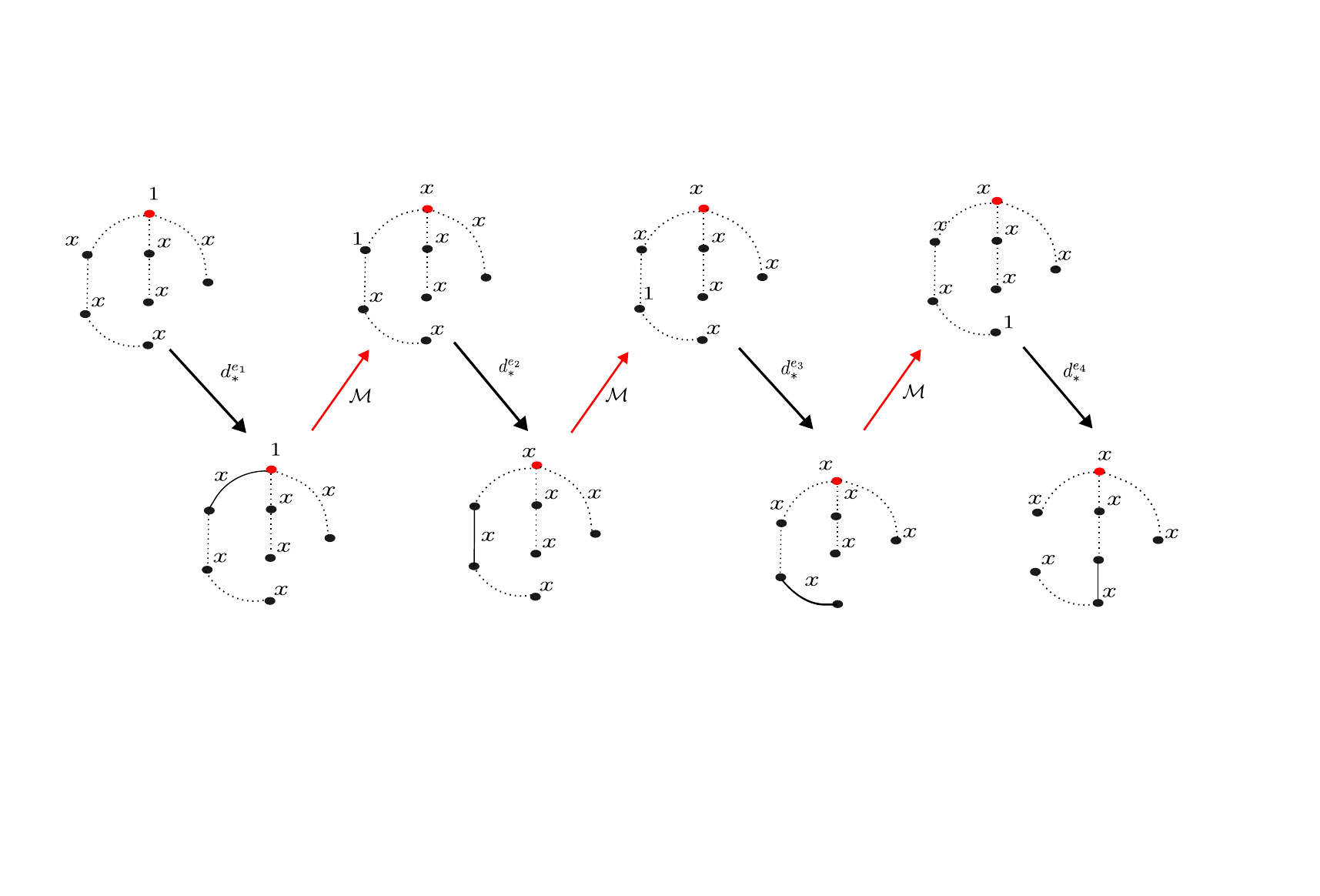}
                }
                \caption{Alternating path corresponding to $P_{v_1}(T)$}
                \label{pv1vd}
                }
            \end{figure}

            Now an alternating path $\Pt$ between $H_c^{T^+}$ and $H_c^{T'^-}$ will consist of a pair of enhanced spanning subgraphs $(S_T,E_T)$, where $S_T=H_c^{T^+}$, $E_T \in [\R(T),T]$ with $\lvert \Pt^{\downarrow}_2(S_T,E_T) \rvert =1$. When $\Pt$ travels from $[\R(T),T]$ to $[\R(T'),T']$, at least one of the end vertices of $e$ in $E(T)$ must have $1$ as their enhancement, and thus, by \Autoref{path within tree}, $E(T)$ either corresponds to $P_{v_1}(T)$ or $P_{v_2}(T)$   (See \autoref{incidencecount}, \autoref{pv2vd}, and \autoref{pv1vd}, in which case we get $[T^+:{T'}^-]_{\cstDiff{\ast,\ast}{2}}$). For each such choice, if we apply the merge map at $e$, then we will reach the critical state $H_c^{T'^-}$. Thus, there are exactly two alternating paths between the critical states $H_c^{T^+}$ and $H_c^{T'^-}$.  \\

            To determine the incidence $[T^+:{T'}^-]_{\cstBiDiff{*}{*}{2}}$, we will need to compute the weight of the alternating paths. Now observe that, using algebraic Morse theory, each matching edge contributes a $-1$ factor to the weight, and finally, when we apply the differential using the edge $e$, a weight factor of ${(-1)}^{\xi(T,e)}$ is multiplied. Therefore, the weight of the alternating path corresponding to $P_{v_1}(T)$ is

            \[
            {(-1)}^{\xi(T,e)}{(-1)}^{|L(P_{v_1}(T))|}
            \]

            and the weight corresponding to $P_{v_2}(T)$ is

            \[
            {(-1)}^{\xi(T,e)}{(-1)}^{|L(P_{v_2}(T))|}.
            \]

            By summing up, we obtain

            \[
            [T^+:{T'}^-]_{\cstBiDiff{*}{*}{2}}= {(-1)}^{\xi(T,e)}({(-1)}^{|L(P_{v_1}(T))|} +  {(-1)}^{|L(P_{v_2}(T))|}).
            \]

            Now, finally from \Autoref{lextrick} and \Autoref{path within tree} % Point 59 addressed.
            it follows that for any pair of spanning trees $(T, T')$ that do not satisfy the hypothesis of this theorem, we have $[T^+:{T'}^-]_{\cstDiff{\ast,\ast}{2}}=0$.

        \end{proof}

        Using \Autoref{differential}, we can write a neater expression of the spanning tree differential. If $e$ is an external edge of an NBC spanning tree $T$, then denote the collection of edges in $\cyc(T,e)$ that are live with respect to $T$ by $L(\cyc(T,e))$. We note that

        $$
        |L(\cyc(T,e))| \equiv  |L(P_{v_1}(T))| + |L(P_{v_2}(T))| \pmod{2}.
        $$

        Also, observe that there is a unique NBC spanning tree  $T'=T \cup \{e\} \setminus \{f\}$ where the external edge $e$ is inactive and $i(T')=i(T)+1$ if such a tree exists. It is obtained by removing the largest internally active edge in $\cyc(T,e)$.

        \begin{definition}[Trees in the image] \label{imagetree}
            Given a tree $T$ and an edge $e \in \EN(T)$ , define
            \[
            \psi_e(T) := T \cup \{e\} \setminus \{\max(L(\cyc(T,e)))\}
            \]
            where $\EN(T)$ denotes the collection of externally inactive edges of $T$.
        \end{definition}

        There is also a dual notation of the above map $\psi_e$, which we define below, as it will be used later.

        \begin{definition}[Trees in the preimage] \label{preimagetree}
            Given a tree $T$ and an edge $e \in \IN(T)$, define
            \[
            \psi'_e(T) := T \cup \{ \min(\cut(T,e))\} \setminus \{ e \}.
            \]
        \end{definition}
        Note that, by \Autoref{observation}, $\psi'_e(T)$ is an NBC spanning tree.
        \begin{lemma} \label{min dead edge}
            Suppose \( T \) is an NBC spanning tree with \( i(T) \geq 1 \). Then there exists \( e \in \IN(T) \) such that \( \psi'_e(T) \) is an NBC spanning tree satisfying
            \[
            i(\psi'_e(T)) = i(T) - 1.
            \]
            Moreover, $\psi_e(\psi'_e(T))=T$.
        \end{lemma}

        \begin{proof}
            Our candidate for $e$ will be the smallest edge in $\IN(T)$. Suppose for the sake of contradiction $i(\psi'_e(T)) \neq i(T)-1$. Then there exists $e' \in T \cap \psi'_e(T)$ for which $a_{T}(e') \neq a_{\psi'_e(T)}(e')$. Let us assume $f = \min(\cut(T,e))$ then we have the following cases:\\

            \textbf{Case 1:} Suppose $a_T(e')=L$ and $a_{\psi'_e(T)}(e')=D$.  Then, by \Autoref{same cut prop}, we may assume $e' \in \cyc (T,f)= \cyc(\psi'_e(T),e)$. Since $a_{\psi'_e(T)}(e')=D$, we get
            \[ \min \cut(\psi'_e(T),e') < e'. \] Suppose $ \psi'_e(T)- \{ e', f \}= T_{\alpha} \sqcup T_{\beta} \sqcup T_{\gamma}$ denotes the decomposition into forest and $e'$ connects $T_{\alpha}$ and $T_{\beta}$. If, the edge $ \min \cut(\psi'_e(T),e')$, connects a vertex of $T_{\alpha}$ with a vertex of $T_{\beta}$, then we get a contradiction to the fact $a_{T}(e')= L$. This is because, the same edge $ \min \cut(\psi'_e(T),e')$ also belongs to $\cut(T,e')$. Thus, we may assume that
            $ \min \cut(\psi'_e(T),e')$ is incident to a vertex in $T_{\gamma}$. In other words,  \[ \min \cut(\psi'_e(T),e') \in \cut (T,e) = \cut (\psi'_e(T), f).  \]
            Since by hypothesis, $f= \min \cut(T,e)$, we conclude $f < \min \cut(\psi'_e(T),e').$ Combining the previous inequality, we conclude the following chains of inequality
            \[ f < \min \cut(\psi'_e(T),e') < e'.\]
            On the other direction, $f \in \cut (T,e')$ and $a_T(e')=L$. Thus, $e'< f$, which is a contradiction to the previous inequality.

            \textbf{Case 2:} Suppose $a_T(e')=D$ and $a_{\psi'_e(T)}(e')=L$. Then $f \in \cut(T,e')$ and $f < e'$, but then $e \in \cut(\psi'_e(T),e')$ which implies $e' < e$ and this contradicts the minimality of $e$.\\

            In order to show $\psi_e(\psi'_e(T))=T$, we need to show that $f=\min(\cut(T,e)) = \max(L(\cyc(\psi'_e(T),e)))$. Suppose there exists $g > f$ with $g \in L(\cyc(\psi'_e(T),e))$ then we have $f < g < e$ but since $f \in \cut(T,g)$, then $g \in \IN(T)$. This contradicts the fact that $e$ is the smallest edge in $\IN(T)$.
        \end{proof}
        The key takeaway from the previous lemma is the following observation:
        \begin{observation}\label{crucial observation}
            For any edge $e'$ that is present in both the trees $ T$ and $\psi'_e(T)$, we have $a_{T}(e')= a_{\psi'_e(T)}(e')$. Also the edge $f= \min (\cut(T,e))$ becomes internally active in $\psi'_e(T)$.
        \end{observation}

        We also define the following function for any vertex $v$ incident to an external edge $e$ of $T$.

        \begin{definition}
            Let $T$ be a spanning tree and let $e$ be an external edge of $T$. Define the function $s_e(T)$ as follows:

            \[
            s_e(T)= \begin{cases}
                (-1)^{|L(P_v(T))|}, & \text{if } |L(\cyc(T,e))| \equiv 0 \pmod{2}\\
                0, & \text{if } |L(\cyc(T,e))| \equiv 1 \pmod{2}
            \end{cases}
            \]

        \end{definition}

        It may seem that $s_e(T)$ depends on the choice of $v$ but the parity of $\lvert L(P_v(T)) \rvert $ does not depend on $v$ as $|L(\cyc(T,e))| \equiv 0 \pmod{2}$ implies $|L(P_{v_1}(T))| + |L(P_{v_2}(T))|$ is even which means the parity of $L(P_{v_i}(T))$ is same for both $i=1,2$. Then, we have the following formula for the spanning tree differential. Thus, \Autoref{differential} can concretely be stated as the following formula:

        \begin{equation}
            \cstDiff{\ast,\ast}{2}(T^+) = \sum_{\substack{e \in \EN(T) \\ \psi_e(T) \in \F(\NBC(G)) \\ i(\psi_e(T)) = i(T) + 1}} (-1)^{\xi(T,e)} \cdot 2 s_e(T) \cdot {\psi_e(T)}^-.
            \label{eq:spanning_tree_differential}
        \end{equation}

        \begin{remark}
            It should be noted that if $e$ is the smallest edge in $\IN(T)$ then, $\psi_e(\psi'_e(T))=T$.
        \end{remark}

        Now, we will prove a theorem that will help us understand the $\Z_2$ torsions in the chromatic homology.

        \begin{theorem}\label{two torsion}
            Let $G$ be a connected graph with an ordering $\omega$ on the edge set. If some NBC spanning tree $T$ satisfies $ \cstDiff{i,\ast}{2}(T^+)=0$ and the homological grading of $T^+$ is $i>0$, then there exists a $\mathbb{Z}_2$ summand at the $i$-th homological grading.

            As a consequence, there is a $2$-torsion in the homological grading $i(T^+)$, where $T$ is the maximal spanning tree with respect to the lexicographic ordering induced by $\omega$.
        \end{theorem} % Point 61 addressed.

        \begin{proof}
            Assume that $T^+ \in {CST^+}_{i,\ast}(G)$ with $i>0$ and $\cstDiff{\ast,\ast}{2}(T^+)=0$. Since $T$ has at least one dead edge, then according to \Autoref{min dead edge}, we choose $D_f$ to be the minimum ordered dead edge of $T$,
            and we have $\psi'_{D_f}(T)= \left(T - D_{f} \right) \cup d_{e}$ ,  where $d_e$ is the minimum edge in $\cut(T,D_f)$. Let $Z$ be the unique cycle $\cyc(\psi'_{D_f}(T), D_f)$ and $v_1$ and $v_2$ be the endpoints of $d_e$. Then by virtue of the concrete description of the differential in Equation \eqref{eq:spanning_tree_differential}, we see that $Z$ must contain an odd number of edges that are internally live with respect to $T$ since $\cstDiff{i,\ast}{2}(T^+)$ is assumed to be $0$.  Also, by \Autoref{crucial observation}, $Z$ contains an even number of internally live edges of $\psi'_{D_f}(T)$.

            \begin{figure}[ht]
                \centering
                \scalebox{0.9}{
                \begin{tikzpicture}[vertex/.style={circle, draw, inner sep=0pt, minimum size=6pt, fill=black}, edge/.style={draw}, baseline=0]
                    \node[vertex]   (a) at (0,2) {};
                    \node[vertex]  (c) at (0,0.4) {};
                    \node[vertex]  (d) at (0,-1) {};
                    \node[vertex] (e) at (-1.5,1) {};
                    \node[vertex]  (f) at (-1.5,0) {};
                    \node[vertex,fill=red,color=red][label=below:$v_{root}$] (g) at (1.5,0.5) {};
                    \path[edge] (a) to node[left] {$L$} (c);
                    \path[edge] (a) to[out=180,in=70] node[ left ]{$L$} (e);
                    \path[edge] (e) to node[left] {$L$} (f);
                    \path[edge] (f) to[out=290,in=180]node[ left ]{$L$} (d);
                    \path[edge] (a) to[out=0,in=90]node[ left ]{$L$} (g);
                    \node at (0,-2) {${\psi'_{D_f}(T)}^+$ };
                \end{tikzpicture}

                \begin{tikzpicture}[baseline]
                    \node at (0,0) (a) {};
                    \node at (1,0) (b) {};
                    \draw[->] (a)--(b) node[midway,above] {$\partial_{ST}$};
                \end{tikzpicture}

                \begin{tikzpicture}[vertex/.style={circle, draw, inner sep=0pt, minimum size=6pt, fill=black}, edge/.style={draw}, baseline=0]
                    \node[vertex]  (a) at (0,2) {};
                    \node[vertex] (c) at (0,0.4) {};
                    \node[vertex] (d) at (0,-1) {};
                    \node[vertex] [label=right:$v_1$] (e) at (-1.5,1) {};
                    \node[vertex][label=right:$v_2$] (f) at (-1.5,0) {};
                    \node[vertex,fill=red,color=red] [label=below:$v_{root}$](g) at (1.5,0.5) {};
                    \path[edge] (a) to node[ left ]{$L$} (c);
                    \path[edge] (c) to node[ left ]{$D_f$}  (d);
                    \path[edge] (a) to[out=180,in=70] node[ left ]{$L$} (e);
                    \path[edge,dotted] (e) to node[ left ]{$d_e$} (f);
                    \path[edge] (f) to[out=290,in=180] node[ left ]{$L$} (d);
                    \path[edge] (a) to[out=0,in=90] node[ left ]{$L$} (g);
                    \node at (0,-2) {${T}^{-}$ };
                \end{tikzpicture}}

                \caption{ $[{\psi'_{D_f}(T}^+):{T}^-]_{\cstDiff{i,\ast}{2}} = \pm 2$}
                \label{preimage}
            \end{figure}

            % Hence, $Z$ must have an odd number of live edges.
            Thus, by Equation \eqref{eq:spanning_tree_differential} and \Autoref{min dead edge}, we have \[ [{\psi'_{D_f}(T)}^+:{T}^-]_{\cstDiff{i,\ast}{2}} = \pm 2. \] Moreover, $\cstDiff{i,\ast}{2}({\psi'_{D_f}(T)}^+) = 2X$, where $X=\pm T^- + \sum_{k \in I} \pm {T_k}^-$ (Here, $I$ is an indexing set that is possibly empty, and the homological grading of each $T_k$ for $k \in I$ is $i$). Since $\cstDiff{i+1,\ast}{2}(X)=0$, it follows that chromatic homology has a $\mathbb{Z}_2$ summand at homological grading $i$.\\

            Now assume $T$ is the NBC spanning tree which is maximal with respect to the lexicographical order. As $\cstDiff{i(T^+),\ast}{2}({T^+})=0$, it follows that there is a $\mathbb{Z}_2$ summand at the grading $i(T^+)$.
        \end{proof}

        \subsection{Computations}
        \begin{example}

            Let us compute the spanning tree homology for an $n$-cycle $P_n$, where $n>2$. Now we have $n$  spanning trees and $n-1$  NBC-spanning trees for $P_n$. Let $\{T_1,T_2, \cdots T_{n-1}\}$  be NBC spanning trees of $P_n$ ordered lexicographically. We consider the root at the vertex $v_{root}$, where the edges ordered $1$ and $n$ are incident. At each homological grading $i$ there are exactly two critical states $H_c^{T_{i}^+},H_c^{T_{i}^-}$ corresponding to each NBC spanning tree $T_i$. This is because each tree has one more internally inactive edge than its immediate predecessor in the lexicographic ordering. \\

            \textbf{For $i=0$}: Then the trees $T_1=[1,2,\cdots ,n-1]$ and $T_2=[1,2,\cdots,n-2,n]$ are in the homological grading $0$ and $1$ respectively(see \autoref{Cycle smallest trees}). Then if $n$ is even then by \autoref{differential} we have that $[T_{1}^+:T_{2}^-]_{\cstBiDiff{*}{*}{2}} = 0$. If $n$ is odd then, $[T_{1}^+:T_{2}^-]_{\cstBiDiff{*}{*}{2}}=2$. As a consequence, we have that
            \[ \cstBiHom{0}{\ast}{2}{P_n}=\mathbb{Z} \oplus \mathbb{Z}  \text{ if n is even }
            \]
            \[
            \cstBiHom{0}{\ast}{2}{P_n}=\mathbb{Z} \text{ if n is odd}
            \]
            \begin{figure}
                \centering
                \resizebox{0.6\textwidth}{!}{
                \begin{tikzpicture}
                    \tikzstyle{new style 0}=[circle, fill=black, minimum size=6pt, inner sep=1pt]
                    \tikzstyle{new style 1}=[circle, fill=red, minimum size=6pt, inner sep=1pt, label=left:$v_{root}$]
                    \tikzstyle{none}=[circle, draw=none, minimum size=6pt, inner sep=3pt]

                    % First figure (compressed)
                    \node [style=new style 1] (0) at (-6, 5) {};
                    \node [style=new style 0] (1) at (-4.8, 5.6) {};
                    \node [style=new style 0] (2) at (-3.6, 5) {};
                    \node [style=new style 0] (3) at (-6, 4) {};
                    \node [style=none] (4) at (-3.6, 4) {};
                    \node [style=none] (5) at (-3.6, 3) {};
                    \node [style=new style 0] (7) at (-4.8, 2.4) {};
                    \node [style=new style 0] (8) at (-6, 3) {};
                    \node [style=none] (9) at (-5.3, 5.6) {$1$};
                    \node [style=none] (10) at (-3.9, 5.6) {$2$};
                    \node [style=none] (12) at (-5.3, 2.4) {$n-2$};
                    \node [style=none] (13) at (-6.6, 3.5) {$n-1$};
                    \node [style=none] (15) at (-5.3, 4.6) {$d$};

                    % Second figure (compressed)
                    \node [style=new style 1] (16) at (0.5, 5) {};
                    \node [style=new style 0] (17) at (1.7, 5.6) {};
                    \node [style=new style 0] (18) at (2.9, 5) {};
                    \node [style=new style 0] (19) at (0.5, 4) {};
                    \node [style=none] (20) at (2.9, 4) {};
                    \node [style=none] (21) at (2.9, 3) {};
                    \node [style=new style 0] (22) at (1.7, 2.4) {};
                    \node [style=new style 0] (23) at (0.5, 3) {};
                    \node [style=none] (24) at (1.1, 4.6) {$D$};
                    \node [style=none] (25) at (1.1, 5.6) {$1$};
                    \node [style=none] (26) at (2.5, 5.6) {$2$};
                    \node [style=none] (27) at (2.5, 2.4) {$n-3$};
                    \node [style=none] (28) at (0.9, 2.4) {$n-2$};
                    \node [style=none] (29) at (0.1, 3.5) {$n-1$};
                    \node [style=none] (30) at (-0.1, 4.6) {$n$};

                    % Edges for first figure
                    \draw (0) to (1);
                    \draw (1) to (2);
                    \draw (2) to (4.center);
                    \draw (4.center) to (5.center);
                    \draw (5.center) to (7);
                    \draw (8) to (7);
                    \draw (3) to (8);

                    % Edges for second figure
                    \draw (16) to (17);
                    \draw (17) to (18);
                    \draw (18) to (20.center);
                    \draw (20.center) to (21.center);
                    \draw (21.center) to (22);
                    \draw (16) to (19);
                    \draw (23) to (22);

                    % Arrow between figures (shorter)
                    \draw [thick, ->] (-2.5, 4) -- (-0.5, 4) node[midway, above] {$\partial_{ST}$};

                \end{tikzpicture}
                }
                \caption{The trees $T_1$ and $T_2$}
                \label{Cycle smallest trees}
            \end{figure}
            \textbf{For $i>0$, $n-i \geq 2$ and $n-i$ is even}: The tree that corresponds to critical states at the $i$-th homological grading is $T_i=[1,2, \cdots,n-i-1, \widehat{n-i},n-i+1, \cdots n]$, here the $\widehat{n-i}$ denotes the tree obtained by removal of $n-i$-th ordered edge. Let $T_{i-1}, T_{i+1}$ be the trees corresponding to the critical states at the $i-1$-th and $i+1$-th grading. Then by \Autoref{differential} it follows that, $[T_{i-1}^+:T_{i}^-]_{\cstBiDiff{*}{*}{2}}=\pm 2$ and $[T_{i}^+:T_{i+1}^-]_{\cstBiDiff{*}{*}{2}}=0$. Thus, we have
            $$\cstBiHom{i}{\ast}{2}{P_n}= \mathbb{Z}_2 \oplus \mathbb{Z}$$
            \textbf{For $i>0$, $n-i \geq 2$ and $n-i$ is odd}: Similarly as above we can easily verify that, $[T_{i}^+:T_{i+1}^-]_{\cstBiDiff{*}{*}{2}}=\pm 2$. Hence, we have,
            $$
            \cstBiHom{i}{\ast}{2}{P_n}= \mathbb{Z}
            $$
            \begin{remark}
                This computation agrees with the known results on the chromatic homology of polygon graphs over $\A_2$ \cite[see Theorem 20]{HelmeGuizonPrzytyckiRong2006_torsion}, where the cohomology of the polygon graphs is computed. In particular, the $\mathbb{Z}_2$-torsion in the homological gradings $i$ with $n-i$ even, $1 \leq i \leq n-2$, is in accordance with the characterization of the graphs whose $\A_2$-chromatic homology contains torsion obtained in \cite{HelmeGuizonPrzytyckiRong2006_torsion}, and the support of the homology on the two diagonals $i+j=n-1$ and $i+j=n$ is in accordance with \cite[Corollary 16]{HelmeGuizonPrzytyckiRong2006_torsion} and \cite{ChmutovChmutovRong2008}.
            \end{remark}
            % Point 62 addressed.
        \end{example}

        \begin{figure}
            \centering
            \begin{tikzpicture}[vertex/.style={circle, draw, inner sep=0pt, minimum size=5pt, fill=black}, edge/.style={draw}, scale=0.9]
                \node[vertex,fill=red,color=red, label=above:$v_{root}$] (a) at (0,2) {};
                \node[vertex] (b) at (0,0.5) {};
                \node[vertex] (d) at (0,-1) {};
                \node[vertex] (e) at (-1.5,1) {};
                \node[vertex] (f) at (-1.5,0) {};
                \node[vertex] (g) at (1.5,0) {};
                \node[vertex] (h) at (1.5,1) {};
                \path[edge] (a) to node[left] {$7$} (b);
                \path[edge] (b) to node[left] {$8$} (d);
                \path[edge] (a) to[out=180,in=70] node[above] {$2$} (e);
                \path[edge] (e) to node[left] {$1$} (f);
                \path[edge] (f) to[out=290,in=180] node[below] {$6$} (d);
                \path[edge] (a) to[out=0,in=100]node[above]{3} (h);
                \path[edge] (g) to[out=270,in=0] node[below] {$5$} (d);
                \path[edge] (g) to node[right] {$4$} (h);
            \end{tikzpicture}
            \caption{The Graph $G$ with a fixed ordering on its edges}
            \label{mainexample1}
        \end{figure}
        \begin{example}
            Let us compute the spanning tree homology of the graph $G$ given in \autoref{mainexample1} with a fixed ordering on its edges. We note that computation of the torsion part of chromatic homology for this graph was previously done in Example 3 of \cite{sazdanovic2018patternskhovanovlinkchromatic}.  Then all possible NBC spanning trees of $G$ are listed in \autoref{NBC spanning trees}. The spanning tree chain complex of $G$ is given in \autoref{chain complex table for G}. The computations of the incidences of the spanning tree differential are given in \autoref{Incidence table for G}. Observe that all the torsions that appear in the homology groups come from the incidences $[T^+:{T'}^-]_{\cstDiff{\ast,\ast}{2}}$, for NBC spanning trees $T,T'$ of $G$. One such pictorial description of the incidences from which torsion appears is given in \autoref{Incidence table for G}. Finally, the chromatic homology computed from this differential description is given in \autoref{homology table 1}.

            \begin{table}

                \centering
                \resizebox{0.8\textwidth}{!}
                {
                \begin{tabular}{|c|c|c|c|}
                    \hline
                    \begin{tikzpicture}[vertex/.style={circle, draw, inner sep=0pt, minimum size=6pt, fill=black}, edge/.style={draw}, scale=1.2]
                        \node[vertex,fill=red,color=red,label=above:$v_{root}$] (a) at (0,2) {};
                        \node[vertex] (b) at (0,0.5) {};
                        \node[vertex] (d) at (0,-1) {};
                        \node[vertex] (e) at (-1.5,1) {};
                        \node[vertex] (f) at (-1.5,0) {};
                        \node[vertex] (g) at (1.5,0) {};
                        \node[vertex] (h) at (1.5,1) {};
                        \path[edge] (a) to node[left] {$7$} (b);
                        \path[edge] (a) to[out=180,in=70] node[above] {$2$} (e);
                        \path[edge] (e) to node[left] {$1$} (f);
                        \path[edge] (a) to[out=0,in=100]node[above]{$3$} (h);
                        \path[edge] (g) to[out=270,in=0] node[below] {$5$} (d);
                        \path[edge] (g) to node[right] {$4$} (h);
                        \node at (0,-2) {$T_1, W(T_1)=LLLLLdLd$};
                    \end{tikzpicture} &	\begin{tikzpicture}[vertex/.style={circle, draw, inner sep=0pt, minimum size=6pt, fill=black}, edge/.style={draw}, scale=1.2]
                        \node[vertex,fill=red,color=red,label=above:$v_{root}$] (a) at (0,2) {};
                        \node[vertex] (b) at (0,0.5) {};
                        \node[vertex] (d) at (0,-1) {};
                        \node[vertex] (e) at (-1.5,1) {};
                        \node[vertex] (f) at (-1.5,0) {};
                        \node[vertex] (g) at (1.5,0) {};
                        \node[vertex] (h) at (1.5,1) {};
                        \path[edge] (b) to node[left] {$8$} (d);
                        \path[edge] (a) to[out=180,in=70] node[above] {$2$} (e);
                        \path[edge] (e) to node[left] {$1$} (f);
                        \path[edge] (a) to[out=0,in=100]node[above]{3} (h);
                        \path[edge] (g) to[out=270,in=0] node[below] {$5$} (d);
                        \path[edge] (g) to node[right] {$4$} (h);
                        \node at (0,-2) {$T_2,W(T_2)=LLLLLddD$};
                    \end{tikzpicture}
                    & 	\begin{tikzpicture}[vertex/.style={circle, draw, inner sep=0pt, minimum size=6pt, fill=black}, edge/.style={draw}, scale=1.2]
                        \node[vertex,fill=red,color=red,label=above:$v_{root}$] (a) at (0,2) {};
                        \node[vertex] (b) at (0,0.5) {};
                        \node[vertex] (d) at (0,-1) {};
                        \node[vertex] (e) at (-1.5,1) {};
                        \node[vertex] (f) at (-1.5,0) {};
                        \node[vertex] (g) at (1.5,0) {};
                        \node[vertex] (h) at (1.5,1) {};
                        \path[edge] (a) to node[left] {$7$} (b);
                        \path[edge] (a) to[out=180,in=70] node[above] {$2$} (e);
                        \path[edge] (e) to node[left] {$1$} (f);
                        \path[edge] (f) to[out=290,in=180] node[below] {$6$} (d);
                        \path[edge] (a) to[out=0,in=100]node[above]{$3$} (h);
                        \path[edge] (g) to node[right] {$4$} (h);
                        \node at (0,-2) {$T_3,W(T_3)=LLLLdDLd$};
                    \end{tikzpicture}
                    &
                    \begin{tikzpicture}[vertex/.style={circle, draw, inner sep=0pt, minimum size=6pt, fill=black}, edge/.style={draw}, scale=1.2]
                        \node[vertex,fill=red,color=red,label=above:$v_{root}$] (a) at (0,2) {};
                        \node[vertex] (b) at (0,0.5) {};
                        \node[vertex] (d) at (0,-1) {};
                        \node[vertex] (e) at (-1.5,1) {};
                        \node[vertex] (f) at (-1.5,0) {};
                        \node[vertex] (g) at (1.5,0) {};
                        \node[vertex] (h) at (1.5,1) {};
                        \path[edge] (b) to node[left] {$8$} (d);
                        \path[edge] (a) to[out=180,in=70] node[above] {$2$} (e);
                        \path[edge] (e) to node[left] {$1$} (f);
                        \path[edge] (f) to[out=290,in=180] node[below] {$6$} (d);
                        \path[edge] (a) to[out=0,in=100]node[right]{$3$} (h);
                        \path[edge] (h) to node[right]{$4$} (g);
                        \node at (0,-2) {$T_4,W(T_4)=LLLLdDdD$};
                    \end{tikzpicture}

                    \\
                    \hline
                    \begin{tikzpicture}[vertex/.style={circle, draw, inner sep=0pt, minimum size=6pt, fill=black}, edge/.style={draw}, scale=1.2]
                        \node[vertex,fill=red,color=red,label=above:$v_{root}$] (a) at (0,2) {};
                        \node[vertex] (b) at (0,0.5) {};
                        \node[vertex] (d) at (0,-1) {};
                        \node[vertex] (e) at (-1.5,1) {};
                        \node[vertex] (f) at (-1.5,0) {};
                        \node[vertex] (g) at (1.5,0) {};
                        \node[vertex] (h) at (1.5,1) {};
                        \path[edge] (a) to node[left] {$7$} (b);
                        \path[edge] (b) to node[left] {$8$} (d);
                        \path[edge] (a) to[out=180,in=70] node[above] {$2$} (e);
                        \path[edge] (e) to node[left] {$1$} (f);
                        \path[edge] (a) to[out=0,in=100]node[above]{3} (h);
                        \path[edge] (g) to node[right] {$4$} (h);
                        \node at (0,-2) {$T_5,W(T_5)=LLLLddDD$};
                    \end{tikzpicture} & \begin{tikzpicture}[vertex/.style={circle, draw, inner sep=0pt, minimum size=6pt, fill=black}, edge/.style={draw}, scale=1.2]
                        \node[vertex,fill=red,color=red,label=above:$v_{root}$] (a) at (0,2) {};
                        \node[vertex] (b) at (0,0.5) {};
                        \node[vertex] (d) at (0,-1) {};
                        \node[vertex] (e) at (-1.5,1) {};
                        \node[vertex] (f) at (-1.5,0) {};
                        \node[vertex] (g) at (1.5,0) {};
                        \node[vertex] (h) at (1.5,1) {};
                        \path[edge] (a) to node[left] {$7$} (b);
                        \path[edge] (a) to[out=180,in=70] node[above] {$2$} (e);
                        \path[edge] (e) to node[left] {$1$} (f);
                        \path[edge] (f) to[out=290,in=180] node[below] {$6$} (d);
                        \path[edge] (a) to[out=0,in=100]node[above]{3} (h);
                        \path[edge] (g) to[out=270,in=0] node[below] {$5$} (d);
                        \node at (0,-2) {$T_6,W(T_6)=LLLdDDLd$};
                    \end{tikzpicture}
                    &	\begin{tikzpicture}[vertex/.style={circle, draw, inner sep=0pt, minimum size=6pt, fill=black}, edge/.style={draw}, scale=1.2]
                        \node[vertex,fill=red,color=red,label=above:$v_{root}$] (a) at (0,2) {};
                        \node[vertex] (b) at (0,0.5) {};
                        \node[vertex] (d) at (0,-1) {};
                        \node[vertex] (e) at (-1.5,1) {};
                        \node[vertex] (f) at (-1.5,0) {};
                        \node[vertex] (g) at (1.5,0) {};
                        \node[vertex] (h) at (1.5,1) {};
                        \path[edge] (b) to node[left] {$8$} (d);
                        \path[edge] (a) to[out=180,in=70] node[above] {$2$} (e);
                        \path[edge] (e) to node[left] {$1$} (f);
                        \path[edge] (f) to[out=290,in=180] node[below] {$6$} (d);
                        \path[edge] (a) to[out=0,in=100]node[above]{3} (h);
                        \path[edge] (g) to[out=270,in=0] node[below] {$5$} (d);
                        \node at (0,-2) {$T_7,W(T_7)=LLLdDDdD$};
                    \end{tikzpicture}
                    &	\begin{tikzpicture}[vertex/.style={circle, draw, inner sep=0pt, minimum size=6pt, fill=black}, edge/.style={draw}, scale=1.2]
                        \node[vertex,fill=red,color=red,label=above:$v_{root}$] (a) at (0,2) {};
                        \node[vertex] (b) at (0,0.5) {};
                        \node[vertex] (d) at (0,-1) {};
                        \node[vertex] (e) at (-1.5,1) {};
                        \node[vertex] (f) at (-1.5,0) {};
                        \node[vertex] (g) at (1.5,0) {};
                        \node[vertex] (h) at (1.5,1) {};
                        \path[edge] (a) to node[left] {$7$} (b);
                        \path[edge] (b) to node[left] {$8$} (d);
                        \path[edge] (a) to[out=180,in=70] node[above] {$2$} (e);
                        \path[edge] (e) to node[left] {$1$} (f);
                        \path[edge] (a) to[out=0,in=100]node[above]{$3$} (h);
                        \path[edge] (g) to[out=270,in=0] node[below] {$5$} (d);
                        \node at (0,-2) {$T_8,W(T_8)=LLLdDdDD$};

                    \end{tikzpicture}
                    \\
                    \hline
                    \begin{tikzpicture}[vertex/.style={circle, draw, inner sep=0pt, minimum size=6pt, fill=black}, edge/.style={draw}, scale=1.2]
                        \node[vertex,fill=red,color=red,label=above:$v_{root}$] (a) at (0,2) {};
                        \node[vertex] (b) at (0,0.5) {};
                        \node[vertex] (d) at (0,-1) {};
                        \node[vertex] (e) at (-1.5,1) {};
                        \node[vertex] (f) at (-1.5,0) {};
                        \node[vertex] (g) at (1.5,0) {};
                        \node[vertex] (h) at (1.5,1) {};
                        \path[edge] (a) to node[left] {$7$} (b);
                        \path[edge] (a) to[out=180,in=70] node[above] {$2$} (e);
                        \path[edge] (e) to node[left] {$1$} (f);
                        \path[edge] (f) to[out=290,in=180] node[below] {$6$} (d);
                        \path[edge] (g) to[out=270,in=0] node[below] {$5$} (d);
                        \path[edge] (g) to node[right] {$4$} (h);
                        \node at (0,-2) {$T_9,W(T_9)=LLdDDDLd$};
                    \end{tikzpicture}&

                    \begin{tikzpicture}[vertex/.style={circle, draw, inner sep=0pt, minimum size=6pt, fill=black}, edge/.style={draw}, scale=1.2]
                        \node[vertex,fill=red,color=red,label=above:$v_{root}$] (a) at (0,2) {};
                        \node[vertex] (b) at (0,0.5) {};
                        \node[vertex] (d) at (0,-1) {};
                        \node[vertex] (e) at (-1.5,1) {};
                        \node[vertex] (f) at (-1.5,0) {};
                        \node[vertex] (g) at (1.5,0) {};
                        \node[vertex] (h) at (1.5,1) {};
                        \path[edge] (b) to node[left] {$8$} (d);
                        \path[edge] (e) to node[left] {$1$} (f);
                        \path[edge] (f) to[out=290,in=180] node[below] {$6$} (d);
                        \path[edge] (a) to[out=180,in=70] node[above] {$2$} (e);
                        \path[edge] (g) to[out=270,in=0] node[below] {$5$} (d);
                        \path[edge] (g) to node[right] {$4$} (h);
                        \node at (0,-2) {$T_{10},W(T_{10})=LLdDDDdD$};
                    \end{tikzpicture}
                    &	\begin{tikzpicture}[vertex/.style={circle, draw, inner sep=0pt, minimum size=6pt, fill=black}, edge/.style={draw}, scale=1.2]
                        \node[vertex,fill=red,color=red,label=above:$v_{root}$] (a) at (0,2) {};
                        \node[vertex] (b) at (0,0.5) {};
                        \node[vertex] (d) at (0,-1) {};
                        \node[vertex] (e) at (-1.5,1) {};
                        \node[vertex] (f) at (-1.5,0) {};
                        \node[vertex] (g) at (1.5,0) {};
                        \node[vertex] (h) at (1.5,1) {};
                        \path[edge] (a) to node[left] {$7$} (b);
                        \path[edge] (e) to node[left] {$1$} (f);
                        \path[edge] (f) to[out=290,in=180] node[below] {$6$} (d);
                        \path[edge] (a) to[out=0,in=100]node[above]{$3$} (h);
                        \path[edge] (g) to[out=270,in=0] node[below] {$5$} (d);
                        \path[edge] (g) to node[right] {$4$} (h);
                        \node at (0,-2) {$T_{11},W(T_{11})=LdDDDDLd$};
                    \end{tikzpicture} & \begin{tikzpicture}[vertex/.style={circle, draw, inner sep=0pt, minimum size=6pt, fill=black}, edge/.style={draw}, scale=1.2]
                        \node[vertex,fill=red,color=red,label=above:$v_{root}$] (a) at (0,2) {};
                        \node[vertex] (b) at (0,0.5) {};
                        \node[vertex] (d) at (0,-1) {};
                        \node[vertex] (e) at (-1.5,1) {};
                        \node[vertex] (f) at (-1.5,0) {};
                        \node[vertex] (g) at (1.5,0) {};
                        \node[vertex] (h) at (1.5,1) {};
                        \path[edge] (b) to node[left] {$8$} (d);
                        \path[edge] (e) to node[left] {$1$} (f);
                        \path[edge] (f) to[out=290,in=180] node[below] {$6$} (d);
                        \path[edge] (a) to[out=0,in=100]node[above]{3} (h);
                        \path[edge] (g) to[out=270,in=0] node[below] {$5$} (d);
                        \path[edge] (g) to node[right] {$4$} (h);
                        \node at (0,-2) {$T_{12},W(T_{12})=LdDDDDdD$};
                    \end{tikzpicture}
                    \\
                    \hline
                    \begin{tikzpicture}[vertex/.style={circle, draw, inner sep=0pt, minimum size=6pt, fill=black}, edge/.style={draw}, scale=1.2]
                        \node[vertex,fill=red,color=red,label=above:$v_{root}$] (a) at (0,2) {};
                        \node[vertex] (b) at (0,0.5) {};
                        \node[vertex] (d) at (0,-1) {};
                        \node[vertex] (e) at (-1.5,1) {};
                        \node[vertex] (f) at (-1.5,0) {};
                        \node[vertex] (g) at (1.5,0) {};
                        \node[vertex] (h) at (1.5,1) {};
                        \path[edge] (b) to node[left] {$8$} (d);
                        \path[edge] (e) to node[left] {$1$} (f);
                        \path[edge] (a) to node[left] {$7$} (b);

                        \path[edge] (f) to[out=290,in=180] node[below] {$6$} (d);
                        \path[edge] (a) to[out=0,in=100]node[above]{$3$} (h);
                        \path[edge] (g) to node[right] {$4$} (h);
                        \node at (0,-2) {$T_{13},W(T_{13})=LdLLdDDD$};
                    \end{tikzpicture}
                    & \begin{tikzpicture}[vertex/.style={circle, draw, inner sep=0pt, minimum size=6pt, fill=black}, edge/.style={draw}, scale=1.2]
                        \node[vertex,fill=red,color=red,label=above:$v_{root}$] (a) at (0,2) {};
                        \node[vertex] (b) at (0,0.5) {};
                        \node[vertex] (d) at (0,-1) {};
                        \node[vertex] (e) at (-1.5,1) {};
                        \node[vertex] (f) at (-1.5,0) {};
                        \node[vertex] (g) at (1.5,0) {};
                        \node[vertex] (h) at (1.5,1) {};
                        \path[edge] (a) to node[left] {$7$} (b);
                        \path[edge] (b) to node[left] {$8$} (d);
                        \path[edge] (e) to node[left] {$1$} (f);
                        \path[edge] (f) to[out=290,in=180] node[below] {$6$} (d);
                        \path[edge] (a) to[out=0,in=100]node[above]{3} (h);
                        \path[edge] (g) to[out=270,in=0] node[below] {$5$} (d);
                        \node at (0,-2) {$T_{14},W(T_{14})=LdLdDDDD$};
                    \end{tikzpicture}
                    & & \\
                    \hline
                \end{tabular}
                }
                \vspace{0.3cm}
                \caption{NBC Spanning trees of the graph G with their activity word function ordered according to the lexicographic ordering}
                \label{NBC spanning trees}
            \end{table}
            \vspace{-5 cm}
            \begin{table}
                \setlength{\tabcolsep}{2mm}
                \def\arraystretch{1.5}
                \centering
                \begin{tabular}{|c|c|c|c|c|c|c|}
                    \hline
                    \backslashbox{$j$}{$i$} & $0$ & $1$ & $2$ & $3$ & $4$ & $5$ \\ \hline
                    $7$  & $T_1^-$ &  &  & & &  \\ \hline
                    $6$ & $T_1^+$& $T_2^-$,$T_3^-$& & & &  \\ \hline
                    $5$ & &$T_2^+$,$T_3^+$ &$T_4^-$,$T_5^-$,$T_6^-$ & & & \\ \hline
                    $4$ & & &$T_4^+$,$T_5^+$,$T_6^+$ &$T_7^-$,$T_9^-$,$T_ 8^-, T_{13}^-$ & &  \\ \hline
                    $3$ & & & &$T_9^+$,$T_8^+$,$T_7^+,T_{13}^+$ &$T_{14}^-,T_{11}^-, T_{10}^-$ & \\ \hline
                    $2$ & & &  & &$T_{14}^+,T_{11}^+,T_{10}^+$ & $T_{12}^-$  \\ \hline
                    $1$ & & & & & &$T_{12}^+$ \\ \hline

                \end{tabular}
                \vspace{0.3cm}
                \caption{Generators of the spanning tree complex of $G$}
                \label{chain complex table for G}
            \end{table}

            \begin{table}[ht]
                \centering
                \setlength{\tabcolsep}{0.01\textwidth}
                \def\arraystretch{1.45}
                \begin{tabular}{|l|l|}
                    \hline
                    $T_1^- \mapsto 0$ & $T_1^+ \mapsto -2 T_2^-$ \\ \hline
                    $T_2^- \mapsto 0$,$T_3^- \mapsto 0$ & $T_2^+ \mapsto 0 $, $T_3^+ \mapsto 2T_6^-$\\ \hline
                    $T_4^- \mapsto 0$, $T_5^- \mapsto 0$, $T_6^- \mapsto 0$ & \begin{tabular}{l}
                        $T_4^+ \mapsto 2T_7^- - 2T_{13}^-$\\ $T_5^+ \mapsto 2T_8^- + 2T_{13}^-$,  $T_6^+ \mapsto 0$

                    \end{tabular}\\ \hline
                    \begin{tabular}{l}
                        $T_8^- \mapsto 0$, $T_9^- \mapsto 0$\\
                        $T_7^- \mapsto 0, T_{13}^- \mapsto 0$
                    \end{tabular} &
                    \begin{tabular}{l}
                        $T_8^+ \mapsto -2T_{14}^-$, $T_9^+ \mapsto 2T_{11}^-$\\
                        $T_7^+ \mapsto 2T_{14}^-$,$T_{13}^+ \mapsto 2T_{14}^- $

                    \end{tabular} \\ \hline
                    $T_{14}^- \mapsto 0$, $T_{11}^- \mapsto 0$, $T_{10}^- \mapsto 0$ & $T_{14}^+ \mapsto 0 $, $T_{11}^+ \mapsto 0, T_{10}^+ \mapsto 2T_{12}^-$ \\ \hline
                    $T_{12}^- \mapsto 0$ & $T_{12}^+ \mapsto 0$ \\ \hline
                \end{tabular}
                \vspace{0.3cm}
                \caption{Incidence Table for $\cstCom{}{2}{G}$}
                \label{Incidence table for G}
            \end{table}
            \begin{figure}
                \centering
                % Requires \usepackage{tikz} and \usetikzlibrary{arrows.meta}.
% Each x or 1 labels a connected component of the solid-edge subgraph.
\begingroup
\tikzset{
  apvertex/.style={circle,fill=black,inner sep=0pt,minimum size=4pt},
  aproot/.style={apvertex,fill=red},
  apabsent/.style={densely dotted,line width=.65pt},
  appresent/.style={line width=.75pt},
  apdiff/.style={-{Stealth[length=2.1mm,width=1.5mm]},line width=.7pt},
  apmatch/.style={apdiff,draw=red},
  aplabel/.style={inner sep=2pt,font=\normalsize},
  apinterval/.style={densely dashed,line width=.55pt,rounded corners=2mm}
}
% Arguments: shift, additional solid edge (0=none,1=upper-left,
% 2=left vertical,3=lower-left), root, upper-left component,
% lower-left component, right component, and merged-left component.
\newcommand{\APexampleState}[7]{%
  \begin{scope}[shift={#1}]
    \coordinate (a) at (0,1.35);
    \coordinate (b) at (-.8,.65);
    \coordinate (c) at (-.8,-.65);
    \coordinate (d) at (0,-1.35);
    \coordinate (e) at (.8,-.65);
    \coordinate (f) at (.8,.65);
    \coordinate (o) at (0,0);
    \draw[appresent] (a)--(o)--(d)--(e);
    \draw[apabsent] (a)--(f)--(e);
    \ifnum#2=1
      \draw[appresent] (a)--(b);
    \else
      \draw[apabsent] (a)--(b);
    \fi
    \ifnum#2=2
      \draw[appresent] (b)--(c);
    \else
      \draw[apabsent] (b)--(c);
    \fi
    \ifnum#2=3 \draw[appresent] (c)--(d);\fi
    \node[aproot] at (a) {};
    \foreach \v in {b,c,d,e,f,o} \node[apvertex] at (\v) {};
    \node[aplabel,above] at (a) {$#3$};
    \node[aplabel,left] at (b) {$#4$};
    \node[aplabel,left] at (c) {$#5$};
    \node[aplabel,right] at (f) {$#6$};
    \ifnum#2=2
      \node[aplabel,left] at (-.8,0) {$#7$};
    \fi
  \end{scope}%
}
\begin{tikzpicture}[x=1cm,y=1cm]
  % Interval [R(T_7),T_7]: the first five states of the longer path.
  \draw[apinterval]
    (-1.35,1.95)--(11.35,1.95)--(11.35,-6.95)
    --(3.25,-6.95)--(3.25,-2.5)--(-1.35,-2.5)--cycle;
  \node[aplabel,above] at (5,2.02) {};
  % The target belongs to the second spanning-tree interval.
  \draw[apinterval] (-1.35,-3.05) rectangle (1.35,-7.15);
  \node[aplabel,below] at (0,-7.2) {};

  % Source, first differential, first reversed matching.
  \APexampleState{(0,0)}{0}{1}{x}{x}{x}{}
  \APexampleState{(5,0)}{1}{x}{}{x}{x}{}
  \APexampleState{(10,0)}{0}{x}{1}{x}{x}{}
  \node[aplabel,below] at (0,-1.5) {$T_7^+$};
  \draw[apdiff] (1.25,0)--node[aplabel,above] {$d_3^{e_2}$} (3.75,0);
  \draw[apmatch] (6.25,0)--node[aplabel,above,text=black] {$\mathcal{M}$} (8.75,0);

  % Continue down, then read the lower row from right to left.
  \APexampleState{(10,-5)}{2}{x}{}{}{x}{x}
  \APexampleState{(5,-5)}{0}{x}{x}{1}{x}{}
  \APexampleState{(0,-5)}{3}{x}{x}{}{x}{}
  \node[aplabel,below] at (0,-6.5) {$T_{14}^-$};
  \draw[apdiff] (10,-1.85)--node[aplabel,right] {$d_3^{e_1}$} (10,-3.15);
  \draw[apmatch] (8.75,-5)--node[aplabel,above,text=black] {$\mathcal{M}$} (6.25,-5);
  \draw[apdiff] (3.75,-5)--node[aplabel,above] {$d_3^{e_6}$} (1.55,-5);

  % Direct path from the same source to the same target.
  % Routing outside the intervals avoids intersecting the longer path.
  \draw[apdiff,rounded corners=2mm]
    (-1.08,0)--(-2.25,0)
    --node[aplabel,left] {$d_3^{e_6}$} (-2.25,-5)--(-1.08,-5);
\end{tikzpicture}
\endgroup
                % \resizebox{0.45\textwidth}{!}{\includesvg{alternatingpathex1.svg}} % Adjust factor as needed
                \caption{Two alternating paths between $T_{7}^+$ to $T_{14}^-$ giving rise to torsion in the homology}
                \label{fig:alternatingpathex1}
            \end{figure}
            \begin{table}
                \centering
                \setlength{\tabcolsep}{0.01\textwidth}
                \def\arraystretch{1.45}
                \begin{tabular}{|c|c|c|c|c|c|c|}
                    \hline
                    homological grading $i$ &\hspace{0.4 cm}  0 \hspace{0.4 cm} & 1 & 2 & 3 & 4 & 5 \\
                    \hline
                    $\cstBiHom{i}{*}{2}{G}$ & $\mathbb{Z}$ & $\mathbb{Z}^2 \oplus \mathbb{Z}_2$ & $\mathbb{Z}^3 \oplus \mathbb{Z}_2$ & $\mathbb{Z}^4 \oplus \mathbb{Z}_2^2$ & $\mathbb{Z}^3 \oplus \mathbb{Z}_2^2$ & $\mathbb{Z}\oplus \mathbb{Z}_2$\\
                    \hline

                \end{tabular}
                \vspace{0.3 cm}
                \caption{Homology table for $\cstCom{}{2}{G}$}
                \label{homology table 1}
            \end{table}
        \end{example}

    \end{section}
    \newpage
    \begin{section}{Spanning tree model for chromatic homology over \texorpdfstring{$\A_m$}{Am} algebra and its applications }

        Recall that the lexicographical shelling decomposes the non-broken circuit complex $\NBC(G)$ into the same interval partitions $[\R(T_i), T_i]$ where $T_i \in \F(\NBC(G))$. Furthermore, the general method of applying algebraic discrete Morse theory remains unchanged. Namely, we define an acyclic matching on the weighted directed graph associated to the chromatic complex over the algebra $\A_m$ of individual trees. We then extend it to these interval pieces. % Point 63 addressed.
        The fundamental property that there is at most one alternating path between two enhanced spanning subgraphs of a fixed $[\R(T_i), T_i]$ as in \Autoref{fundamental}, \Autoref{atmost one path cor} also holds in the general case.\\

        However, the generalization from $\A_2$ to $\A_m$ alters the algebraic interactions of the enhanced spanning subgraphs, requiring new arguments for the combinatorial description of the differential $\cstDiff{}{m}$. Because $\A_m$ has $m-1$ generators, the number of critical enhanced spanning subgraphs per interval $[\R(T_i), T_i]$ increases from two to $m(m-1)^{|\pi_0(H_c^T)|-1}$, in which case tracking the enhancement of the components becomes more involved than in the case of $\A_2$. However, we introduce new notation for the generators of the spanning tree complex over the $\A_m$ algebra and paths in spanning trees beginning with \textit{non-vanishing}(See \Autoref{non-vanishing edge}) % Point 64 addressed.
        edges to keep track of the enhancements of the connected components of a general enhanced spanning subgraph along an alternating path, which helps in describing $\cstDiff{}{m}$ combinatorially.\\

        % Recall that the lexicographical shelling of the non-broken circuit complex $\NBC(G)$ yields a decomposition into intervals of the form $[\R(T_i), T_i]$, indexed by the facets $T_i \in \F(\NBC(G))$. This combinatorial structure provides a natural framework in which the machinery of algebraic discrete Morse theory can be applied locally. In particular, one constructs an acyclic matching on the weighted directed graph associated to the chromatic complex over the algebra $\A_m$ by first defining the matching on the minimal lexicographically ordered spanning tree and then extending it coherently to each interval $[\R(T_i), T_i]$. So the fundamental approach to define the matching and the spanning tree complex remains unaltered.

        % A key feature underlying this construction is its robustness: the essential combinatorial properties established in the case of $\A_2$ persist in this broader setting. Notably, within any fixed interval $[\R(T_i), T_i]$, there exists at most one alternating path between any two enhanced spanning subgraphs, as established in \Autoref{fundamental} and \Autoref{atmost one path cor}. This ensures that the acyclicity of the matching continues to hold without modification at the structural level.

        Throughout this section, we denote by $\chBiCom{\ast}{\ast}{m}{G}$ the chromatic complex defined over the algebra $\A_m$, and by $H_{*,*}^{\A_m}(G)$ its associated homology. Our approach follows the same overarching strategy as in the $\A_2$ case: we construct an acyclic matching on the weighted directed graph $\G(\chCom{\ast,\ast}{m}{G})$ corresponding to the chain complex $\chBiCom{*}{*}{m}{G}$, building upon the constructions developed in Subsection \ref{Matching on tree} % Point 65 addressed.
        and \Autoref{thm: Matching on G for A_2}.

        It is not hard to see that the acyclic matching $\M_{\nbc}$ on the chromatic complex over the $\A_2$ algebra can be generalized to provide an acyclic matching $\M_{\nbc,m}$ for the chromatic complex over the $\A_m$ algebra for any $m \geq 2$. We provide a proof for completeness.

        \begin{theorem}{\label{Matching on T am}}
            Let $T$ be a tree with a fixed ordering on its edge set and let $v_{root}$ be a fixed root of $T$. For any $m \geq 2$, there exists an acyclic matching $\M_{T,m}$ on the weighted directed graph $\G(\chCom{\ast,\ast}{m}{T})$, with $m{(m-1)}^{|E(T)|}$ critical vertices.
        \end{theorem}
        % Point 67 addressed.
        \begin{proof}
            Recall that the generators of the chromatic complex $\chCom{\ast,\ast}{m}{T}$ are enhanced spanning subgraphs $(H, \epsilon)$ of the tree $T$, where $H$ is a spanning subgraph of $T$ together with a set-theoretic map
            $\epsilon : \pi_0(H) \rightarrow \{1,x, \dots, x^{m-1}\}$. Now similarly to the case $m=2$, let us first provide an acyclic matching $\M_{T,m}$ on the weighted directed graph $\G(\chCom{\ast,\ast}{m}{T})$ for a tree.\\ % Point 69 addressed

            We inductively define the matching $\M_{T,m}$ based on the number of edges of $T$ as follows:

            \begin{enumerate}
                \item \textbf{Base case:} Suppose $T$ has only one edge with one of its end vertices being the distinguished vertex $v_{root}$. The matching in this case is shown in \autoref{base_case_matching}.

                \begin{figure}[ht]
                    \centering
                    \begin{tikzpicture}[vertex/.style={circle,fill,inner sep=2pt},scale=0.75]
                        \node[vertex] [label=above:$v_{root}$,label=below:$x^i$] at (0,0) (a) {};
                        \node[vertex] [label=above:$v_1$,label=below:$1$] at (2,0) (b) {};
                        \draw[->] (3,0.25) -- (4.5,0.25) node[midway,above] {$d$};
                        \draw[<-] (3,-0.25) -- (4.5,-0.25) node[midway,below] {$\M$};
                        \node[vertex] [label=above:$v_{root}$] at (5.5,0) (c) {};
                        \node[vertex] [label=above:$v_1$] at (7.5,0) (d) {};
                        \draw (c) -- (d) node[midway,above] {$x^i$};
                    \end{tikzpicture}
                    \caption{Base case matching for the $\A_m$ algebra}
                    \label{base_case_matching}
                \end{figure}

                Here, $i \in \{0,\cdots,m-1\}$. The critical enhanced spanning subgraphs $(H_c^T, \epsilon_{(i_1,i_2,\cdots,i_k)})$ are of the following type: $H_c$ is the graph with two isolated vertices $\{v_{root},v_1\}$ and $\epsilon_{(i_1,i_2,\cdots,i_k)}(v_{root}) \in \{1,x,\cdots,x^{m-1}\}$ and $\epsilon_{(i_1,i_2,\cdots,i_k)}(v_1) \in \{x,\cdots,x^{m-1}\}$. Moreover, it is clear that this matching is acyclic.\\

                \item \textbf{Inductive step.}
                Assume that the matching has been constructed for every rooted
                tree with $n-1$ edges. Let $T$ have $n>1$ edges and root $v_{root}$.
                Choose the least ordered leaf edge $e_n$ among those whose leaf
                vertex $v_n$ differs from $v_{root}$. Denote by $\mathcal V_n$ the
                collection of all enhanced spanning subgraphs whose underlying graph
                contains $e_n$.

                For $(H,\epsilon)\in\mathcal V_n$, put
                $H'=H\setminus\{e_n\}$. Let $Z$ be the component of $H$ containing
                $e_n$. Removing $e_n$ splits $Z$ into the isolated vertex
                $\{v_n\}$ and the component $Z'=Z\setminus\{v_n\}$. Pair
                $(H,\epsilon)$ with $(H',\epsilon')$, where
                \[
                \epsilon'(\{v_n\})=1,
                \qquad
                \epsilon'(Z')=\epsilon(Z),
                \]
                and $\epsilon'$ agrees with $\epsilon$ on every other component.
                This matches $\mathcal V_n$ bijectively with the states in which
                $e_n$ is absent and $v_n$ has enhancement $1$. Each matching edge
                has coefficient $\pm1$.

                The remaining states have $e_n\notin E(H)$ and
                \[
                \epsilon(\{v_n\})=x^j,
                \qquad 1\leq j\leq m-1.
                \]
                For each fixed $j$, deleting the isolated vertex $v_n$ identifies
                these states with the enhanced spanning subgraphs of
                $T\setminus\{v_n\}$, rooted at $v_{root}$. Lift every matched pair of
                $\M^{n-1}_{T\setminus\{v_n\},m}$ by adjoining the isolated
                vertex $v_n$, with enhancement $x^j$, to both endpoints.
                Leave the lifts of critical states unmatched.

                The union of these lifted pairs and the pairs constructed along
                $e_n$ defines $\M^n_{T,m}$. By induction, its critical
                states have no edges, with arbitrary enhancement on the root
                and positive powers of $x$ on every other vertex, as in
                Equation~(7).

                A schematic picture of the matching $\M^{n}_{T,m}$ is given in \autoref{aminduction}.

                \begin{figure}[ht]
                    \centering
                    \begin{tikzpicture}[vertex/.style={circle,fill,inner sep=2pt},scale=0.7]
                        \node[vertex] [label=above:$x^i$] at (0.5,0.5) (a) {};
                        \node[vertex] [label=below:$1$,label=above:$v_n$] (b) [above right=of a] {};
                        \node [draw,thick,circle through=(a)] at (0,0) {};
                        \node at (0,0) {$H'$} (a);
                        \draw[->] (2,0.25) -- (4,0.25) node[midway,above] {$d$};
                        \draw[<-] (2,-0.75) -- (4,-0.75) node[midway,above] {$\M^n$}; % Inductive step must be mentioned.
                        \node[vertex] at (6.5,0.5) (c) {};
                        \node[vertex] [label=above:$v_n$] (d) [above right=of c] {};
                        \draw (c) -- (d) node[midway,above] {$x^i$};
                        \node [draw,thick,circle through=(c)] at (6,0) {};
                        \node at (6,0) {$H$} (c);
                    \end{tikzpicture}
                    \vspace{4mm}
                    \begin{tikzpicture}[vertex/.style={circle,fill,inner sep=2pt},scale=0.7]
                        \node[vertex] at (0.5,0.5) (a) {};
                        \node[vertex] [label=below:$x^j$,label=above:$v_n$] (b) [above right=of a] {};
                        \node [draw,thick,circle through=(a)] at (0,0) {};
                        \node at (0,0) {$H$} (a);
                        \draw[->] (2,0.25) -- (4,0.25) node[midway,above] {$d$};
                        \draw[<-] (2,-0.75) -- (4,-0.75) node[midway,above] {$\M^{n-1}$}; % should be superscript not subscript to maintain consistency
                        \node[vertex] at (6.5,0.5) (c) {};
                        \node[vertex] [label=above:$v_n$,label=below:$x^j$] (d) [above right=of c] {};
                        \node [draw,thick,circle through=(c)] at (6,0) {};
                        \node at (6,0) {$H'$} (c);
                    \end{tikzpicture}
                    \caption{The matching for any tree. Here, $i \in \{ 0,\cdots,m-1 \}$ and $j \in \{ 1,\cdots,m-1 \}$}
                    \label{aminduction}
                \end{figure}
                % Point 71 and 72 addressed.

                Thus, the critical enhanced spanning subgraphs $(H_c^T, \epsilon_{(i_1,i_2,\cdots,i_k)})$ obtained from $\M^{n}_{T,m}$ are of the following type: $H_c$ is the graph with $n+1$ isolated vertices and $\epsilon_{(i_1,i_2,\cdots,i_k)}$ is defined as follows:

                \begin{equation}\label{am critical enhancement}
                    \epsilon_{(i_1,i_2,\cdots,i_k)}(v_k):= \begin{cases}
                        x^i & v_k = v_{root}, \: i \in \{0,\cdots,m-1\} \\
                        x^j & v_k \neq v_{root}, \: j \in \{1,\cdots,m-1\}
                    \end{cases}
                \end{equation}
            \end{enumerate}

            Thus, we can define $\M_{T,m} = \M^n_{T,m}$ for the chromatic complex of any tree T with $n$ edges. By a similar argument in \Autoref{acyclicity}, it can be shown that the matching $\M_{T,m}$ is acyclic. Moreover, any critical point of the graph $\G^{\M_{T,m}}(\chCom{\ast,\ast}{m}{T})$ is an enhanced spanning subgraph $(H_c^T, \epsilon_{(i_1,i_2,\cdots,i_k)})$, where $H_c$ is the spanning subgraph of $T$ with no edges, with $\epsilon_{(i_1,i_2,\cdots,i_k)}$ being defined as in Equation \eqref{am critical enhancement}. This implies there are $m(m-1)^{|E(T)|}$ critical points. % Point 73 addressed.
        \end{proof}

        \begin{corollary}\label{criticalstates-AN}
            Let $G$ be a connected graph. Then there exists an acyclic matching $\M_{\nbc,m}$ in the weighted directed graph $\G(\nbcCom{m}{G}, F(\NBC(G))$ with the number of critical points in the graph $\G^{\M_{\nbc,m}}(\nbcCom{m}{G})$ being
            \[
            \sum_{T \in {\F
            (\NBC(G))}} m(m-1)^{|\pi_0(H_c^T)|-1}
            \]
        \end{corollary}

        \begin{proof}
            Let $T$ be an NBC spanning tree of $G$ with $\IN(T) \neq \emptyset$. Let the weighted directed graph $\G(\nbcCom{m}{T})$ denote the subgraph of $\G(\nbcCom{m}{G})$ spanned by enhanced subgraphs whose underlying graph is in $[\R(T), T]$. % Point 74 addressed.
            Then the vertices correspond to enhanced spanning subgraphs $(H,\epsilon)$ such that $\IN(T) \subseteq H \subseteq T$. Now one can contract all edges $e \in \IN(T) \subseteq H$ in $H$ to identify every $H$ as a spanning subgraph of a tree $\widetilde{T}$ such that $\IN(\widetilde{T}) = \emptyset$. % Point 76 addressed.
            Under this identification where $\widetilde{T}$ is rooted at the vertex corresponding to the component of $\IN(T)$ containing $v_{root}$ and then we can use $\M_{\widetilde{T},m}$ from \Autoref{Matching on T am} to provide a similar acyclic matching in the graph $\G(\nbcCom{m}{T})$. It is then straightforward to see that there are exactly $m{(m-1)}^{|\pi_0( \IN(T))|-1}$ number of critical points for any NBC spanning tree $T$ of $G$. Now we use the idea of \Autoref{thm: Matching on G for A_2} and $\M_{T,m}$ for each NBC spanning tree of $G$ to provide an acyclic matching $\M_{\nbc,m}$ in $\G(\nbcCom{m}{G})$.
        \end{proof}

        Now, we introduce some notation to denote each critical point of the weighted directed graph $\G^{\M_{\nbc,m}}(\nbcCom{m}{G})$ to keep track of the enhancements. Each critical point of $\G^{\M_{\nbc,m}}(\nbcCom{m}{G})$ lies in some $[\R(T),T]$ for some NBC spanning tree $T$, which is an enhanced spanning subgraph of the form
        \[
        (H^T_c, \epsilon_{(i_1,i_2,\cdots,i_k)})
        \]
        where, $H^T_c$ is a spanning subgraph of $T$ with $E(H_c^T)= \IN(T)$ and
        \[
        \epsilon_{(i_1,\cdots,i_k)}: \pi_0(H_c^T) \to \{1,x,\cdots,x^{m-1}\},
        \]
        where $k = |\pi_0(H_c^T)| = |\IA(T)| + 1$ and $\epsilon_{(i_1,\cdots,i_k)}(Z_{root}) = x^{i_1}$ . Moreover, if $\IA(T) = \{e_1,e_2\cdots,e_{k-1}\}$ with $e_1 <_\M e_2 <_\M \cdots <_\M e_{k-1}$, with respect to the matching order $\M_{\widetilde{T},m}$, and for $1 \leq p \leq k-1$, let $Z_{w_p}$ denote the component of $\IN(T)$ corresponding to the leaf $w_p$ that is removed together with $e_p$ in the inductive construction of the matching. Then, $\epsilon_{(i_1,\cdots,i_k)}(Z_{w_p})=x^{i_{p+1}}$. Here, $i_1 \in \{0,\cdots,m-1\}$ and $i_j \in \{1,\cdots,m-1\}$ for all $2\leq j \leq k$. Thus, every generator of the Morse complex
        \[
        (\morCom{}{\M_{\nbc,m}}(G), \morDiff{}{\M_{\nbc,m}})
        \]
        is of the form $(H^T_c, \epsilon_{(i_1,\cdots,i_k)})$. Now to define the spanning tree complex of $G$ over $\A_m$, we introduce the following notation for its generators which will correspond to the generators of $(\morCom{}{\M_{\nbc,m}}(G), \morDiff{}{\M_{\nbc,m}})$. We define
        \[
        T^{(i_1,\cdots,i_k)} := (H_c^T, \epsilon_{(i_1,\cdots,i_k)}).
        \]

        We will interchangeably sometimes use $(H_c^T, \epsilon_{(i_1,\cdots,i_k)})$ to denote $T^{(i_1,\cdots,i_k)}$ whenever required. % Point 75 addressed.

        % We denote a generator of the spanning tree complex over $\A_m$ by
        % \[
        % T^{(i_1,i_2,\cdots,i_k)}, \quad i_1 \in\{0,1,\cdots,m-1\}, \  i_t \in \{1,2,\cdots,m-1\} \text{ for } t \in \{2,\cdots,k\}
        % \]
        % defined as follows. Let $T$ be an NBC spanning tree and let $k := |\pi_0(\IN(T))| = |\IA(T)| + 1$. Write $\IA(T) = \{e_1 <_\M e_2 <_\M \cdots <_\M e_{k-1}\}$ in the matching order of $\M_{\widetilde{T},m}$, and, for $1 \leq p \leq k-1$, let $Z_{w_p}$ denote the component of $\IN(T)$ corresponding to the leaf $w_p$ that is removed together with $e_p$ in the inductive construction of the matching. Then $T^{(i_1,i_2,\cdots,i_k)}$ denotes the critical enhanced spanning subgraph $(\IN(T), \epsilon_c)$, where $\epsilon_c(Z_{root}) = x^{i_1}$ is the enhancement of the component containing $v_{root}$ and $\epsilon_c(Z_{w_p}) = x^{i_{p+1}}$ for $1 \leq p \leq k-1$, as well as the corresponding generator of the spanning tree complex. We use this notation throughout the remainder of the paper. % Point 77 addressed.

        \begin{definition}
            Let $G$ be a graph with a fixed ordering on its set of edges and with a distinguished vertex $v_{root}$, then the spanning tree complex of $G$ over $\A_m$ is defined as
            \[
            \begin{aligned}
                & \cstBiCom{i}{j}{m}{G} := \Z \langle T^{(i_1,i_2,\cdots,i_k)} \mid i(T^{(i_1,i_2,\cdots,i_k)}) = i,\ j(T^{(i_1,i_2,\cdots,i_k)})=j  \rangle\\
                & \cstBiDiff{i}{j}{m} : \cstBiCom{i}{j}{m}{G} \to \cstBiCom{i+1}{j}{m}{G}
            \end{aligned}
            \]
            where, $i(T^{(i_1,i_2,\cdots,i_k)}) := i(H_c^T, \epsilon_{(i_1,i_2,\cdots,i_k)})$ and $j(T^{(i_1,i_2,\cdots,i_k)}) := j(H_c^T, \epsilon_{(i_1,i_2,\cdots,i_k)})$.
        \end{definition}

        Let $(H,\epsilon)$ and $(H',\epsilon')$ be two enhanced spanning subgraphs of $[\R(T),T]$. Define
        \[
        \Pt_m^{\downarrow}((H, \epsilon), (H',\epsilon'))
        \]
        to be the collection of alternating paths between $(H,\epsilon)$ and $(H',\epsilon')$ in the graph $\G^{\M_{\nbc,m}}(\nbcCom{m}{G})$ such that $i((H,\epsilon)) = i((H',\epsilon'))$.

        \begin{proof}[Proof of \autoref{main} for $m > 2$]
            This follows from \autoref{NBC Isomorphism}, \autoref{dmt maintheorem}, and \Autoref{criticalstates-AN}.
        \end{proof}

        \begin{lemma}\label{fundamental lemma for am}
            $ |\Pt_m^{\downarrow}((H, \epsilon), (H',\epsilon'))| \in \{0,1\}$.
        \end{lemma}

        \begin{proof}
            The proof follows from a similar argument in \Autoref{fundamental} and \Autoref{atmost one path cor}.
        \end{proof}

        \begin{definition}\label{non-vanishing edge}
            Let $T^{(i_1,i_2,\cdots,i_k)} \in \cstBiCom{\ast}{\ast}{m}{G}$ and $\IA(T) = \{e_1,e_2,\cdots,e_{k-1}\}$ with $e_1 <_\M e_2 <_\M \cdots <_\M e_{k-1}$. We say $e_p$ is \textit{non-vanishing} if $i_{t+1} + i_{p+1} \leq m-1$, where $e_{t} <_\M e_p$, and the components $Z_{w_{t}}$ and $Z_{w_p}$ are connected by $e_p$ and whenever $t=0$ we have $Z_{w_0}=Z_{root}$.
        \end{definition} % Point 78 addressed.

        Recall that $T^{(i_1,i_2,\cdots,i_{k})}=(H_c^T,\epsilon_{(i_1,i_2,\cdots,i_k)})$. Let $e,e' \in \IA(T)$ be non-vanishing edges with $Z_1, Z_2 \in \pi_0(H_c^T)$ such that $e$ connects $Z_1$ and $Z_2$. Let $P(e,e')$ be a path in $T$ with $P(e,e') = v_1  \to v_2 \to \cdots \to v_m$ such that $e = \{v_1,v_2\}$ and $e'=\{v_{m-1},v_m\}$ and $v_1$ is nearest to $v_{root}$ and $v_1 \in V(Z_1)$. Let $\{w_1,w_2,\cdots,w_t\} \subset V(P(e,e'))$ be such that all the edges in $P(e,e') \cap \IA(T) = \{f_1,f_2,\cdots, f_t\}$ and $w_i$ is the initial vertex of $f_i$. In this case, we have $w_1=v_1$. Furthermore, let $Z_{w_i}$ denote the component in $H_c^t$ containing the vertex $f_i \setminus\{w_i\}$. Here $Z_{w_2}=Z_2$.\\

        Now define $(H^{P(e,e')}, \epsilon^{P(e,e')})$ to be the enhanced spanning subgraph, where $H^{P(e,e')} = H_c^T$, and define $\epsilon^{P(e,e')}: \pi_0(H^{P(e,e')}) \to \{1,x,\cdots,x^{m-1}\}$ as follows:

        \begin{equation}\label{enhancement of P(e)}
            \epsilon^{P(e,e')}(Z)= \begin{cases}
                \epsilon_{(i_1,i_2,\cdots,i_k)}(Z), & \text{if } Z \cap P(e,e') = \emptyset\\
                \epsilon_{(i_1,i_2,\cdots,i_k)}(Z_1).\epsilon_{(i_1,i_2,\cdots,i_k)}(Z_2), & \text{if } Z= Z_1\\
                \epsilon_{(i_1,i_2,\cdots,i_k)}(Z_{w_j}), & \text{if } Z = Z_{w_{j-1}}, \ 3 \leq j \leq t \\
                1, & \text{if } Z=Z_{w_t}
            \end{cases}
        \end{equation}

        Define
        \[
        P((e_1,e_1'),(e_2,e_2'),\cdots,(e_l,e_l')) = P(e_1,e_1') \cup P(e_2,e_2') \cup \cdots \cup P(e_l,e_l'),
        \]
        where each $P(e_i,e_i')$ is a path in $T$ as described above, with each $e_i$ and $e_i'$ being non-vanishing in $(H_c^T, \epsilon_{(i_1,i_2,\cdots,i_k)})$, such that $E(P(e_i,e_i')) \cap E(P(e_j,e_j')) = \emptyset$ for all $i \neq j$. Denote
        \[
        \left(H^{P((e_1,e_1'),(e_2,e_2'),\cdots,(e_l,e_l'))}, \epsilon^{P((e_1,e_1'),(e_2,e_2'),\cdots,(e_l,e_l'))}\right)
        \]
        to be the enhanced spanning subgraph, where $H^{P((e_1,e_1'),(e_2,e_2'),\cdots,(e_l,e_l'))} = H_c^T$ and

        \[
        \epsilon^{P((e_1,e_1'),(e_2,e_2'),\cdots,(e_l,e_l'))}(Z):=\begin{cases}
            \epsilon_{(i_1,i_2,\cdots,i_k)}(Z), & \text{if } Z \cap P((e_1,e_1'),(e_2,e_2'),\cdots,(e_l,e_l')) = \emptyset\\
            \epsilon^{P(e_i,e_i')}(Z), & \text{if } Z \cap P(e_i,e_i') \neq \emptyset
        \end{cases}
        \]

        We now characterize all the enhanced spanning subgraphs $(H,\epsilon)$ for which there is an alternating path from $(H_c^T, \epsilon_{(i_1,i_2,\cdots,i_k)})$ to $(H,\epsilon)$ in $[\R(T), T]$, which is a generalization of \Autoref{path within tree}.

        \begin{lemma}\label{path-within-tree-am}
            Let  $T^{(i_1,i_2,\cdots, i_k)}$ be a generator of the complex $\cstCom{}{m}{G}$. Then,
            \[
            \lvert \Pt^{\downarrow}_m((H_c^T,\epsilon_{(i_1,i_2,\cdots,i_k)}),(H,\epsilon))\rvert =1
            \]
            \sloppy{if and only if $H = H^{P((e_1,e_1'),(e_2,e_2'),\cdots,(e_l,e_l'))}$ and $\epsilon = \epsilon^{P((e_1,e_1'),(e_2,e_2'),\cdots,(e_l,e_l'))}$ for some $P((e_1,e_1'),(e_2,e_2'),\cdots,(e_l,e_l')) \subset T$.}
        \end{lemma}

        \begin{proof}
            The internally active edges in $P((e_1,e_1'),(e_2,e_2'),\cdots,(e_l,e_l'))$ can be arranged in the matching order with respect to $\M_{T,m}$ and we follow a similar argument as in \Autoref{path within tree} to produce an alternating path from $(H_c^T,\epsilon_{(i_1,i_2,\cdots,i_k)})$ to $(H^{P((e_1,e_1'),(e_2,e_2'),\cdots,(e_l,e_l'))}, \epsilon^{P((e_1,e_1'),(e_2,e_2'),\cdots,(e_l,e_l'))})$.\\
            For the sufficient part of the statement, we observe that in the case of $m=2$, % Point 80 addressed.
            the only non-vanishing internally active edge was the one whose one of the end vertices is contained in $Z_{root}$, but in the case of $\A_m$, that need not be the situation. Thus, we can initiate an alternating path from $(H_c^T,\epsilon_{(i_1,i_2,\cdots,i_k)})$ by taking a differential along any non-vanishing internally active edge $e_1$, and one can continue this alternating path along a path $P(e_1,e_1')$ in $T$ beginning with $e_1$. At some point, we can further begin to take a differential at some other non-vanishing edge $e_2$ which is disjoint from $P(e_1,e_1')$ such that $e_2 <_\M f$ for all $f \in \IA(T) \cap P(e_1,e_1')$. We continue this procedure to produce a sequence $P(e_1,e_1'), P(e_2,e_2'), \cdots, P(e_m,e_m')$. See \autoref{figure:am-path-example}.
        \end{proof}

        \begin{figure}
            \resizebox{\textwidth}{!}{
            \begin{tabular}{c c c c c c c c c}
                \begin{tikzpicture}[vertex/.style={circle, draw, inner sep=0pt, minimum size=6pt, fill=black}, edge/.style={draw}, baseline=0]
                    \node[vertex,fill=red,color=red] [label=above:$v_{root} \:\: x$] (a) at (0,2) {};
                    \node[vertex] [label=right:$x^2$] (b) at (0,1) {};
                    \node[vertex][label=right:$x$]  (c) at (0,0) {};
                    \node[vertex][label=right:$x^2$] (d) at (0,-1) {};
                    \node[vertex] [label=left:$x^2$] (e) at (-1.5,1) {};
                    \node[vertex][label=left:$x$] (f) at (-1.5,0) {};
                    \node[vertex][label=right:$x^2$]  (g) at (1.5,0.5) {};
                    \path[edge,dotted] (a) to (b);
                    \path[edge,dotted] (b) to (c);
                    \path[edge,dotted] (a) to[out=180,in=70] (e);
                    \path[edge,dotted] (e) to (f);
                    \path[edge,dotted] (f) to[out=290,in=180] (d);
                    \path[edge,dotted] (a) to[out=0,in=90] (g);
                \end{tikzpicture} &
                \Large{$\xrightarrow{d}$} &
                \begin{tikzpicture}[vertex/.style={circle, draw, inner sep=0pt, minimum size=6pt, fill=black}, edge/.style={draw}, baseline=0]
                    \node[vertex,fill=red,color=red] [label=above:$v_{root} \:\: x$] (a) at (0,2) {};
                    \node[vertex] [label=right:$x^2$] (b) at (0,1) {};
                    \node[vertex][label=right:$x$]  (c) at (0,0) {};
                    \node[vertex][label=right:$x^2$] (d) at (0,-1) {};
                    \node[vertex] (e) at (-1.5,1) {};
                    \node[vertex] (f) at (-1.5,0) {};
                    \node[vertex][label=right:$x^2$]  (g) at (1.5,0.5) {};
                    \path[edge,dotted] (a) to (b);
                    \path[edge,dotted] (b) to (c);
                    \path[edge,dotted] (a) to[out=180,in=70] (e);
                    \path[edge,red,thick] (e) to node[left]{$x^3$} (f);
                    \path[edge,dotted] (f) to[out=290,in=180] (d);
                    \path[edge,dotted] (a) to[out=0,in=90] (g);
                \end{tikzpicture} & \Large{$\xrightarrow{\M}$} &
                \begin{tikzpicture}[vertex/.style={circle, draw, inner sep=0pt, minimum size=6pt, fill=black}, edge/.style={draw}, baseline=0]
                    \node[vertex,fill=red,color=red] [label=above:$v_{root} \:\: x$] (a) at (0,2) {};
                    \node[vertex] [label=right:$x^2$] (b) at (0,1) {};
                    \node[vertex][label=right:$x$]  (c) at (0,0) {};
                    \node[vertex][label=right:$x^2$] (d) at (0,-1) {};
                    \node[vertex] [label=left:$x^3$] (e) at (-1.5,1) {};
                    \node[vertex][label=left:$1$] (f) at (-1.5,0) {};
                    \node[vertex][label=right:$x^2$]  (g) at (1.5,0.5) {};
                    \path[edge,dotted] (a) to (b);
                    \path[edge,dotted] (b) to (c);
                    \path[edge,dotted] (a) to[out=180,in=70] (e);
                    \path[edge,dotted] (e) to (f);
                    \path[edge,dotted] (f) to[out=290,in=180] (d);
                    \path[edge,dotted] (a) to[out=0,in=90] (g);
                \end{tikzpicture} & \Large{$\xrightarrow{d}$} &
                \begin{tikzpicture}[vertex/.style={circle, draw, inner sep=0pt, minimum size=6pt, fill=black}, edge/.style={draw}, baseline=0]
                    \node[vertex,fill=red,color=red] [label=above:$v_{root} \:\: x$] (a) at (0,2) {};
                    \node[vertex] [label=right:$x^2$] (b) at (0,1) {};
                    \node[vertex][label=right:$x$]  (c) at (0,0) {};
                    \node[vertex] (d) at (0,-1) {};
                    \node[vertex] [label=left:$x^3$] (e) at (-1.5,1) {};
                    \node[vertex] (f) at (-1.5,0) {};
                    \node[vertex][label=right:$x^2$]  (g) at (1.5,0.5) {};
                    \path[edge,dotted] (a) to (b);
                    \path[edge,dotted] (b) to (c);
                    \path[edge,dotted] (a) to[out=180,in=70] (e);
                    \path[edge,dotted] (e) to (f);
                    \path[edge,red,thick] (f) to[out=290,in=180] node[left]{$x^2$} (d);
                    \path[edge,dotted] (a) to[out=0,in=90] (g);
                \end{tikzpicture} & \Large{$\xrightarrow{\M}$} &
                \begin{tikzpicture}[vertex/.style={circle, draw, inner sep=0pt, minimum size=6pt, fill=black}, edge/.style={draw}, baseline=0]
                    \node[vertex,fill=red,color=red] [label=above:$v_{root} \:\: x$] (a) at (0,2) {};
                    \node[vertex] [label=right:$x^2$] (b) at (0,1) {};
                    \node[vertex][label=right:$x$]  (c) at (0,0) {};
                    \node[vertex][label=right:$1$] (d) at (0,-1) {};
                    \node[vertex] [label=left:$x^3$] (e) at (-1.5,1) {};
                    \node[vertex][label=left:$x^2$] (f) at (-1.5,0) {};
                    \node[vertex][label=right:$x^2$]  (g) at (1.5,0.5) {};
                    \path[edge,dotted] (a) to (b);
                    \path[edge,dotted] (b) to (c);
                    \path[edge,dotted] (a) to[out=180,in=70] (e);
                    \path[edge,dotted] (e) to (f);
                    \path[edge,dotted] (f) to[out=290,in=180] (d);
                    \path[edge,dotted] (a) to[out=0,in=90] (g);
                \end{tikzpicture}
            \end{tabular}}
            \caption{An alternating path with $\A_4$ as the underlying algebra}
            \label{figure:am-path-example}
        \end{figure}

        We now provide an explicit form of the differential $\cstDiff{}{m}$ in the spanning tree complex $\cstCom{}{m}{G}$.\\

        Let $T$ be an NBC spanning tree with $\lvert \pi_0(H_c^T) \rvert= k$ and $e\in \EN(T)$ with $e=\{v_1,v_2\}$. Let $P_{v_i}$ be the unique path from $v_{root}$ to $v_i$ in $T$ for $i=1,2$ and suppose $Z_{v_1}, Z_{v_2} \in \pi_0(H_c^T)$ be the two components containing $v_1$ and $v_2$ respectively. Furthermore, let $\operatorname{NV}(T^{(i_1,\cdots, i_k)}, P_{v_j})$ be the collection of all non-vanishing internally active edges of $T$ with respect to $\epsilon_{(i_1,i_2,\cdots,i_k)}$, and for each $f \in \operatorname{NV}(T^{(i_1,\cdots, i_k)}, P_{v_j})$, let $P(f,v_j)$ be the subpath in $P_{v_j}$ between the starting vertex of $f$ and $v_j$. Finally, let $L(P(f,v_j))$ be the set $P(f,v_j) \cap \IA(T)$. Then, we have the following theorem describing the combinatorial form of $\cstDiff{}{m}$:

        \begin{theorem}\label{Am-differential}
            Let $T$ be an NBC spanning tree with $|\pi_0(H_c^T)|=k$ and $e \in \EN(T)$ with $e = \{v_1,v_2\}$. Then for a generator $T^{(i_1,\cdots,i_k)} \in \cstCom{}{m}{G}$, the image of $T^{(i_1,\cdots,i_k)}$ under $\cstDiff{}{m}$ can be described combinatorially by the following formula:

            \begin{equation}\label{am-diff-formula}
                \begin{aligned}
                    \cstDiff{}{m}(T^{(i_1,i_2,\cdots, i_k)}) = \sum_{\substack{e \in \EN(T) \\ \psi_e(T) \in \F(\NBC(G)) \\ i(\psi_e(T)) = i(T) + 1}} & (-1)^{\xi(T,e)}\cdot \psi_e(T)^{(u_1,\cdots, u(v_1,v_2), \cdots, u_{k-1})}  \\
                    & + \sum_{j=1}^2 \sum_{f \in \operatorname{NV}(T^{(i_1,\cdots, i_k)}, P_{v_j})} (-1)^{|L(P(f,v_j))| + \xi(T,e)} \cdot \psi_e(T)^{P(f,v_j)}
                \end{aligned}
            \end{equation}

            \sloppy{where, $\psi_e(T)^{(u_1,\cdots, u(v_1,v_2), \cdots, u_{k-1})}$ denotes the critical enhanced spanning subgraph $H_c^{\psi_e(T)}$ with $\epsilon_{(u_1,\cdots, u(v_1,v_2), \cdots, u_{k-1})}$ being identical on all components with $\epsilon_{(i_1,\cdots,i_k)}$ which are different from $Z_{v_1}$ and $Z_{v_2}$ and}
            \[
            \epsilon_{(u_1,\cdots, u(v_1,v_2), \cdots, u_{k-1})}(Z_{v_1} \cup e \cup Z_{v_2}) =  u(v_1,v_2) = \epsilon_{(i_1,i_2,\cdots,i_k)}(Z_{v_1}) \cdot\epsilon_{(i_1,i_2,\cdots,i_k)}(Z_{v_2})
            \]
            While, $\psi_e(T)^{P(f,v_j)}$ denotes the critical enhanced spanning subgraph $(H_c^{\psi_e(T)}, \epsilon^{P(f,v_j)})$ and $\epsilon^{P(f,v_j)}$ coincides with $\epsilon^{P(f,f')}$, as described in Equation \eqref{enhancement of P(e)} with the endpoint of $f'$ being $v_j$, for all components in $\pi_0(H_c^T) \cap \pi_0(H_c^{\psi_e(T)})$ except
            \[
            \epsilon^{P(f,f')}(Z_{v_1} \cup e \cup Z_{v_2}) = \epsilon^{P(f,f')}(Z_{v_1}) \cdot \epsilon^{P(f,f')}(Z_{v_2}).
            \]
        \end{theorem}

        \begin{proof}
            By virtue of \Autoref{lextrick}, the incidence between a critical enhanced spanning subgraph of $T$ and $T'$ in $\cstCom{}{m}{G}$ is non-zero only if $T' = \psi_e(T)$, where $e \in \EN(T)$. Let $T^{(i_1,\cdots,i_k)}$ be a fixed generator; then any alternating path starting from the critical enhanced spanning subgraph of $T^{(i_1,\cdots,i_k)}$ and ending at a critical enhanced spanning subgraph in $[\R(\psi_e(T)), \psi_e(T)]$ has a directed edge whose endpoints have two enhanced spanning subgraphs, one from $[\R(T), T]$ and the other in $[\R(\psi_e(T)), \psi_e(T)]$. The one which is in $[\R(T), T]$ has the form
            \[
                \left(H^{P((e_1,e_1'),(e_2,e_2'),\cdots,(e_l,e_l'))},\epsilon^{P((e_1,e_1'),(e_2,e_2'),\cdots,(e_l,e_l'))}\right)
            \]
            by \Autoref{path-within-tree-am} such that
            \begin{equation}\label{non-vanishing ext edge}
                \epsilon^{P((e_1,e_1'),(e_2,e_2'),\cdots,(e_l,e_l'))}(Z_{v_1}) \cdot \epsilon^{P((e_1,e_1'),(e_2,e_2'),\cdots,(e_l,e_l'))}(Z_{v_2}) \neq 0.
            \end{equation}
            Now there are two possibilities, either $\epsilon_{(i_1,\cdots,i_k)}(Z_{v_1})\cdot \epsilon_{(i_1,\cdots,i_k)}(Z_{v_2}) \neq 0$ in which case we have the first summand of $\cstDiff{}{m}(T^{(i_1,i_2,\cdots, i_k)})$ or else there must exist $P((e_1,e_1'),(e_2,e_2'),\cdots,(e_l,e_l'))$ such that Equation \eqref{non-vanishing ext edge} holds. Now we claim that $P((e_1,e_1'),(e_2,e_2'),\cdots,(e_l,e_l'))  = P(f, v_j)$ for some $f \in \operatorname{NV}(T^{(i_1,\cdots,i_k)}, P_{v_j})$ with $j=1,2$.\\
            \sloppy{If not, then either there exists a non-vanishing internally active edge outside $P_{v_j}$ or $P((e_1,e_1'),(e_2,e_2'),\cdots,(e_l,e_l')) \subset P_{v_1} \cup P_{v_2}$ but no $P(e_m,e_m')$ contains $v_1$ or $v_2$. The former case implies that there is a component $Z \not\in \{Z_{v_1}, Z_{v_2}\}$ and $Z \not\in \pi_0(H_c^T) \cap P_{v_j}$ such that $\epsilon^{P((e_1,e_1'),(e_2,e_2'),\cdots,(e_l,e_l'))}(Z)=1$ and thus, this enhancement is preserved in the enhanced spanning subgraph of $[\R(\psi_e(T)), \psi_e(T)]$ through the alternating path, which ultimately contradicts the existence of this alternating path. In the latter case, the argument is a special case of the former, where $Z \not\in \pi_0(H_c^T) \cap P_{v_j}$ is eliminated.}
        \end{proof}

        As a special case, we compute the chromatic homology of a tree over the algebra $\A_m$.

        \begin{corollary}\label{tree hom Am}
            Let $G=T_n$ be a tree with $n+1$ vertices and $n$ edges. Then
            \[
            \cstHom{i}{m}{T_n} \cong \begin{cases}
                \A_m \otimes (x\A_m)^{\otimes n}, & i=0\\
                0, & i>0
            \end{cases}
            \]
        \end{corollary}

        \begin{proof}
            The only spanning tree of $T_n$ is $T_n$ itself, and thus for any ordering on its edges, all the edges of $T_n$ are internally active. Hence, all the enhanced critical states of $T_n$ belong to homological grading $i=0$. From \Autoref{criticalstates-AN}, it follows that the number of generators in $\cstCom{0}{m}{T_n}$ is $m(m-1)^n$. As the higher homological grading chain groups are trivial, $\cstDiff{}{m}$ is trivial and thus the result follows.
        \end{proof}

        An independent computation of the above example was performed by Helme-Guizon and Rong in \cite{graphcohomologiesarbitraryalgebras}.

        \subsection{Homological Span of Chromatic Homology}

        In this subsection, we provide a proof of a conjecture posed by Sazdanovic and Scofield \cite{sazdanovic2018patternskhovanovlinkchromatic}. See \autoref{hspanconjecture}
        % \Homologicalspanconjecture*

        It should be noted that they proved this result for the $\A_2$ algebra and conjectured that it should hold for any $\A_m$ algebra based on computational evidence. In their paper, they also showed that as the girth of the graph approaches infinity, the span of Khovanov homology tends to infinity. We reprove this result using the spanning tree model.\\

        Before proving the result, we first recall some basic definitions.

        \begin{definition}
            Let $G$ be a graph. Then a \textit{block} $B$ of $G$ is a subgraph which is either a maximal $2$-connected subgraph of $G$ or a bridge.
        \end{definition}

        \begin{definition}
            Let $i_{\min}$ and $i_{\max}$ be the minimum and maximum homological gradings of chromatic homology with non-trivial homology groups for a graph $G$. $hspan \left(H_{*,*}^{\A_m}(G)\right):=i_{\max}-i_{\min}+1$.
        \end{definition}

        \begin{proof}[Proof of \Autoref{hspanconjecture}]
            It is well known that $i_{\min}=0$. It also follows from the fact that if $T_1$ is the minimum NBC spanning tree with respect to the lexicographic order, then $[{T_{1}^{(m-1,m-1,\cdots, m-1)}}] \in \cstBiHom{0}{(m-1)v}{m}{G}$ is a non-zero homology class.\\

            We will show that in a graph with $b$ blocks, there exists an NBC spanning tree $T_{\max}$
            that has maximum homological grading and has exactly $b$ live edges. As a consequence, the homological degree of $T_{\max}$ is $v - b - 1$. Now, in the spanning tree complex $\cstBiCom{\ast}{\ast}{m}{G}$, $[{T^{(0,1,\cdots, 1)}_{\max}}]$ (i.e., the enhanced spanning tree with enhancement $1$ in the component containing $v_{root}$ and enhancement $x$ in other components) is obviously a cycle. Furthermore, by \Autoref{Am-differential}, $[{T^{(0,1,\cdots, 1)}_{\max}}] \in \cstBiHom{\ast}{\ast}{m}{G}$ is a non-zero homology class. This will prove that $ hspan \left(H_{*,*}^{\A_m}(G)\right) = v - b$.

            \begin{enumerate}
                \item Let $B$ be a block that is $2$-connected. Recall that an \textit{ear} of a subgraph $H$ of $G$ is a path $P$ in $G$ whose end vertices are distinct vertices of $H$ and whose internal vertices and edges do not belong to $H$ (a single edge of $G$ not in $H$ joining two vertices of $H$ is also an ear). % Point 82 addressed.
                Also, recall that for a two-connected graph, there exists an \textit{ear decomposition}. An ear decomposition of $B$ is a nested sequence $(G_0,G_1,\cdots,G_k)$ of non-separable subgraphs such that

                \begin{enumerate}
                    \item $G_0$ is a cycle,
                    \item $G_{i+1} = G_i \cup P_i$, where $P_i$ is an ear (a path whose endpoints may be vertices of $ G_i$, but the internal vertices are not in $G_i$) of $G_i$ in $G$ for all $0 \leq i < k$,
                    \item $G_k = B$.
                \end{enumerate}

                We fix an ordering of edges in $B$ such that whenever $i<j$, we have $e < f$ for all $e \in G_i$ and $f \in G_j$.\\

                Now, we construct the tree $T_{B}$ by removing the second smallest edge from $G_0$ and the smallest edge from each ear. It is easy to see that $T_B$ is an NBC spanning tree. Furthermore, the only live edge in $T_B$ is the smallest in $G_0$.

                \item If $B$ is a block that is a bridge, then in any spanning tree of $G$, $B$ will be a live edge. In fact, we define $T_B = B$ in this case.

                \item Given a block decomposition $(B_1,\cdots,B_b)$ of a connected graph $G$,

                \[ T_{\max} := T_{B_1} \cup \cdots \cup T_{B_b}. \]

                $T_{\max}$ is an NBC spanning tree with $b$ live edges. It is also clear that in any NBC spanning tree there should be at least one live edge, as otherwise it would contain a broken circuit.

                Finally, to see that $T_{\max}$ has the maximum homological grading, we observe that there is no basis exchange (i.e., removing an internal edge $e$ and adding an external edge $f$) in $T_{\max}$ that increases the homological grading. This concludes our construction.
            \end{enumerate}
        \end{proof}

        Now we will establish a more general theorem about the effect on chromatic homology for the operation of vertex gluing. First, let us define the reduced spanning tree complex $\cstredCom{}{m}{G}$ to be the quotient complex of $\cstCom{}{m}{G}$ generated by $T^{(i_1,\cdots,i_k)} \in \cstCom{}{m}{G}$ such that $i_1=0$, that is, the enhancement of the connected component containing the root vertex is equal to $1$.

        \begin{theorem} \label{vertex gluing}
            Let $G$ be a connected graph obtained by gluing two connected graphs $G_1$ and $G_2$ in a single vertex. Then, we have an isomorphism of chain complexes $\cstCom{}{m}{G} \cong \cstCom{}{m}{G_1}\otimes \cstredCom{}{m}{G_2}$.
        \end{theorem}
        \begin{proof}

            Without loss of generality, we assume that the root is placed at the vertex where the two graphs are glued. This vertex will also serve as the root when considering the graphs $G_1$ and $G_2$ individually.

            Suppose $T$ is a spanning tree of $G$. Then, $T_i = T \cap G_i$ is a spanning tree of $G_i$. Since $G_1$ and $G_2$ share the root vertex, for any critical state $(H_c^T, \epsilon_{(i_1,\cdots,i_k)})$ of $G$, the enhancement $\epsilon_{(i_1,\cdots,i_k)}$ decomposes as $\epsilon_1 \otimes \epsilon_2$, where $\epsilon_1$ and $\epsilon_2$ are enhancements of $H_c^{T_1}$ and $H_c^{T_2}$, respectively with $\epsilon_2$  assigning $1$ to the component containing the root vertex.

            Consider any two critical states $(H_c^{T_i}, \epsilon_i)$ and $(H_c^{T_i'}, \epsilon'_i)$ of $G_i$. Let $T'$ be a spanning tree of $G$ such that $T'_i = T' \cap G_i$. Let $\epsilon$ and $\epsilon'$ be any two enhancements of $H_c^T$ in $G$ such that they restrict to the enhancements $\epsilon_i$ of $H_c^{T_i}$ and $\epsilon'_i$ of $H_c^{T_i'}$, respectively, in the chromatic complex of $G_i$. Since $\cyc(T_i, e) = \cyc(T, e)$ for any external edge $e$, using the description of $\cstDiff{\ast}{m}$ in \Autoref{Am-differential}, it follows that
            \[
            [ (H_c^{T_i}, \epsilon_i) : (H_c^{T_i'}, \epsilon'_i)]_{\cstBiDiff{*}{*}{m}} =  [(H_c^T, \epsilon) : (H_c^{T'}, \epsilon')]_{\cstBiDiff{*}{*}{m}}.
            \]

            We  define the following chain map:

            \[
            \Psi : \cstCom{}{m}{G} \to \cstCom{}{m}{G_1} \otimes \cstredCom{}{m}{G_2},
            \]

            \[    \Psi((H_c^T, \epsilon)) := \sum_{\substack{i((H_c^{T_1}, \epsilon_1))=i_1,  \: i(H_c^{T_2}, \epsilon_2))=i_2 \\ j((H_c^{T_1}, \epsilon_1))=j_1, \: j((H_c^{T_2}, \epsilon_2))=j_2 \\
            i_1+i_2 = i, \: j_1+j_2=j
            }} (H_c^{T_1}, \epsilon_1) \otimes (H_c^{T_2}, \epsilon_2).\]

            where, $i((H_c^T, \epsilon))=i$ and $j((H_c^T, \epsilon))=j$.

            Similarly, given two spanning trees \( T_1 \) and \( T_2 \) of \( G_1 \) and \( G_2 \), respectively, one can construct a spanning tree \( T = T_1 \cup T_2 \) of \( G \). We define the map:

            \[
            \Upsilon: \cstCom{}{m}{G_1} \otimes \cstredCom{}{m}{G_2} \to \cstCom{}{m}{G},
            \]

            \[  \Upsilon ( (H_c^{T_1}, \epsilon_1) \otimes (H_c^{T_1}, \epsilon_2) ) = (H_c^{(T_1 \cup T_2)}, \epsilon_1 \otimes \epsilon_2).\]

            where, $i((H_c^{T_k}, \epsilon_k))=i_k$, $j((H_c^{T_k}, \epsilon_k))=j_k$ for $k=1,2$ with $i((H_c^{(T_1 \cup T_2)}, \epsilon_1 \otimes \epsilon_2))=i_1 + i_2$ and $j((H_c^{(T_1 \cup T_2)}, \epsilon_1 \otimes \epsilon_2))=j_1+j_2$.\\
            It is easy to see that $\Upsilon = \Psi^{-1}$, and it is also a chain map. \end{proof}

        \subsection{Existence of $m$ torsion in $H_{*,*}^{\A_m}(G)$ }

        Helme-Guizon, Przytycki, and Rong conjectured that chromatic homology over the $\A_{m}$ algebra has torsion of order dividing $m$ \cite{HelmeGuizonPrzytyckiRong2006_torsion}. For the special case $m=2$, we gave a proof using the spanning tree complex in \Autoref{two torsion}, which also produced the explicit chain representing the torsion class. We will borrow the same construction idea to establish the general case. Before proving \Autoref{torsionconjecture}, we will first explore some model computations that demonstrate the torsion class on the chromatic complexes of the graphs $K_3$ and $K_{3,3}$.

        \subsubsection{m-torsion class for $K_3$}\label{torsion k_3}. Let $G=K_3$ with $V(G)=\{v_0,v_1,v_2\}$ and $E(G)=\{e_1=\{v_0,v_1\}, e_2 = \{v_0,v_2\}, e_3=\{v_1,v_2\}\}$ with $e_1 < e_2 < e_3$. Assume, $v_{root}=v_0$.\\
        The only broken circuit is $\{e_2,e_3\}$, so there are exactly two
        NBC spanning trees:
        \begin{center}
            \begin{tabular}{|l|l|l|l|}
                \hline
                Tree($T$) & $E(T)$ & $\IA(T)$ & $\IN(T)$\\ \hline
                $T_{1}$ & $\{e_1,e_2\}$ & $\{e_1,e_2\}$ & $\varnothing$\\ \hline
                $T_2$ & $\{e_1,e_3\}$ & $\{e_1\}$ & $\{e_3\}$\\
                \hline
            \end{tabular}
        \end{center}
        Their activity words are $W(T_{1})=LLd$ and $W(T_2)=LdD$. Since $\cut(T_2,e_3)=\{e_2,e_3\}$, hence we have
        \[
        \psi'_{e_3}(T_2)
        =T_2\cup\{e_2\}\setminus\{e_3\}=T_{1}.
        \]
        Conversely, $\cyc(T_{1},e_3)=\{e_1,e_2,e_3\}$ and its largest
        internally active edge is $e_2$, so this gives
        \[
        \psi_{e_3}(T_{1})
        =T_{1}\cup\{e_3\}\setminus\{e_2\}=T_2.
        \]
        The matching order is $e_1 <_\M e_2$. The critical components of
        $T_{1}$ are $\{v_{root}\},\{v_1\},\{v_2\}$, while those of
        $T_2$ are $\{v_{root}\},\{v_1,v_2\}$. Now consider the following critical enhanced spanning subgraphs of $T_1$ and $T_2$ where $i(T_1)=0$ and $i(T_2)=1$:
        \begin{center}
            \begin{tikzpicture}[scale=1]
                \begin{scope}
                    \node[root,label=above:{$v_{root} : x^p$}] (r) at (0,1.4) {};
                    \node[vtx,label=below left:{$v_1:x^i$}] (a) at (-1.1,0) {};
                    \node[vtx,label=below right:{$v_2:x^j$}] (b) at (1.1,0) {};
                    \draw[live] (r)--node[left]{$e_1$}(a);
                    \draw[live] (r)--node[right]{$e_2$}(b);
                    \draw[external] (a)--node[above]{$e_3$}(b);
                    \node at (0,-.8) {$T_{1}^{(p,i,j)}$};
                \end{scope}
                \begin{scope}[xshift=5.3cm]
                    \node[root,label=above:{$v_{root}:x^p$}] (r) at (0,1.4) {};
                    \node[vtx,label=below left:{$v_1$}] (a) at (-1.1,0) {};
                    \node[vtx,label=below right:{$v_2$}] (b) at (1.1,0) {};
                    \draw[live] (r)--node[left]{$e_1$}(a);
                    \draw[dead] (a)--node[above]{$e_3$}node[below]{$x^{i+j}$}(b);
                    \node at (0,-.8) {$T_2^{(p,i+j)}$};
                \end{scope}
            \end{tikzpicture}
        \end{center}
        where $p \in \{0,1\}$ and $i,j \in \{1,\cdots,m-1\}$ such that $i+j=m$ when $p=0$ and $i+j=m-1$ when $p=1$. This gives $j\left(T_1^{(p,i,j)}\right) = j\left(T_2^{(p,i,j)}\right)$. Now by Equation \eqref{am-diff-formula} we have three alternating paths between $T_1^{(p,i,j)}$ and $T_2^{(p,i,j)}$:
        \begin{enumerate}
            \item $P_1$ obtained by taking differential at $e_3$.
            \item $P_2$ obtained by first taking the differential and then successively matching at $e_1$, and finally taking differential at $e_3$.
            \item $P_3$ obtained by first taking the differential and then successively matching at $e_2$, and finally taking differential at $e_3$.
        \end{enumerate}
        Their weights are
        \[
        w(P_1)=1,\qquad
        w(P_2)=(1)(-1)(1)=-1,\qquad
        w(P_3)=(1)(-1)(1)=-1.
        \]
        Therefore, for $p=1$ we have
        \begin{equation}
            \cstDiff{}{m}\bigl(T_{1}^{(0,i,j)}\bigr)
            =-T_2^{(i,j)}-T_2^{(j,i)}.
            \label{eq:k3-second}
        \end{equation}
        and for $p=0$ we have
        \begin{equation}
            \cstDiff{}{m}\bigl(T_{1}^{(1,i,j)}\bigr)
            =T_1^{(1,m-1)}-T_1^{(i+1,j)}-T_1^{(j+1,i)}.
            \label{eq:k3-first}
        \end{equation}
        Thus, we take the following chain element
        \[
        S=\sum_{\substack{i+j=m-1\\i,j\geq1}}T_{1}^{(1,i,j)}
        -\sum_{\substack{i+j=m\\i,j\geq1}}T_{1}^{(0,i,j)}.
        \]
        Using Equation \eqref{eq:k3-first} and~\eqref{eq:k3-second}, we have
        \begin{equation}\label{eq:k3S}
            \begin{aligned}
                \cstDiff{}{m}(S)
                &=(m-2)T_2^{(1,m-1)}
                -\sum_{\substack{i+j=m-1\\i,j\geq1}}
                \bigl(T_2^{(i+1,j)}+T_2^{(j+1,i)}\bigr)\nonumber\\
                &\hspace{2.5em}
                +\sum_{\substack{i+j=m\\i,j\geq1}}
                \bigl(T_2^{(i,j)}+T_2^{(j,i)}\bigr)\nonumber\\
                &=(m-2)T_2^{(1,m-1)}
                -2\sum_{k=2}^{m-1}T_2^{(k,m-k)}
                +2\sum_{k=1}^{m-1}T_2^{(k,m-k)}\nonumber\\
                &=mT_2^{(1,m-1)}.
            \end{aligned}
        \end{equation}
        Since the homological grading of $T_2$ is maximal,
        $T_2^{(1,m-1)}$ is a cycle. We verify non-exactness by comparing
        coefficients of enhanced trees for an arbitrary integral linear combination of all possible enhanced critical spanning subgraphs of $T_1$. Suppose, for contradiction, that
        \begin{equation}
            \cstDiff{}{m}\left(
            \sum_{\substack{p+i+j=m\\0\leq p\leq m-2\\1\leq i,j\leq m-1}}
            a_{p,i,j}\,T_{1}^{(p,i,j)}
            \right)=T_2^{(1,m-1)},\qquad a_{p,i,j}\in\mathbb Z.
            \label{eq:arbitrary-boundary}
        \end{equation}
        In the sums below, $p+i+j=m$ always carries these same bounds.
        The three alternating paths give the individual boundaries
        \[
        \cstDiff{}{m}\bigl(T_{1}^{(p,i,j)}\bigr)
        =\begin{cases}
            T_2^{(p,m-p)}-T_2^{(p+i,j)}-T_2^{(p+j,i)},&p\geq1,\\
            -T_2^{(i,j)}-T_2^{(j,i)},&p=0.
        \end{cases}
        \]
        Now the entire boundary on the left-hand side of~\eqref{eq:arbitrary-boundary} is
        \begin{equation}\label{eq:full-boundary-expansion}
            \sum_{\substack{p+i+j=m\\p\geq1}}
            a_{p,i,j}\,T_2^{(p,m-p)}
            -\sum_{p+i+j=m}a_{p,i,j}\,T_2^{(p+i,j)}
            -\sum_{p+i+j=m}a_{p,i,j}\,T_2^{(p+j,i)}.
        \end{equation}
        Fix $1\leq k\leq m-1$ and compare the coefficient of
        $T_2^{(k,m-k)}$. The first sum in Equation \eqref{eq:full-boundary-expansion} contributes precisely when $p=k$, giving
        \[
        \sum_{\substack{i+j=m-k\\i,j\geq1}}a_{k,i,j}.
        \]
        This sum is empty for $k=m-1$. The second sum contributes when
        $p+i=k$, so $i=k-p$, $j=m-k$, and $0\leq p\leq k-1$.
        The third contributes when $p+j=k$, so $j=k-p$, $i=m-k$, and
        again $0\leq p\leq k-1$. Consequently every $k$ must satisfy
        \begin{equation}
            \sum_{\substack{i+j=m-k\\i,j\geq1}}a_{k,i,j}
            -\sum_{p=0}^{k-1}
            \bigl(a_{p,k-p,m-k}+a_{p,m-k,k-p}\bigr)
            =\begin{cases}
                1,&k=1,\\
                0,&2\leq k\leq m-1.
            \end{cases}
            \label{eq:all-coefficient-equations}
        \end{equation}
        Multiply the equation for a fixed $k$ by $k$ and add all the equations for $1\leq k\leq m-1$. The right-hand side is $1$.
        On the left, each $a_{p,i,j}$ occurs with direct weight $p$ when
        $p\geq1$, with weight $-(p+i)$ from the second sum, and with weight $-(p+j)$ from the third sum. For $p=0$, the missing direct contribution has weight $0$, so the same formula applies. Reindexing the sums gives
        \begin{align}
            1
            &=\sum_{p+i+j=m}p\,a_{p,i,j}
            -\sum_{p+i+j=m}(p+i)a_{p,i,j}
            -\sum_{p+i+j=m}(p+j)a_{p,i,j}\nonumber\\
            &=\sum_{p+i+j=m}
            \bigl(p-(p+i)-(p+j)\bigr)a_{p,i,j}\nonumber\\
            &=-m\sum_{p+i+j=m}a_{p,i,j}.
            \label{eq:integer-contradiction}
        \end{align}
        This gives a contradiction as $m \geq 2$. More generally, if the right-hand side of Equation \eqref{eq:arbitrary-boundary}
        is $qT_M^{(1,m-1)}$, the same calculation gives
        \[
        q=-m\sum_{p+i+j=m}a_{p,i,j},
        \]
        so exactness of that multiple requires $m\mid q$.
        Together with $\cstDiff{}{m}(S)=mT_2^{(1,m-1)}$ from Equation \eqref{eq:k3S}, this proves that $[T_2^{(1,m-1)}]$ has exact order $m$.

        \subsubsection{3-torsion class for $K_{3,3}$}

        Let $G=K_{3,3}$.
        \begin{center}
            \begin{tikzpicture}[
                vertex/.style={circle, fill=black, inner sep=2pt},
                edge label/.style={above, inner sep=2pt},
                line width=0.7pt
                ]
                \node[vertex, label=left:{$v_0$}] (v0) at (0, 2) {};
                \node[vertex, label=left:{$v_1$}] (v1) at (0, 0) {};
                \node[vertex, label=left:{$v_2$}] (v2) at (0,-2) {};

                \node[vertex, label=right:{$v_3$}] (v3) at (5, 2) {};
                \node[vertex, label=right:{$v_4$}] (v4) at (5, 0) {};
                \node[vertex, label=right:{$v_5$}] (v5) at (5,-2) {};

                \draw (v0)--node[edge label,pos=.25] {$e_1$} (v3);
                \draw (v0)--node[edge label,pos=.25] {$e_2$} (v4);
                \draw (v0)--node[edge label,pos=.25] {$e_3$} (v5);

                \draw (v1)--node[edge label,pos=.75] {$e_4$} (v3);
                \draw (v1)--node[edge label,pos=.25] {$e_5$} (v4);
                \draw (v1)--node[edge label,pos=.50] {$e_6$} (v5);

                \draw (v2)--node[edge label,pos=.75] {$e_7$} (v3);
                \draw (v2)--node[edge label,pos=.75] {$e_8$} (v4);
                \draw (v2)--node[edge label,pos=.75] {$e_9$} (v5);
            \end{tikzpicture}
        \end{center}
        with edge order $e_1 < e_2 \cdots < e_9$, and $v_{root}=v_0$. There are five NBC spanning trees in maximum homological grading $4$:
        \[
        \begin{gathered}
            T_1 = \{e_1,e_4,e_5,e_6,e_8\},\qquad
            T_2 = \{e_1,e_4,e_5,e_6,e_9\},\qquad
            T_3 = \{e_1,e_4,e_5,e_7,e_9\},\\
            T_4 = \{e_1,e_4,e_5,e_8,e_9\},\qquad
            T_5 = \{e_1,e_4,e_6,e_7,e_8\}.
        \end{gathered}
        \]
        The lexicographically maximal one is
        \[
        T_5=\{e_1,e_4,e_6,e_7,e_8\},\qquad
        \IA(T_5)=\{e_1\},\qquad
        \IN(T_5)=\{e_4,e_6,e_7,e_8\}.
        \]
        Its critical components are
        $\{v_{root}\}$ and $\{v_1,v_2,v_3,v_4,v_5\}$.

        Take the inactive edge $e_7$ and consider the NBC spanning tree $\psi'_{e_7}(T_5)$. Removing $e_7$ separates
        $\{v_{root},v_1,v_3,v_5\}$ from $\{v_2,v_4\}$, and
        \[
        \cut(T_5,e_7)=\{e_2,e_5,e_7,e_9\},\qquad
        \min\cut(T_5,e_7)=e_2.
        \]
        Thus the preimage is
        \begin{equation}
            T_5'=\psi'_{e_7}(T_5)
            =T_5\cup\{e_2\}\setminus\{e_7\}
            =\{e_1,e_2,e_4,e_6,e_8\}.
            \label{eq:k33-preimage}
        \end{equation}
        It satisfies
        \[
        \IA(T_{5}')=\{e_1,e_2\},\qquad
        \IN(T_{5}')=\{e_4,e_6,e_8\},
        \]
        so $i(T_5)-i(T_{5}')=1$.
        The activity words are
        \[
        W(T_{5}')=LLdDdDdDd,\qquad W(T_5)=LddDdDDDd.
        \]
        The matching order is $e_1<_\M e_2$. Hence the enhancement coordinates
        of $T_{5}'^{(p,i,j)}$ refer, in order, to the components
        \[
        \{v_{root}\},\qquad \{v_1,v_3,v_5\},\qquad \{v_2,v_4\}.
        \]

        The external edges of $T_{5}'$ are $e_3,e_5,e_7,e_9$.
        For $e_7$ the fundamental circuit is $\cyc(T_{5}',e_7)=\{e_1,e_2,e_7,e_8\}$. Its largest active edge is $e_2$, so $\psi_{e_7}(T_{5}')=T_5$. There is also an NBC spanning tree in the image along $e_5$:
        \[
        \cyc(T_{5}',e_5)=\{e_1,e_2,e_4,e_5\},\qquad
        T_1=\psi_{e_5}(T_{5}')=\{e_1,e_4,e_5,e_6,e_8\}.
        \]
        Here
        \[
        \IA(T_1)=\{e_1\},\qquad \IN(T_1)=\{e_4,e_5,e_6,e_8\},
        \qquad W(T)=LddDDDdDd.
        \]
        Thus $T_1$ also lies in maximum homological grading $4$, and $T_1<T_5$ in the lexicographic order. Its critical components are the same vertex sets as those of $T_5$.

        Contracting $e_4,e_6,e_8$ in the preimage gives two possible Type I configurations (See \autoref{amincidences-typeI}): $e_5$ and $e_7$ both join the two non-root components.
        Both of these edges contribute to the differential:
        \begin{center}
            \begin{tikzpicture}[scale=1]
                \node[root,label=above:{$v_{root}:x^p$}] (r) at (0,1.7) {};
                \node[vtx,label=below left:{$x^i$}] (a) at (-1.5,0) {};
                \node[vtx,label=below right:{$x^j$}] (b) at (1.5,0) {};
                \draw[live] (r)--node[left]{$e_1$}(a);
                \draw[live] (r)--node[right]{$e_2$}(b);
                \draw[external] (a) to[bend left=12] node[above]{$e_7$} (b);
                \draw[external] (a) to[bend right=20] node[below]{$e_5$} (b);
                \node at (0,-1) {$T_{5}'^{(p,i,j)}$ after contracting $\IN(T_{5}')$};
            \end{tikzpicture}
        \end{center}
        For the other two external edges, the image contains a broken circuit; therefore, we exclude them. We will have three similar alternating paths as in the case of $K_3$. The weights of the three alternating paths are
        \begin{center}
            \begin{tabular}{|l|l|l|l|l|}
                \hline
                $e$ & $\psi_e(T_5')$ & $w(P_1)$ & $w(P_2)$ & $w(P_3)$\\ \hline
                $e_7$ & $T_5$ & $(-1)^2=1$ & $(-1)^3=-1$ & $(-1)^3=-1$\\ \hline
                $e_5$ & $T_1$ & $(-1)^1=-1$ & $(-1)^2=1$ & $(-1)^2=1$\\ \hline
            \end{tabular}
        \end{center}
        Thus, we define the following chain:
        \[
        S=T_{5}'^{(1,1,1)}-T_{5}'^{(0,1,2)}-T_{5}'^{(0,2,1)}
        \]
        and the image under the differential $\cstDiff{}{3}$ for individual summand is given by
        \begin{align}
            \cstDiff{}{3}\bigl(T_{5}'^{(1,1,1)}\bigr)
            &=T_5^{(1,2)}-2T_5^{(2,1)}
            -T_1^{(1,2)}+2T_1^{(2,1)},\nonumber\\
            \cstDiff{}{3}\bigl(T_{5}'^{(0,1,2)}\bigr)
            &=-T_5^{(1,2)}-T_5^{(2,1)}
            +T_1^{(1,2)}+T_1^{(2,1)},\label{eq:k33-full-differentials}\\
            \cstDiff{}{3}\bigl(T_{5}^{(0,2,1)}\bigr)
            &=-T_5^{(1,2)}-T_5^{(2,1)}
            +T_1^{(1,2)}+T_1^{(2,1)}.\nonumber
        \end{align}
        Hence, the image of $S$ is
        \begin{align}
            \cstDiff{}{3}(S)
            &=\bigl(T_5^{(1,2)}-2T_5^{(2,1)}-T_1^{(1,2)}+2T_1^{(2,1)}\bigr)+2\bigl(T_5^{(1,2)}+T_5^{(2,1)}
            -T_1^{(1,2)}-T_1^{(2,1)}\bigr)\nonumber\\
            &=3\bigl(T_5^{(1,2)}-T_1^{(1,2)}\bigr).
            \label{eq:k33-signed-boundary}
        \end{align}

        Both $T_5^{(1,2)}$ and $T^{(1,2)}$ lie in maximum homological
        grading $4$, so their difference is a cycle.
        To show that this difference is not a boundary, suppose that an arbitrary integral chain in $\cstBiCom{3}{3}{3}{K_{3,3}}$ has boundary $T_5^{(1,2)}-T_1^{(1,2)}$. Then, we must have
        \begin{align}
            \cstDiff{}{3}\left(
            \sum_{\substack{T'\in \F(\mathrm{NBC}(K_{3,3}))\\i(T')=3}}
            \Bigl(a_{T',1}T'^{(0,1,2)}+a_{T',2}T'^{(0,2,1)}
            +b_{T',1}T'^{(1,1,1)}\Bigr)\right) =T_5^{(1,2)}-T_1^{(1,2)}.
            \label{eq:k33-arbitrary-sum}
        \end{align}

        Now define the coefficient sums
        \[
        a_i=\sum_{T'}a_{T',i}\quad(i=1,2),\qquad
        b_1=\sum_{T'}b_{T',1}.
        \]
        Comparing only the coefficients of $T_5^{(1,2)}$ and $T_5^{(2,1)}$
        in Equation \eqref{eq:k33-arbitrary-sum} gives
        \[
        -(a_1+a_2)+b_1=1,\qquad
        -(a_1+a_2)-2b_1=0.
        \]
        Subtracting the two equations gives $3b_1=1$, which is a contradiction. Thus, $T_5^{(1,2)}-T_1^{(1,2)}$
        is not a boundary.

        Similarly, for a boundary equal to $q(T_5^{(1,2)}-T_1^{(1,2)})$, the same comparison gives $3b_1=q$. Therefore, any multiple must have $3\mid q$. Together with Equation \eqref{eq:k33-signed-boundary}, this proves that $T_5^{(1,2)}- T_1^{(1,2)}$ has order 3 in $\cstBiHom{4}{3}{3}{K_{3,3}}$.\\

        The construction of the torsion class for $K_3$ provides an instructive example that motivated our general construction. As we observed in the $K_{3,3}$ example, the proof of the general case will boil down to counting alternating paths from critical states in the complex of a specific preimage tree $i_{\max} - 1$ homological grading to critical states of trees in the $\max$ homological grading. Our primary tool will be \autoref{Am-differential}, which classifies such weighted path counts. In the proof, we will represent the enhanced states in formal notation and supplement them with a diagrammatic representation for the reader's convenience.

        % \Torsionconjecture*

        \begin{proof}[Proof of \autoref{torsionconjecture}]
            By Corollary 3.2 and Proposition 3.6 of \cite{graphcohomologiesarbitraryalgebras}, we may assume that $ G$ is a connected simple graph. First, we prove the theorem for $2$-connected graphs. Since $G$ contains a cycle of length greater than or equal to $3$, there exists an NBC spanning tree of $G$ with non-zero homological grading. In particular, $i_{\max} > 0$. We will construct an explicit chain $X$ representing the $m$-torsion class. The chain $X$ will be obtained as
            $\frac{1}{m} \cdot \cstDiff{\ast}{m}(S)$ ($\cstDiff{}{m}(S)$ will be shown to be a multiple of $m$) for a specific chain $S$ in the homological grading $i_{\max}-1$. We divide the proof into three main steps.

            \begin{enumerate}
                \item We compute all possible alternating paths from the critical states with homological grading $i_{\max}-1$ and quantum grading $m$. Then, we show that $\cstDiff{\ast}{m}(S) = m \cdot X$, for a specific chain $S$ in appropriate grading and where $X= T_{M}^{(1,{m-1})} + \sum_{k \in I} \pm {T_k}^{(1,{m-1})}$ and $I$ is an indexing set that is possibly empty, and for each $k\in I$, $T_k$ is an NBC spanning tree in the maximum homological grading and $T_M$ is a fixed NBC spanning tree in the maximum homological grading.\\

                \item We construct a specific chain $S$ in quantum grading $m$, generated by the critical states in a distinguished preimage tree $T_{M-1}$ of a maximum homological grading tree $T_M$. We use the notation $\left.C \right\rvert_{T_M}$ to denote the projection of a chain $C$ into the submodule generated by critical states of $T_M$. We compute alternating paths from critical states appearing as summands of $S$ to critical states of $T_M$ to show that $\left. \cstDiff{\ast}{m}(S)  \right\rvert_{T_M}= m \cdot T_M^{(1,{m-1})}$.\\

                \item We show $X$ is not a boundary: for an arbitrary chain $Y$ in quantum grading $m$ and homological grading $i_{max}-1$, we show that $\cstDiff{\ast}{m}(Y) \neq X$ by demonstrating that the coefficient of the summand given by the distinguished critical state of $T_M$ in $\cstDiff{\ast}{m}(Y)$ cannot be $1$. To show this, we use previously computed path counts to obtain algebraic relations that give a contradiction.
            \end{enumerate}

            \medskip
            \noindent\textbf{(1) Computing paths out of critical states having  homological grading $i(T_M)-1$ and quantum grading $m$.}

            Let $T$ be any NBC spanning tree in the homological grading $i_{\max}-1$ and $f \in \IN(T)$ be any internally inactive edge such that $i(\psi_f(T))=i_{\max}$. Then, $T$ has to be one of the following types:
            \begin{table}[!tbph]
                \centering
                \begin{tabular}{c c c c}
                    \TmI{v_{root}}{}{} & \TmII{v_{root}}{}{} & \TmIII{}{v_{root}}{} & \TmIV{}{v_{root}}{}\\
                    (I) & (II) & (III) & (IV)
                \end{tabular}
                \label{type of contracted trees}
            \end{table}

            Here the types are determined based on the placement of $v_{root}$ labelled vertex, and depending on the number of internally active edges $\cyc(T,f)$ contains. The internally inactive edges of $T$ have been contracted (the blue dashed edge represents the external edge $f$). From now on we will denote $T$ by $T_{M-1}$, as $T_M$ will be used to denote an NBC spanning tree of maximum homological grading.\\

            We use the following notation to denote the critical enhanced spanning subgraphs of $[\R(T_{M-1}), T_{M-1}]$ and $[\R(T_M), T_M]$.
            \[
            \begin{array}{c@{\qquad}c}
                \begin{aligned}
                    T^{(p,i,j)}_{M-1,\text{I}}
                    &\longleftrightarrow
                    \scalebox{0.9}{\TmI{v_{root}:x^p}{x^i}{x^j}}
                \end{aligned}
                &
                \begin{aligned}
                    T^{(p,i,j)}_{M-1,\text{III}}
                    &\longleftrightarrow
                    \scalebox{0.9}{\TmIII{x^i}{v_{root}:x^p}{x^j}}
                \end{aligned}
                \\[0.8em]
                \multicolumn{2}{c}{
                T^{(q,k)}_M
                \longleftrightarrow
                \scalebox{0.9}{\Tm{v_{root}:x^q}{x^k}{x^k}{v_{root}:x^q}}
                }
                \\[0.8em]
                \begin{aligned}
                    T^{(p,i,j)}_{M-1,\text{II}}
                    &\longleftrightarrow
                    \scalebox{0.9}{\TmII{v_{root}:x^p}{x^i}{x^j}}
                \end{aligned}
                &
                \begin{aligned}
                    T^{(p,i,j)}_{M-1,\text{IV}}
                    &\longleftrightarrow
                    \scalebox{0.9}{\TmIV{x^i}{v_{root}:x^p}{x^j}}
                \end{aligned}
            \end{array}
            \]

            The $v_{root}$ labelled vertex represents the root component of $T_{M-1}$ and $T_M$, and the two isolated vertices on the leftmost diagram in the above figure represent the two components corresponding to two active edges in $T_{M-1}$. The edge in the rightmost diagram represents the inserted edge $f$. We use $T_M^{(q,k)}$ to represent the two different trees, since contracting the internally inactive edge $f$ in $T_M$ obtained from the image of both type-I and type-III trees gives the same tree.\\

            The alternating paths between the critical enhanced spanning subgraph of $T_{M-1}^{(p,i,j)}$, where $p \in \{0,1\}$, and of $T_M^{(i,j)}$ and $T_M^{(j,i)}$ for type I are obtained from the description of $\cstDiff{}{m}$ given by Equation \eqref{am-diff-formula}. There are three alternating paths given by the formula, as described in \Cref{torsion k_3}
            \begin{itemize}
                \item $\Pt_1$ is obtained by taking the differential at edge $f$.
                \item $\Pt_2$ is obtained by first applying a differential, then matching at edge $e_i$, and finally taking a differential at edge $f$.
                \item $\Pt_3$ is obtained by first applying a differential, then matching at edge $e_j=e$, and finally taking a differential at edge $f$.
            \end{itemize}
            \autoref{amincidences-typeI} visually describes the alternating paths $\Pt_1, \Pt_2$, and $\Pt_3$ for $p \in \{0,1\}$.

            \begin{figure}[ht]
                \centering
                \scalebox{0.85}
                {% Requires \usepackage{tikz} and \usetikzlibrary{arrows.meta}.
\begingroup
\tikzset{
 vertex/.style={circle,fill=black,draw=black,minimum size=4.2pt,inner sep=0pt},
 root/.style={vertex,fill=red,draw=red},
 inactive/.style={densely dotted,line width=.65pt},
 selected/.style={line width=.7pt},
 differential/.style={-{Stealth[length=2.2mm,width=1.6mm]},line width=.65pt},
 matching/.style={differential,draw=red},
 lab/.style={inner sep=1.5pt,font=\normalsize}
}
% Arguments: shift, root (0=top,1=left), solid edge (0=none,1=left,
% 2=right,3=bottom), top label, left label, right label,
% merged-component label, initial-state activity labels (0=no,1=yes).
\newcommand{\AltState}[8]{%
\begin{scope}[shift={#1}]
 \coordinate (t) at (0,1.45);
 \coordinate (l) at (-.84,0);
 \coordinate (r) at (.84,0);
 \ifnum#3=1 \draw[selected] (l)--(t);
 \else \draw[inactive] (l)--(t);\fi
 \ifnum#3=2 \draw[selected] (t)--(r);
 \else \ifnum#3=3\else\draw[inactive] (t)--(r);\fi\fi
 \ifnum#3=3 \draw[selected] (l)--(r);
 \else \ifnum#8=1\else\draw[inactive] (l)--(r);\fi\fi
 \ifnum#2=0
  \node[root] at (t) {};
  \node[lab,anchor=east] at (-.13,1.48) {$v_{root}$};
  \node[vertex] at (l) {};
 \else
  \node[vertex] at (t) {};
  \node[root] at (l) {};
  \node[lab,anchor=east] at (-.96,.03) {$v_{root}$};
 \fi
 \node[vertex] at (r) {};
 \node[lab,anchor=south] at (0,1.58) {$#4$};
 \node[lab,anchor=north] at (-.84,-.12) {$#5$};
 \node[lab,anchor=north] at (.84,-.12) {$#6$};
 \ifnum#3=1 \node[lab,anchor=west] at (-.43,.7) {$#7$};\fi
 \ifnum#3=2 \node[lab,anchor=west] at (.46,.98) {$#7$};\fi
 \ifnum#3=3
  \node[lab,anchor=north] at (0,-.12) {$#7$};
  \node[lab] at (-.59,.8) {$L$};
  \node[lab] at (.61,.8) {$d$};
  \node[lab,anchor=south] at (0,.08) {$D$};
 \fi
 \ifnum#8=1
  \node[lab] at (-.64,.82) {$L$};
  \node[lab] at (.64,.82) {$L$};
  \node[lab,anchor=south] at (0,.05) {$d$};
 \fi
\end{scope}}
\begin{tikzpicture}[x=1cm,y=1cm]
% P_1: direct differential along the external edge of activity d.
\node[lab] at (6.75,2.2) {$\Pt_1 \: (p=1)$};
\AltState{(4.5,0)}{0}{0}{x}{x^i}{x^j}{}{1}
\AltState{(9,0)}{0}{3}{x}{}{}{x^{i+j}}{0}
\draw[differential] (5.65,0.72)--node[lab,above,text=black] {$d_*^f$} (7.85,0.72);
% P_2: d_*^{e_i}, M, d_*^e.
\node[lab] at (6.75,-0.9) {$\Pt_2\: (p=0)$};
\AltState{(0,-3.1)}{0}{0}{1}{x^i}{x^j}{}{1}
\AltState{(4.5,-3.1)}{0}{1}{}{}{x^j}{x^i}{0}
\AltState{(9,-3.1)}{0}{0}{x^i}{1}{x^j}{}{0}
\AltState{(13.5,-3.1)}{0}{3}{x^i}{}{}{x^j}{0}
\draw[differential] (1.15,-2.38)--node[lab,above,text=black] {$d_*^{e_i}$} (3.35,-2.38);
\draw[matching] (5.65,-2.38)--node[lab,above,text=black] {$\M$} (7.85,-2.38);
\draw[differential] (10.15,-2.38)--node[lab,above,text=black] {$d_*^f$} (12.35,-2.38);
% P_3: d_*^{e_j}, M, d_*^e.
\node[lab] at (6.75,-4) {$\Pt_3 \: (p=0)$};
\AltState{(0,-6.2)}{0}{0}{1}{x^i}{x^j}{}{1}
\AltState{(4.5,-6.2)}{0}{2}{}{x^i}{}{x^j}{0}
\AltState{(9,-6.2)}{0}{0}{x^j}{x^i}{1}{}{0}
\AltState{(13.5,-6.2)}{0}{3}{x^j}{}{}{x^i}{0}
\draw[differential] (1.15,-5.48)--node[lab,above,text=black] {$d_*^{e_j}$} (3.35,-5.48);
\draw[matching] (5.65,-5.48)--node[lab,above,text=black] {$\M$} (7.85,-5.48);
\draw[differential] (10.15,-5.48)--node[lab,above,text=black] {$d_*^f$} (12.35,-5.48);
% P_4: d_*^{e_i}, M, d_*^e.
\node[lab] at (6.75,-7.1) {$\Pt_2 \: (p=1)$};
\AltState{(0,-9.3)}{0}{0}{x}{x^i}{x^j}{}{1}
\AltState{(4.5,-9.3)}{0}{1}{}{}{x^j}{x^{i+1}}{0}
\AltState{(9,-9.3)}{0}{0}{x^{i+1}}{1}{x^j}{}{0}
\AltState{(13.5,-9.3)}{0}{3}{x^{i+1}}{}{}{x^j}{0}
\draw[differential] (1.15,-8.58)--node[lab,above,text=black] {$d_*^{e_i}$} (3.35,-8.58);
\draw[matching] (5.65,-8.58)--node[lab,above,text=black] {$\M$} (7.85,-8.58);
\draw[differential] (10.15,-8.58)--node[lab,above,text=black] {$d_*^f$} (12.35,-8.58);
% P_5: d_*^{e_j}, M, d_*^e.
\node[lab] at (6.75,-10.2) {$\Pt_3 \: (p=1)$};
\AltState{(0,-12.4)}{0}{0}{x}{x^i}{x^j}{}{1}
\AltState{(4.5,-12.4)}{0}{2}{}{x^i}{}{x^{j+1}}{0}
\AltState{(9,-12.4)}{0}{0}{x^{j+1}}{x^i}{1}{}{0}
\AltState{(13.5,-12.4)}{0}{3}{x^{j+1}}{}{}{x^i}{0}
\draw[differential] (1.15,-11.68)--node[lab,above,text=black] {$d_*^{e_j}$} (3.35,-11.68);
\draw[matching] (5.65,-11.68)--node[lab,above,text=black] {$\M$} (7.85,-11.68);
\draw[differential] (10.15,-11.68)--node[lab,above,text=black] {$d_*^f$} (12.35,-11.68);
\end{tikzpicture}
\endgroup}
                \caption{Alternating paths between enhancements of $T_{M-1}$ to $T_M$ (Type I)}
                \label{amincidences-typeI}
            \end{figure}

            The weights of these paths are given as follows:

            \begin{enumerate}[(a)]
                \item $w(\Pt_1) = w(d^{f}_*) = (-1)^{\xi(T_{M-1},f)}$

                \item $w(\Pt_2) = w(d^{e_i}_*) \cdot w(\M) \cdot w(d^{f}_*)  = (-1) \cdot \frac{w(d^{e_i}_*) \cdot w(d^{f}_*)}{w(d^{e_i}_*)} = (-1) \cdot w(d^{f}_*) = (-1)^{\xi(T_{M-1},f) + 1}$

                \item $w(\Pt_3) = w(d^{e_j}_*) \cdot w(\M) \cdot w(d^{f}_*) = (-1) \cdot \frac{w(d^{e_j}_*) \cdot w(d^{f}_*)}{w(d^{e_j}_*)} = (-1) \cdot w(d^{f}_*) = (-1)^{\xi(T_{M-1},f) + 1}$
            \end{enumerate}

            Without loss of generality, assume that $\xi(T_{M-1},f)$ is even; otherwise, the image of $\cstDiff{*}{m}(S)$ differs by a sign in the computation shown above.

            Out of these four possibilities, we can show that generators corresponding to type II and type IV NBC spanning trees will vanish under $\cstDiff{}{m}$, since there is only one internally active edge in $\cyc(T_{M-1}, f)$ in both type II and type IV NBC spanning trees. Let us denote the generators corresponding to type~II spanning trees by $T^{(p,i,j)}_{M-1, \text{II}}$ with $p+i+j=m$. Then we have the following:

            \begin{equation}\label{type ii vanish}
                \cstDiff{}{m}\left(T^{(p,i,j)}_{M-1, \text{II}} \right) =  (-1)^{\xi(T_{M-1, \text{II}},f)} \cdot T_M^{(p+j,i)} + (-1)^{\xi(T_{M-1, \text{II}},f)+1} \cdot T_M^{(p+j,i)} = 0
            \end{equation}

            Since $i \geq 1$, $p+j < m$, and hence we can directly apply a differential at the edge $f$ to obtain the first summand in Equation \eqref{type ii vanish} and the other alternating path can be achieved by taking a matching along the path $P(e_j,e_j)$ and then taking a differential along $f$ to obtain the second summand whose weight differs by a sign from the first summand, as we have taken a single matching edge for $e_j$. Thus, we can exclude type II NBC spanning trees.\\

            For the type IV case, let us denote the generators by $T_{M-1, \text{IV}}^{(q, i, j)}$ such that $q+i+j=m$; then we have

            \begin{equation}\label{type iv vanish}
                \begin{aligned}
                    \cstDiff{}{m}\left(T^{(q,i,j)}_{M-1, \text{IV}} \right) = & (-1)^{\xi(T_{M-1, \text{IV}},f)} \cdot T_{M}^{(q,i+j)} + (-1)^{\xi(T_{M-1, \text{IV}},f)+1} \cdot T_M^{(q,i+j)} +\\
                    & (-1)^{\xi(T_{M-1, \text{IV}},f)+1}  \cdot T_M^{(q+i,j)} + (-1)^{\xi(T_{M-1, \text{IV}},f)+2} \cdot T_M^{(q+i,j)} = 0\\
                \end{aligned}
            \end{equation}

            The first summand in Equation \eqref{type iv vanish} is obtained by taking a direct differential along $f$, the second one is obtained by using the path $P(e_j,e_j)$ and then taking a differential along $f$, the third one is obtained by taking a matching just along $e_i$ and a differential along $f$, and the last summand is obtained by taking two matchings along the path $P(e_i, v_2)$ and then a differential along $f$, where $e = \{v_1,v_2\}$.\\

            The alternating paths between the critical enhanced spanning subgraph of $T_{M-1}^{(p,i,j)}$, where $p \in \{0,1\}$, and of $T_M^{(i,j)}$ and $T_M^{(j,i)}$ for type III are obtained from the description of $\cstDiff{}{m}$ given by Equation \eqref{am-diff-formula}. The formula gives two alternating paths:
            \begin{itemize}
                \item $\Pt_1'$ is obtained by taking the differential at edge $f$.
                \item $\Pt_2'$ is obtained by first applying a differential, then matching at edge $e_i$, then at $e_j$, and finally taking a differential at edge $f$.
                \item $\Pt_3'$ is obtained by first applying a differential, then matching at edge $e_j$, and finally taking a differential at edge $f$.
            \end{itemize}
            \autoref{amincidences-typeIII} visually describes the alternating paths $\Pt_1'$,$\Pt_2'$, and $\Pt_3'$ for $p\in\{0,1\}$.

            \begin{figure}[ht]
                \centering
                \scalebox{0.85}
                {% Requires \usepackage{tikz} and \usetikzlibrary{arrows.meta}.
\begingroup
\tikzset{
 vertex/.style={circle,fill=black,draw=black,minimum size=4.2pt,inner sep=0pt},
 root/.style={vertex,fill=red,draw=red},
 inactive/.style={densely dotted,line width=.65pt},
 selected/.style={line width=.7pt},
 differential/.style={-{Stealth[length=2.2mm,width=1.6mm]},line width=.65pt},
 matching/.style={differential,draw=red},
 lab/.style={inner sep=1.5pt,font=\normalsize}
}
% Arguments: shift, root (0=top,1=left), solid edge (0=none,1=left,
% 2=right,3=bottom), top label, left label, right label,
% merged-component label, initial-state activity labels (0=no,1=yes).
\newcommand{\AltState}[8]{%
\begin{scope}[shift={#1}]
 \coordinate (t) at (0,1.45);
 \coordinate (l) at (-.84,0);
 \coordinate (r) at (.84,0);
 \ifnum#3=1 \draw[selected] (l)--(t);
 \else \draw[inactive] (l)--(t);\fi
 \ifnum#3=2 \draw[selected] (t)--(r);
 \else \ifnum#3=3\else\draw[inactive] (t)--(r);\fi\fi
 \ifnum#3=3 \draw[selected] (l)--(r);
 \else \ifnum#8=1\else\draw[inactive] (l)--(r);\fi\fi
 \ifnum#2=0
  \node[root] at (t) {};
  \node[lab,anchor=east] at (-.13,1.48) {$v_{root}$};
  \node[vertex] at (l) {};
 \else
  \node[vertex] at (t) {};
  \node[root] at (l) {};
  \node[lab,anchor=east] at (-.96,.03) {$v_{root}$};
 \fi
 \node[vertex] at (r) {};
 \node[lab,anchor=south] at (0,1.58) {$#4$};
 \node[lab,anchor=north] at (-.84,-.12) {$#5$};
 \node[lab,anchor=north] at (.84,-.12) {$#6$};
 \ifnum#3=1 \node[lab,anchor=west] at (-.43,.7) {$#7$};\fi
 \ifnum#3=2 \node[lab,anchor=west] at (.46,.98) {$#7$};\fi
 \ifnum#3=3
  \node[lab,anchor=north] at (0,-.12) {$#7$};
  \node[lab] at (-.59,.8) {$L$};
  \node[lab] at (.61,.8) {$d$};
  \node[lab,anchor=south] at (0,.08) {$D$};
 \fi
 \ifnum#8=1
  \node[lab] at (-.64,.82) {$L$};
  \node[lab] at (.64,.82) {$L$};
  \node[lab,anchor=south] at (0,.05) {$d$};
 \fi
\end{scope}}
\begin{tikzpicture}[x=1cm,y=1cm]
% P_1: direct differential along the external edge of activity d.
\node[lab] at (2.25,2.2) {$\Pt_1' \: (p=1)$};
\AltState{(0,0)}{1}{0}{x^i}{x}{x^j}{}{1}
\AltState{(4.5,0)}{1}{3}{x^i}{}{}{x^{j+1}}{0}
\draw[differential] (1.15,0.72)--node[lab,above,text=black] {$d_*^f$} (3.35,0.72);
% P_3: direct differential along the external edge of activity d.
\node[lab] at (11.25,2.2) {$\Pt_1' \:  (p=0)$};
\AltState{(9,0)}{1}{0}{x^i}{1}{x^j}{}{1}
\AltState{(13.5,0)}{1}{3}{x^i}{}{}{x^j}{0}
\draw[differential] (10.15,0.72)--node[lab,above,text=black] {$d_*^f$} (12.35,0.72);
% P_2: the existing longer path, read across, down, then back.
\node[lab] at (6.75,-0.9) {$\Pt_2' \: (p=0)$};
\AltState{(2.25,-3.1)}{1}{0}{x^i}{1}{x^j}{}{1}
\AltState{(6.75,-3.1)}{1}{1}{}{}{x^j}{x^i}{0}
\AltState{(11.25,-3.1)}{1}{0}{1}{x^i}{x^j}{}{0}
\draw[differential] (3.4,-2.38)--node[lab,above,text=black] {$d_*^{e_i}$} (5.6,-2.38);
\draw[matching] (7.9,-2.38)--node[lab,above,text=black] {$\M$} (10.1,-2.38);
\AltState{(11.25,-6.2)}{1}{2}{}{x^i}{}{x^j}{0}
\AltState{(6.75,-6.2)}{1}{0}{x^j}{x^i}{1}{}{0}
\AltState{(2.25,-6.2)}{1}{3}{x^j}{}{}{x^i}{0}
\draw[differential] (11.25,-3.65)--node[lab,right,text=black] {$d_*^{e_j}$} (11.25,-4.27);
\draw[matching] (10.1,-5.48)--node[lab,above,text=black] {$\M$} (7.9,-5.48);
\draw[differential] (5.6,-5.48)--node[lab,above,text=black] {$d_*^f$} (3.4,-5.48);
% P_4: d_*^{e_i}, M, d_*^{e_j}, M, d_*^e, starting from P_1.
% Read across the upper row, down, then back across the lower row.
\node[lab] at (6.75,-7.1) {$\Pt_2' \: (p=1)$};
\AltState{(2.25,-9.3)}{1}{0}{x^i}{x}{x^j}{}{1}
\AltState{(6.75,-9.3)}{1}{1}{}{}{x^j}{x^{i+1}}{0}
\AltState{(11.25,-9.3)}{1}{0}{1}{x^{i+1}}{x^j}{}{0}
\draw[differential] (3.4,-8.58)--node[lab,above,text=black] {$d_*^{e_i}$} (5.6,-8.58);
\draw[matching] (7.9,-8.58)--node[lab,above,text=black] {$\M$} (10.1,-8.58);
\AltState{(11.25,-12.4)}{1}{2}{}{x^{i+1}}{}{x^j}{0}
\AltState{(6.75,-12.4)}{1}{0}{x^j}{x^{i+1}}{1}{}{0}
\AltState{(2.25,-12.4)}{1}{3}{x^j}{}{}{x^{i+1}}{0}
\draw[differential] (11.25,-9.85)--node[lab,right,text=black] {$d_*^{e_j}$} (11.25,-10.47);
\draw[matching] (10.1,-11.68)--node[lab,above,text=black] {$\M$} (7.9,-11.68);
\draw[differential] (5.6,-11.68)--node[lab,above,text=black] {$d_*^f$} (3.4,-11.68);
% P_5: modify P_1 using the differential and matching along e_j.
\node[lab] at (6.75,-13.3) {$\Pt_3' \: (p=1)$};
\AltState{(0,-15.5)}{1}{0}{x^i}{x}{x^j}{}{1}
\AltState{(4.5,-15.5)}{1}{2}{}{x}{}{x^{i+j}}{0}
\AltState{(9,-15.5)}{1}{0}{x^{i+j}}{x}{1}{}{0}
\AltState{(13.5,-15.5)}{1}{3}{x^{i+j}}{}{}{x}{0}
\draw[differential] (1.15,-14.78)--node[lab,above,text=black] {$d_*^{e_j}$} (3.35,-14.78);
\draw[matching] (5.65,-14.78)--node[lab,above,text=black] {$\M$} (7.85,-14.78);
\draw[differential] (10.15,-14.78)--node[lab,above,text=black] {$d_*^f$} (12.35,-14.78);
\end{tikzpicture}
\endgroup}
                \caption{Alternating paths between enhancements of $T_{M-1}$ to $T_M$ (Type III)}
                \label{amincidences-typeIII}
            \end{figure}

            The weights of these paths are given as follows:

            \begin{enumerate}[(a)]
                \item $w(\Pt_1') = w(d^{f}_*) = (-1)^{\xi(T_{M-1},f)}$

                \item $w(\Pt_2') = w(d^{e_i}_*) \cdot w(\M) \cdot w(d^{e_j}_*) \cdot w(\M) \cdot w(d_*^f)  = (-1)^2 \cdot \frac{w(d^{e_i}_*) \cdot w(d^{e_j}_*) \cdot w(d^{f}_*)}{w(d^{e_i}_*) \cdot w(d^{e_j}_*)} =  (-1)^{\xi(T_{M-1},f)}$

                \item $w(\Pt_3') = w(d^{e_j}_*) \cdot w(\M) \cdot w(d_*^f)  = (-1) \cdot \frac{w(d^{e_j}_*) \cdot w(d^{f}_*)}{w(d^{e_j}_*)} =  (-1)^{\xi(T_{M-1},f)+1}$
            \end{enumerate}

            \medskip
            \noindent\textbf{(2) Constructing the chain $S$ with $\cstDiff{\ast}{m}(S) = m\cdot X$.}

            Let $T_M$ be an NBC spanning tree in the maximum homological grading. Since $G$ is $2$-connected and has a cycle of length greater than equal to 3, so, $T_M$ has exactly one live edge. Let $f$ be the minimum ordered internal inactive edge in $T_M$ and $T_{M-1}=\psi'_{f}(T_M)$ be the distinguished NBC spanning tree that is in the preimage (\Autoref{preimagetree}) of $T_M$ with $T_{M}= (T_{M-1} \setminus \{e\}) \cup \{f\}$. From \Autoref{min dead edge}, it follows that $i(T_{M})- i(T_{M-1})=1$, $\IN(T_{M}) \setminus \IN(T_{M-1})= \{ f \} $ and $\IA(T_{M-1}) \setminus \IA(T_{M}) = \{ e \} $.\\

            Let $e_M$ be the unique internally active edge of $T_M$. We show that $T_{M-1}$ is either of type I or III. For this, we only need to show that $e_M \in \cyc(T_{M-1},f)$. Suppose, $e_M \not\in \cyc(T_{M-1},f)$, this implies that every edge $g \in \cyc(T_{M-1},f)$ with $g\neq e$ is internally inactive. As $f$ was chosen to be the minimal internally inactive edge of $T_M$, we have $e < f \leq g$. But then $e$ becomes the least ordered edge in $\cyc(T_M,e)$ and thus becomes externally active with respect to $T_M$, which contradicts the fact that $T_M$ is an NBC spanning tree.\\

            Now, consider the following chain element in the homological grading $i(T_{M})-1$:
            \[
            S= \left(\sum_{\substack{ i+j=m-1 \\ i,j \geq 1}} T_{M-1}^{(1,i,j)} - \sum_{\substack{ i+j=m \\ i,j \geq 1}} T_{M-1}^{(0,i,j)}\right)
            \]

            In the previous step we saw that the only trees with non-trivial incidence with respect to Equation \eqref{am-diff-formula} were of type I and III. Let $Q(T_{M-1})$ denote the set of trees $T_l \in \NBC(G)$ that are related to $T_{M-1}$ of type I or type III. Then, by the description in the \autoref{amincidences-typeI} and \autoref{amincidences-typeIII} for all possible alternating paths to $T_M^{(1,m-1)}$, we obtain

            \begin{equation}\label{eq:genbd}
                \begin{split}
                    \cstDiff{}{m}(S)
                    &= \sum_{T_l \in Q(T_{M-1})} \Biggl( (m-2)T_l^{(1,m-1)} - \sum_{\substack{i+j=m-1 \\ i,j \ge 1}} \bigl(T_l^{(i+1,j)}+T_l^{(j+1,i)}\bigr) + \sum_{\substack{i+j=m \\ i,j \ge 1}} \bigl(T_l^{(i,j)}+T_l^{(j,i)}\bigr) \Biggr)\\[1em]
                    &= \sum_{T_l \in Q(T_{M-1})} \Biggl( (m-2)T_l^{(1,m-1)} - 2\sum_{k=2}^{m-1}T_l^{(k,m-k)} + 2\sum_{k=1}^{m-1}T_l^{(k,m-k)} \Biggr) \\[1em]
                    &= \sum_{T_l \in Q(T_{M-1})} m \, T_l^{(1,m-1)}.
                \end{split}
            \end{equation}

            We conclude that $\cstDiff{\ast}{m}(S) = m \cdot X$ where $X= T_{M}^{(1,{m-1})} + \sum_{k \in I} \pm {T_k}^{(1,{m-1})}$ and $I$ is an indexing set that is possibly empty, and for each $k\in I$, $T_k$ is an NBC spanning tree in maximum homological grading. Here, the signs of $T_k$ in the expression depend on the parity of $\xi(T_{M-1},e')$ where $e'$ is unique edge in $\IN(T_k) \setminus \IN(T_{M-1})$. \\

            \medskip
            \noindent\textbf{(3) Showing that $X$ is not a boundary}

            Thus, if $X$ is a boundary, then there is a chain $Y$ satisfying $\cstDiff{\ast}{m}Y=X$. After taking projection into the submodule generated by critical states in $T_M$, we obtain $\left. \cstDiff{\ast}{m}Y \right\rvert_{T_M} = T_{M}^{(1,{m-1})}$. Notice that for computing $\left. \cstDiff{\ast}{m}Y \right\rvert_{T_M}$, it suffices to consider summands in trees that are in the preimage of $T_M$.\\

            Let $T_{k,\mathrm{I}}$ and $T_{h,\mathrm{III}}$ range over all
            type-I and type-III preimages of $T_M$ in homological grading
            $i(T_M)-1$. The type-II and type-IV contributions supported on
            $T_M$ vanish under $\cstDiff{}{m}$ as described in Equation \eqref{type ii vanish} and \eqref{type iv vanish}. For each type-I preimage $T_{k,\mathrm{I}}$, let $f_k$ be the externally inactive edge such that $T_M = \psi_{f_k}(T_{k,\mathrm{I}})$. Similarly, let $f_h$ be the corresponding externally inactive edge for each Type~III preimage $T_{h,\mathrm{III}}$ such that $\psi_{f_h}(T_{h,\mathrm{III}})$. Denote their incidence signs by
            \[
            \varepsilon_k=(-1)^{\xi(T_{k,\mathrm{I}},f_k)},
            \qquad
            \eta_h=(-1)^{\xi(T_{h,\mathrm{III}},f_h)}.
            \]
            Thus, by our assumption of exactness, we would have
            \begin{equation}\label{exactness}
                \cstDiff{}{m}
                \biggl(\sum_k\sum_{p+i+j=m}
                a_{k,p,i,j}\,T_{k,\mathrm{I}}^{(p,i,j)}
                +\sum_h\sum_{p+i+j=m}
                b_{h,p,i,j}\,T_{h,\mathrm{III}}^{(p,i,j)}
                \biggr)
                =T_M^{(1,m-1)}
            \end{equation}

            For a type I preimage, the computation shown in \autoref{amincidences-typeI} under the map $\cstDiff{}{m}$ gives
            \[
            \begin{cases}
                \varepsilon_kT_M^{(p,i+j)}
                -\varepsilon_kT_M^{(p+i,j)}
                -\varepsilon_kT_M^{(p+j,i)}, & p\geq 1,\\[2mm]
                -\varepsilon_kT_M^{(i,j)}-\varepsilon_kT_M^{(j,i)}, & p=0.
            \end{cases}
            \]
            A special case of this was observed in the example provided in \Cref{torsion k_3}.

            For a Type~III pre-image, the computation shown in \autoref{amincidences-typeIII} under the map $\cstDiff{}{m}$ gives the corresponding terms
            \[
            \begin{cases}
                \eta_hT_M^{(p+j,i)}
                +\eta_hT_M^{(p+i,j)}
                -\eta_hT_M^{(p,i+j)}, & p\geq 1,\\[2mm]
                \eta_hT_M^{(j,i)}+\eta_hT_M^{(i,j)}, & p=0.
            \end{cases}
            \]
            Observe that, for each enhancement, the image of a type~III tree is the negative of the image of a type-I tree. We now define
            \[
            a_{p,i,j}=\sum_k \varepsilon_k a_{k,p,i,j},
            \qquad
            b_{p,i,j}=\sum_h \eta_h b_{h,p,i,j}.
            \]
            Given the exactness assumption in Equation \eqref{exactness}, we have
            \[
            \sum_{\substack{p+i+j=m\\p\geq1}}
            (a_{p,i,j}-b_{p,i,j})T_M^{(p,i+j)}-
            \sum_{p+i+j=m}
            (a_{p,i,j}-b_{p,i,j})
            \bigl(
            T_M^{(p+i,j)}+T_M^{(p+j,i)}
            \bigr)
            =T_M^{(1,m-1)}.
            \]

            For each $1\leq r\leq m-1$, comparison of the coefficient of
            $T_M^{(r,m-r)}$ gives
            \[
            \begin{aligned}
                &\sum_{\substack{i+j=m-r\\i,j\geq1}}
                (a_{r,i,j}-b_{r,i,j})
                \\
                &\quad-
                \sum_{p=0}^{r-1}
                \Bigl(
                a_{p,r-p,m-r}
                +a_{p,m-r,r-p}
                -b_{p,r-p,m-r}
                -b_{p,m-r,r-p}
                \Bigr)
                =
                \begin{cases}
                    1, & r=1,\\
                    0, & 2\leq r\leq m-1.
                \end{cases}
            \end{aligned}
            \]

            Multiply the equation indexed as $r$ by $r$ and sum over
            $r=1,\ldots,m-1$. For each Type~I preimage, its contribution to this weighted sum is
            \[
            p-(p+i)-(p+j)=-(p+i+j)=-m.
            \]
            For each Type~III preimage, it is
            \[
            (p+j)+(p+i)-p=p+i+j=m.
            \]

            Therefore,
            \[
            \begin{aligned}
                1
                ={}&
                \sum_k \varepsilon_k
                \sum_{p+i+j=m}
                \bigl(p-(p+i)-(p+j)\bigr)\ a_{k,p,i,j}
                \\
                &+
                \sum_h \eta_h
                \sum_{p+i+j=m}
                \bigl((p+j)+(p+i)-p\bigr)\ b_{h,p,i,j}\\
                ={}&m\Biggl(\sum_h\eta_h \sum_{p+i+j=m}b_{h,p,i,j} \ - \ \sum_k \varepsilon_k \sum_{p+i+j=m}a_{k,p,i,j}\Biggr).
            \end{aligned}
            \]
            This is impossible for $m\geq2$. Hence we have a contradiction.

            More generally, if the boundary has coefficient $q$ at
            $T_M^{(1,m-1)}$ and coefficient zero at every other enhancement
            of $T_M$, the same calculation gives
            \[
            q=m\sum_{p+i+j=m}(b_{p,i,j}-a_{p,i,j}),
            \]
            and therefore $m\mid q$.
            This also applies when the boundary contains additional terms
            supported on spanning trees distinct from $T_M$.

            Now, for a general connected graph with a cycle of length at least $3$, we use block decomposition and apply \Autoref{vertex gluing} along with K\"{u}nneth's formula to prove the existence of a torsion of order dividing $m$. \\

            Conversely, a simple connected graph with no cycle of order greater than or equal to $3$ is a tree. Therefore, it has no torsion by \Autoref{tree hom Am}.
        \end{proof}

    \end{section}

    \bibliographystyle{alpha}
    \bibliography{citation}

\end{document}